\documentclass[10pt,letterpaper,reqno]{amsart}
\usepackage{amssymb,a4wide,amsthm,amsmath,amsfonts,xcolor}
\usepackage[retainorgcmds]{IEEEtrantools}
\usepackage{lipsum}
\usepackage{color}
\usepackage{mathrsfs}
\usepackage{hyperref}
\usepackage{relsize}
\usepackage{amsmath}
\allowdisplaybreaks
\newcommand\ada[1]{{\color{black} #1}}
\newcommand\gr[1]{{\color{black} #1}}

\newtheorem{theorem}{Theorem}[section]
\newtheorem{lemma}[theorem]{Lemma}

\newtheorem{proposition}[theorem]{Proposition}
\newtheorem{assumption}[theorem]{Assumption}
\theoremstyle{definition}
\newtheorem{definition}[theorem]{Definition}
\theoremstyle{claim}
\newtheorem{claim}{Claim}
\newtheorem{pf}{Proof of Claim}

\theoremstyle{remark}
\newtheorem{remark}[theorem]{Remark}
\newtheorem{example}[theorem]{Example}

\newcommand{\Tr}{\mathop{\mathrm{Tr}}}
\newcommand{\Ll}{\mathbb{L}^2}
\newcommand{\z}{\hat{z}}
\newcommand{\tz}{\tilde{z}}
\newcommand{\n}{\tilde{N}(dt,dz)}
\newcommand{\N}{\tilde{N}(ds,dz)}
\newcommand{\ns}{\tilde{N}(ds,dz)}
\newcommand{\h}{\mathbb{H}_0^1}
\newcommand{\Hh}{H_0^1}
\newcommand{\la}{\lambda(dz)}
\newcommand{\ou}{\overline{u}}
\newcommand{\on}{\overline{N}}
\newcommand{\ow}{\overline{W}}
\newcommand{\E}{\mathbb{E}}
\newcommand{\oE}{\overline{\mathbb{E}}}
\newcommand{\tE}{\tilde{\mathbb{E}}}
\newcommand{\oz}{\overline{z}}

\def\L{t\wedge\tau_N}
\renewcommand{\d}{\/ d \/}
\def\s{\sigma}
\def\mJ{\mathcal{J}}

\numberwithin{equation}{section}

\begin{document}

\title[Stochastic Control of Tide Equations]{Stochastic Control of Tidal Dynamics Equation with L\'evy Noise}

\author{Pooja Agarwal}

\address{Division of Applied Mathematics\\
Brown University\\
Providence, Rhode Island 02912, USA}
\email{Pooja\_Agarwal@Brown.edu}

\author{Utpal Manna}
\address{Indian Institute of Science Education and Research Thiruvananthapuram\\
Thiruvananthapuram 695016\\
Kerala, India}
\email{manna.utpal@iisertvm.ac.in}

\author{Debopriya Mukherjee}
\address{Indian Institute of Science Education and Research Thiruvananthapuram\\
Thiruvananthapuram 695016\\
Kerala, India}
\email{debopriya13@iisertvm.ac.in}



\begin{abstract}
In this work we first present the existence, uniqueness and regularity of the strong solution of the tidal dynamics model perturbed by L\'evy noise. Monotonicity arguments have been exploited in the proofs. We then formulate  a martingale problem of Stroock and Varadhan associated to an initial value control problem and establish existence of optimal controls. 
\end{abstract}

\maketitle
\tableofcontents

\section{Introduction}
Ocean-tide information has considerably many applications. The data obtained is used to solve vital problems in oceanography and geophysics, and to study earth tides, elastic properties of the Earth's crust and tidal gravity variations. It is also used in space studies to calculate the trajectories of man-made satellites of the Earth and to interpret the results of satellite measurements. The interaction of tides with deep sea ridges and chains of seamounts give rise to deep eddies which transport nutrients from the deep to the surface. The alternate flooding and exposure of the shoreline due to tides is an important factor in the determination of the ecology of the region.

One of the first mathematical explanation for tides was given by Newton by describing tide generating forces. The first dynamic theory of tides was proposed by Laplace. Here we consider the tidal dynamics model proposed by Marchuk and Kagan \cite{marchuk}. The existence and uniqueness of weak solutions of the deterministic tide equation and that of strong solutions of the stochastic tide equation with additive trace class Gaussian noise have been proved in Manna, Menaldi and Sritharan\cite{manna2008stochastic}. In this work, we consider the stochastic tide equation with L\'evy noise and prove the existence and uniqueness and regularity of solution in bounded domains. Control of fluid flow has numerous applications in control of pollutant transport, oil recovery/transport problems, weather predictions, control of underwater vehicles etc. Unification of many control problems in the engineering sciences have been done by studying the optimal control problem of Navier-Stokes equations (see \cite{sritharan2000deterministic},  \cite{sritharan}). Here we consider the initial data optimal control of the  stochastic tidal dynamics model. We consider the Stroock-Varadhan martingale formulation \cite{var} of the stochastic model to prove the existence of optimal initial value control.  

Organization of the paper is as follows.  A brief description of the model has been given in Section \ref{model}. Section \ref{setting} describes the functional setting of the problem and states the \ada{monotonicity property of the linear and non-linear operators}. In Section \ref{estimate} we consider the a-priori estimates and prove the existence, uniqueness and regularity of strong solution. In Section \ref{stochastic} we consider the stochastic optimal control problem with initial value control.
 
Let $\mathcal{O}$ be a bounded domain in $\mathbb{R}^2$ with smooth boundary. Let $(\Omega,\mathcal{F},\mathcal{F}_t,P)$ be a given filtered probability space. In the framework of Gelfand triple $\h(\mathcal{O})\subset\Ll (\mathcal{O})\subset\mathbb{H}^{-1}(\mathcal{O})$, we consider the following tidal dynamics model with L\'evy noise
\begin{align}
&du(t) + [Au(t)+B(u(t))+g\nabla\z(t)]dt
=f(t)dt+\s(t,u(t))dW(t)\notag\\& \quad \quad+\int_Z H(u(t-),z) \n, \quad\text{in }\quad [0,T]\times\mathcal{O}\times \Omega, \yesnumber\label{me1}\\
&d\z(t)+Div(hu(t)) dt=0\quad\quad\text{in } \quad[0,T]\times\mathcal{O}\times \Omega,\yesnumber\label{me2}\\
&u(0,x,\omega)=u_0(x,\omega),\qquad \z(0,x,\omega)=\z_0(x, \omega), \ (x,\omega)\in\mathcal{O}\times\Omega.\yesnumber\label{me3}
\end{align}
The operators $A$ and $B$ are defined in Sections \ref{model} and \ref{setting}. $(W(t))_{t\geq 0}$ is an $\Ll$-valued Wiener process with trace class covariance. $\n=N(dt,dz)-\la dt$ is a compensated Poisson random measure, where $N(dt,dz)$ denotes the Poisson counting measure associated to the point process $p(t)$ on $Z\in\mathscr{B}(\Ll\backslash\{0\})$, where the solutions of the above system have its paths, and $\la$is a $\sigma$-finite measure on $(Z,\mathscr{B}(Z))$.
 
The following theorem states the main result of Section \ref{estimate}. The functional spaces appearing in the statement of the theorem have been defined in Section \ref{setting}.
\begin{theorem}
\label{thmint1}
Let us consider the above stochastic tide model with $f,u_0$ and $\z_0$ such that
\begin{equation*}
f\in L^2(\Omega; L^2(0,T;\Ll(\mathcal{O}))),\quad u_0\in L^2(\Omega; \Ll(\mathcal{O})),\quad \z_0\in L^2(\Omega; L^2(\mathcal{O})).
\end{equation*}
Let the noise coefficients $\s$ and $H$ satisfy Assumption \ref{Hyp}. 
Then there exist path-wise unique adapted processes $u(t,x,\omega)$ and $\z(t,x,\omega)$  with the regularity 
\begin{equation}
\left\{\begin{array}{ll}
&u\in L^2(\Omega; L^{\infty}(0,T;\mathbb{L}^2(\mathcal{O}))\cap L^2(0,T;\h(\mathcal{O})))\\
&\z\in L^2(\Omega ;C([0,T];L^2(\mathcal{O})))
\end{array}\right. 
\end{equation}
satisfying the the stochastic tide model  \eqref{me1}-\eqref{me3} in the weak sense.
\end{theorem}
In Section \ref{stochastic}, we consider the following stochastic optimal control problem with initial value control, 
\begin{align}
&du(t) + [Au(t)+B(u(t))+g\nabla\z(t)]dt
=f(t)dt+\s(t,u(t))dW(t)\nonumber \\& \quad\quad\quad+\int_Z H(u(t-),z) \n
\qquad \text{in }[0,T]\times\mathcal{O}\times \Omega,\label{scint1}\\
&d\z(t)+Div(hu(t)) dt=0\quad\quad\text{in }[0,T]\times\mathcal{O}\times \Omega,\label{scint2}\\
&u(0)=u_0+U,\qquad \z(0)=\z_0,\label{scint3}
\end{align}
where $u_0\in L^2(\Omega; \Ll(\mathcal{O}))$, $\z_0\in L^2(\Omega; L^2(\mathcal{O}))$ and $U\in L^2(\Omega; \Ll(\mathcal{O}))$. The regularities on the initial values and the  assumptions on $\s$ and $H$ are the same as considered in Theorem \ref{thmint1}. The cost functional is given by
\begin{equation}
\mathcal{J}(u, \z, U)=\E\left[\int_0^T \int_\mathcal{O} L(t,u, \z, U) dx\,dt\right],
\end{equation}
where the function $L$ is defined in Section \ref{stochastic}.

The main result of Section \ref{stochastic} is the following:
\begin{theorem} \label{main2.1}
Suppose there exists $u_0\in L^2(\Omega; \Ll(\mathcal{O}))$ and $\z_0\in L^2(\Omega; L^2(\mathcal{O}))$ and $\pi \in \bar{\mathcal{U}}^{w}_{ad}(u_0,\z_0,T)$ such that $J(\pi)<+\infty,$ where $J$ is defined by 
\eqref{form.barJ}. Then the optimal control problem admits a weak optimal control with time horizon $[0,T]$ where $\bar{\mathcal{U}}^{w}_{ad}(u_0,\z_0,T)$ denotes the set of all weak admissible controls (with time horizon $[0,T]$). 
\end{theorem}

\section{Tidal Dynamics: The Deterministic Model}\label{model}
Under the assumptions that: (1) Earth is perfectly solid, (2) ocean tides do not change Earth's gravitational field, and (3) no energy exchange takes place between the mid-ocean and shelf zone, Marchuk and Kagan \cite{marchuk} obtained the following mathematical model
\begin{equation}
\partial_t w + A_1 w -\kappa_h\triangle w +\dfrac{r}{h}|w|w+g\nabla\xi =f,
\end{equation} 
\begin{equation}
\partial_t \xi + Div(hw)=0,
\end{equation}
in $[0,T]\times\mathcal{O}$, where $\mathcal{O}$ is a bounded 2-D domain (horizontal ocean basin) with coordinates $x=(x_1,x_2)$ and $t$ represents the time. Here $\partial_t$ denotes the time derivative, $\triangle,\nabla$ and $Div$ are the Laplacian, gradient and the divergence operators respectively.

The unknown variables $(w,\xi)$ represent the total transport 2-D vector (i.e., the vertical integral of the velocity from the ocean surface to the ocean floor) and the displacement of the free surface with respect to the ocean floor. The coefficients $A_1=[a_{ij}]$ is a 2-dimensional antisymmetric
square matrix with constant coefficients $a_{11}=a_{22}=0$ and
$-a_{12}=a_{21}=2\omega_z$, the Coriolis parameter (i.e.,
$\omega_z=\omega\cos(\varphi),$ $\omega$ is the angular velocity of
the Earth rotation and $\varphi$ the latitude), $\kappa_h>0$ the
constant horizontal macro turbulent viscosity coefficient, $r>0$ the
constant bottom friction coefficient equal to a numerical constant,
$g$ the Earth gravitational constant, $h=h(x)$ is the (vertical)
depth at $x$ in the region $\mathcal{O}$ and $f=\gamma_L
g\nabla \xi^+$ is the known tide-generating force with $\gamma_L$
the Love factor.

Following Manna et al. \cite{manna2008stochastic}, and Marchuk et al. \cite{marchuk}, we denote by $A$ the following matrix operator
\begin{equation}
\label{eq3}
A:=\left(\begin{array}{cc}
-\alpha\triangle & -\beta\\
\beta &  -\alpha\triangle
\end{array}\right),
\end{equation}
and the nonlinear vector operator
\begin{equation}
v\mapsto\gamma |v|v:=\left(\begin{array}{c}
\gamma(x)v_1\sqrt{v_1^2+v_2^2}\\
\gamma(x)v_2\sqrt{v_1^2+v_2^2}
\end{array}\right),
\end{equation}
where $\alpha:=\kappa_h$ and $\beta:=2\omega\cos (\varphi)$ are positive constants, $\gamma(x):=r/h(x)$ is a strictly positive smooth function. In this model we assume the depth $h(x)$ to be a continuously differentiable function of $x$, nowhere becoming zero, so that
\begin{equation}\label{bddh}
\min\limits_{x\in\mathcal{O}} h(x)=\epsilon >0,\qquad\max\limits_{x\in\mathcal{O}}h(x)=\mu,\quad\max\limits_{x\in\mathcal{O}}|\nabla h(x)|\leq M ,
\end{equation}
where M is some positive constant which equals zero at a constant ocean depth.

To reduce to homogeneous Dirichlet boundary conditions consider the natural change of unknown functions
\begin{equation}
u(t,x):=w(t,x)-w^0(t,x),
\end{equation}
and
\begin{equation}
\z(t,x):=\xi(t,x)+\int_0^t Div(hw^0(s,x))ds,
\end{equation}
which are referred to the tidal flow and the elevation. The full flow $w^0$ which is given a-priori on the boundary $\partial\mathcal{O}$, has been extended to the whole domain $[0,T]\times\mathcal{O}$ as a smooth function and still denoted by $w^0$.\\\\
Then the tidal dynamic equation can be written as 
\begin{equation}
\left\{\begin{array}{l}
\partial_t u+Au+\gamma |u+w^0|(u+w^0)+g\nabla\z=f^\prime\quad\text{in}\quad[0,T]\times\mathcal{O},\\
\partial_t \z + Div(hu)=0\quad\text{in}\quad[0,T]\times\mathcal{O},\\
u=0\quad\text{on}\quad[0,T]\times\partial\mathcal{O},\\
u=u_0,\;\;\z=\z_0\quad\text{in}\quad\{0\}\times\mathcal{O},
\end{array}\right.
\end{equation}
where
\begin{IEEEeqnarray}{lCl}
f^\prime =f-\dfrac{\partial w^0}{\partial t}+g\nabla\int_0^t Div(hw^0)dt-Aw^0,\\
u_0(x)=w_0(x)-w^0(x,0),\\
\z_0(x)=\xi_0(x).
\end{IEEEeqnarray}


\section{Functional Setting}\label{setting}
We use the (vector-valued) Sobolev spaces $\h(\mathcal{O}):=H^1_0(\mathcal{O},\mathbb{R}^2)$ and $\Ll(\mathcal{O}):=L^2(\mathcal{O},\mathbb{R}^2)$, with:\\
the norm on $\h(\mathcal{O})$ as
\begin{equation}
\|v\|_{\h}:=\left(\int_{\mathcal{O}}|\nabla v|^2 dx\right)
^{1/2},
\end{equation}
and the norm on $\Ll(\mathcal{O})$ as
\begin{equation}
\|v\|_{\Ll}:=\left(\int_{\mathcal{O}}|v|^2 dx\right)^{1/2}.
\end{equation}
Using the Gelfand triple $\h(\mathcal{O})\subset\Ll (\mathcal{O})\subset\mathbb{H}^{-1}(\mathcal{O})$, we may consider $\triangle$ or $\nabla$ as a linear map from $\h(\mathcal{O})$ or $\Ll(\mathcal{O})$ into the dual of $\h(\mathcal{O})$ respectively. The inner product in $\Ll(\mathcal{O})$ is denoted by $(\cdot,\cdot)_{\Ll}$ and is defined by
\begin{equation}
(u,v)_{\Ll}=\int_{\mathcal{O}}u(x)\cdot v(x)dx,
\end{equation}
for any $u$ and $v$ in $\Ll(\mathcal{O})$. Likewise, inner product in $L^2(\mathcal{O})$ is denoted by $(\cdot,\cdot)_{L^2}$. The induced duality between the spaces $\h(\mathcal{O})$ and $\mathbb{H}^{-1}(\mathcal{O})$ is denoted by $\langle\cdot,\cdot\rangle$.

\begin{lemma}
	\label{lemw}For any real-valued smooth function $\varphi$ and $\psi$ with compact support in $\mathbb{R}^2$, the following hold:
	\begin{align}
		&\|\varphi\psi\|_{L^2}^2\leq 4\|\varphi\partial_1\varphi\|_{L^2}\|\psi\partial_2\psi\|_{L^2},\\
		\label{eq13}
		&\|\varphi\|_{L^4}^4\leq 2\|\varphi\|_{L^2}^2\|\nabla\varphi\|_{L^2}^2.
	\end{align}
\end{lemma}
\noindent For proof see Ladyzhenskaya \cite{lady}.\\\\
Notice that by means of the Gelfand triple we may consider $A$, given by \eqref{eq3}, as a mapping of $\h(\mathcal{O})$ into its dual $\mathbb{H}^{-1}(\mathcal{O})$.

Define the non-symmetric bilinear form 
\begin{equation}
a(u,v):=\alpha[(\partial_1 u_1,\partial_1 v_1)_{L^2}+(\partial_2 u_2,\partial_2 v_2)_{L^2}]+\beta[(u_1,v_2)_{L^2}-(u_2,v_1)_{L^2}],
\end{equation} 
on $\h$. Thus if $u$ has a smooth second derivative then
\begin{equation*}
a(u,v)=(Au,v)_{\Ll},
\end{equation*}  
for every $v$ in $\h(\mathcal{O})$. Moreover, the bilinear form $a(\cdot,\cdot)$ is continuous and coercive in $\h(\mathcal{O})$,i.e.,
\begin{align}\label{prop.A}
	&|a(u,v)|\leq C_1\|u\|_{\mathbb{H}_0^1}\|v\|_{\mathbb{H}_0^1}\qquad\forall u,v\in\h(\mathcal{O}),\\
	&(Au,u)_{\Ll}=a(u,u)=\alpha\|u\|_{\mathbb{H}_0^1}^2,
\end{align}
for some positive constant $C_1=\alpha + \beta$.

Let us denote the nonlinear operator $B(\cdot)$ by
\begin{equation}
\label{eq4}
v\mapsto B(v):=\gamma |v+w^0|[v+w^0].
\end{equation}
Then we have the following lemma:
\begin{lemma}
	\label{mon}
	Let $u$ and $v$ be in $\mathbb{L}^4(\mathcal{O},\mathbb{R}^2)$. Then the following estimate holds:
	\begin{equation}
	\langle B(u)-B(v),u-v\rangle\geq 0.
	\end{equation}
\end{lemma}
\noindent For proof see Lemma 3.3 in Manna, Menaldi and Sritharan \cite{manna2008stochastic}.

The nonlinear operator $B(\cdot)$ is a continuous operator from $\mathbb{L}^4(\mathcal{O})$ to $\Ll(\mathcal{O})$, where
\begin{IEEEeqnarray}{lCr}
	\label{e2}
	\|B(v)\|_{\Ll}\leq C_2\|v+w^0\|_{\mathbb{L}^4}^2,\\
	\label{e1}
	\|B(u)-B(v)\|_{\Ll}\leq C_2[\|u+w^0\|_{\mathbb{L}^4}+\|v+w^0\|_{\mathbb{L}^4}]\|u-v\|_{\mathbb{L}^4},
\end{IEEEeqnarray}
where the constant $C_2$  is the sup-norm of the function $\gamma$.

\subsection{Preliminaries on Stochastic Processes}
In this Subsection we provide definitions and some properties of Hilbert space valued Wiener processes, L\'evy processes and Skorokhod spaces, most of which have been borrowed from the books by Da Prato and Zabczyk \cite{da}, Applebaum \cite{apple} and M\'{e}tivier \cite{metivier}.

Let $U$ and $V$ be two separable Hilbert spaces and let
$\mathcal{L}(U,V)$ denote the space of all bounded linear operators
from $U$ to $V$. Let $\{e_j\}_{j=1}^{\infty}$ be an orthonormal
basis in $U$. Then we say that a positive operator $Q\in\mathcal{L}(U,U)$
\ada{(i.e., $(Qx,x)_U
\geq 0$ for all non-zero $x\in U$)} is \textit{trace
	class} if $\mathlarger{\sum_{j=1}^{\infty}}(Qe_j,e_j)_{U}<\infty$.
Let $Q$ be a symmetric (i.e., $Q^*=Q$), positive, trace class
operator on $U$. Then there exist a sequence of eigenvalues
$\lambda_j$ with the corresponding sequence of eigenvectors
$\{e_j\}$ such that $Qe_j=\lambda_je_j$ for all $j=1,2,\cdots$ and
$$\mathrm{Tr}(Q)=\sum_{j=1}^{\infty}(Qe_j,e_j)_U=\sum_{j=1}^{\infty}(\lambda_je_j,e_j)_U
=\sum_{j=1}^{\infty}\lambda_j\|e_j\|_U^2=\sum_{j=1}^{\infty}\lambda_j<\infty.$$
\begin{definition}
	Let $U$ be a Hilbert space. A stochastic process $\{W(t)\}_{0\leq
		t\leq T}$ is said to be a $U$-valued $\mathcal{F}_t$-adapted Wiener
	process with covariance operator $Q$ if
	\begin{enumerate}
		\item [(i)] For each non-zero $h\in U$, $|Q^{\frac{1}{2}}h|^{-1} (W(t), h)_U$ is a standard one-dimensional Wiener process,
		\item [(ii)] For any $h\in U, (W(t), h)_U$ is a martingale adapted to $\mathcal{F}_t$.
	\end{enumerate}
\end{definition}
If $W$ is a $U$-valued Wiener process with covariance operator $Q$
with $\Tr Q < \infty$, then $W$ is a Gaussian process on $U$ and $
\mathbb{E}(W(t)) = 0,$ $\text{Cov}\ (W(t)) = tQ,$ $t\geq 0.$ Let
$U_0 = Q^{\frac{1}{2}}U.$ Then $U_0$ is a Hilbert space equipped
with the inner product $(\cdot, \cdot)_0$, $$(u, v)_0
=\sum_{k=1}^{\infty}\frac{1}{\lambda_k}(u,e_k)_U(v,e_k)_U=
\left(Q^{-\frac{1}{2}}u, Q^{-\frac{1}{2}}v\right)_U,\;\forall\; u,
v\in U_0,$$ \ada{where $Q^{-\frac{1}{2}}$ is the pseudo-inverse of
$Q^{\frac{1}{2}}$}. Since $Q$ is a trace class operator, the
imbedding of $U_0$ in $U$ is Hilbert-Schmidt.

Let $\mathcal{L}_Q(U,V)$ denote the space of
all Hilbert-Schmidt operator from $U_0$ to $V$. Let
$\{e_j\}_{j=1}^{\infty}$,
$\{g_j\}_{j=1}^{\infty}=\{\lambda_j^{\frac{1}{2}}e_j\}_{j=1}^{\infty}$
and $\{f_j\}_{j=1}^{\infty}$ be orthonormal bases for $U,U_0$ and
$V$ respectively.  Then the space $\mathcal{L}_Q(U,V)$ is also a
separable Hilbert space, equipped with the norm
\begin{align}\label{for1}
\|\Psi\|^2_{\mathcal{L}_Q(U,V)}&=\sum_{h=1}^{\infty}\|\Psi
g_h\|_{V}^2=\sum_{h=1}^{\infty}\sum_{k=1}^{\infty}\left|\left(\Psi
g_h,f_k\right)_V\right|^2=\sum_{h=1}^{\infty}\sum_{k=1}^{\infty}\left|\left(\Psi
\lambda_h^{\frac{1}{2}}e_h,f_k\right)_V\right|^2\nonumber\\&=\sum_{h=1}^{\infty}\sum_{k=1}^{\infty}\lambda_h\left|\left(\Psi
e_h,f_k\right)_V\right|^2=\sum_{h=1}^{\infty}\sum_{k=1}^{\infty}\left|\left(\Psi
Q^{\frac{1}{2}}e_h,f_k\right)_H\right|^2\nonumber\\&=\sum_{h=1}^{\infty}\|\Psi
Q^{\frac{1}{2}}\|_{V}^2 =\mathrm{Tr}((\Psi
Q^{\frac{1}{2}})^*\Psi Q^{\frac{1}{2}})=\mathrm{Tr}\left(\Psi
Q^{\frac{1}{2}}(\Psi Q^{\frac{1}{2}})^*\right),
\end{align}
where we used the fact that for a Hilbert-Schmidt operator $S$,
$\mathrm{Tr}(S^*S)=\mathrm{Tr}(SS^*)$. The scalar product between
two operators $\Psi,\Phi\in\mathcal{L}_Q(U,V)$ is defined by
\begin{align}\label{for2}
\left(\Psi,\Phi\right)_{\mathcal{L}_Q(U,V)}=\mathrm{Tr}\left((\Psi
Q^{\frac{1}{2}})(\Phi Q^{\frac{1}{2}})^*\right).
\end{align}
Since the Hilbert spaces $U_0$ and $V$ are separable, the space
$\mathcal{L}_Q(U,V)$ is also separable. For an orthonormal basis
$\{e_j\}_{j=1}^{\infty}$ in $U$, an element $u\in U_0$ can be
represented as
$$u=\sum_{k=1}^{\infty}\left(u,\lambda_k^{\frac{1}{2}}e_k\right)_0\lambda_k^{\frac{1}{2}}e_k.$$
Let $\Psi\in\mathcal{L}(U,V)$ be considered as an operator from
$U_0$ to $V$, then we can write $\Psi u$ as
$$\Psi u=\sum_{k=1}^{\infty}\left(u,\lambda_k^{\frac{1}{2}}e_k\right)_0\lambda_k^{\frac{1}{2}}\Psi
e_k.$$ Then $\Psi$ has a finite Hilbert-Schmidt norm, since
\begin{align}\label{hsnorm}
\|\Psi\|_{\mathcal{L}_Q(U,V)}^2&=\sum_{k=1}^{\infty}\|\Psi
Q^{\frac{1}{2}}e_k\|_{V}^2=\sum_{k=1}^{\infty}\|\Psi
\lambda_{k}^{\frac{1}{2}}e_k\|_V^2=\sum_{k=1}^{\infty}\lambda_k\|\Psi
e_k \|_V^2\nonumber\\&\leq \sum_{k=1}^{\infty}\lambda_k\|\Psi
\|_{\mathcal{L}(U,V)}^2\|e_k\|_U^2=\mathrm{Tr}(Q)\|\Psi\|_{\mathcal{L}(U,V)}^2,
\end{align}
and hence $\mathcal{L}(U,V)\subset \mathcal{L}_Q(U,V)$. Hence if
$\Psi,\Phi\in\mathcal{L}(U,V)$, then from \eqref{for1} and
(\ref{for2}), we have
$$\|\Psi\|^2_{\mathcal{L}_Q(U,V)}=\mathrm{Tr}(\Psi
Q\Psi^*)\textrm{ and
}\left(\Psi,\Phi\right)_{\mathcal{L}_Q(U,V)}=\mathrm{Tr}(\Psi
Q\Phi^*).$$ For more details see Da Prato and Zabczyk \cite{da},
Gawarecki and Mandrekar \cite{GaMa} etc.

\begin{definition}
	A c\`{a}dl\`{a}g adapted process (paths are right continuous with
	left limits), $(\mathbf{L}_t)_{t\geq 0}$, is called a L\'{e}vy
	process if it has stationary independent increments and is
	stochastically continuous.
\end{definition}
Let $(\mathbf{L}_t)_{t\geq 0}$ be a $V$-valued L\'{e}vy process.
Hence, for every $\omega\in\Omega$, $\mathbf{L}_t(\omega)$ has
countable number of jumps on $[0, t]$. Note that for every
$\omega\in\Omega$, the jump $\triangle \mathbf{L}_t(\omega) =
\mathbf{L}_t(\omega) - \mathbf{L}_{t-}(\omega)$ is a point function
in $\mathscr{B}(V\backslash\{0\})$. Let us define 
\begin{align*}
N(t, Z)&= N(t, Z,
\omega)\\&:= \#\left\{s\in (0, \infty): \triangle\mathbf{L}_s(\omega)\in
Z\right\}, t>0, Z\in\mathscr{B}(V\backslash\{0\}), \omega\in\Omega
\end{align*}
as the \emph{Poisson random measure associated with the L\'{e}vy
	process} $(\mathbf{L}_t)_{t\geq 0}$.

The differential form of the measure $N(t,Z,\omega)$ is written as
$N(dt, dz)(\omega)$. We call $\tilde{N}(dt, dz) = N(dt, dz) -
\lambda(dz)dt $ a \emph{compensated Poisson random measure (cPrm)},
where $\lambda(dz)dt $ is known as \emph{compensator} of the
L\'{e}vy process $(\mathbf{L}_t)_{t\geq 0}$. Here $dt$ denotes the
Lebesgue measure on $\mathscr{B}(\mathbb{R}^{+})$, and $\lambda(dz)$
is a $\sigma$-finite L\'{e}vy measure on $(Z, \mathscr{B}(Z))$.
\begin{definition}
	Let $V$ and $F$ be separable Hilbert spaces. Let $F_t:=
	\mathscr{B}(V)\otimes\mathcal{F}_t$ be the product $\sigma$-algebra
	generated by the semi-ring $ \mathscr{B}(V)\times\mathcal{F}_t$ of
	the product sets $Z\times F,$ $Z\in\mathscr{B}(V),$ $F\in
	\mathcal{F}_t$ (where $\mathcal{F}_t$ is the filtration of the
	additive process $(\mathbf{L}_t)_{t\geq 0}$). Let $T>0$, define
	\begin{align*}
	\mathbb{H}(Z) = & \big\{g : \mathbb{R}^+ \times Z \times \Omega
	\rightarrow F,\text{ such that $g$ is}\;\; F_T/\mathscr{B}(F) \;\;
	measurable\;\;and \nonumber\\& \;\;\qquad g(t,z,\omega)\;\;is\;\;
	\mathcal{F}_t - adapted\;\;\forall z\in Z, \forall t\in (0,T]\big\}.
	\end{align*}
	For $p\geq1$, let us define,
	$$\mathbb{H}^{p}_\lambda([0,T]\times Z;F) = \left\{g\in \mathbb{H}(Z) :
	\int_0^T\int_Z\mathbb{E}[\|g(t,z,\omega)\|^p_F]\lambda(dz)dt < \infty \right\}.$$
\end{definition}
For more details see Mandrekar and R\"{u}diger \cite{MR}.

Let $S$ be a complete separable metric space with a metric $d$. Let
us denote by $D([0,T];S)$, the set of all $S$-valued functions
defined on $[0,T]$, which are right continuous and have left limits
(c\`{a}dl\`{a}g functions) for every $t\in[0,T]$. The space
$D([0,T];S)$ is endowed with the Skorokhod $J$-topology.

\begin{definition}(The $J$-topology on $D([0,T];S)$) This
	topology can be defined by the following metric $\delta_T$.
	\begin{align*}
		&\delta_T(x,y)\\&:=\inf_{\lambda\in\Lambda_T}\left[\sup_{t\in[0,T]}d(x(t),x\circ\lambda(t))+\sup_{t\in[0,T]}|t-\lambda(t)|+\sup_{s\neq
			t}\left|\log\left(\frac{\lambda(t)-\lambda(s)}{t-s}\right)\right|\right],
	\end{align*}
	where $\Lambda_T$ is the set of increasing homeomorphisms of
	$[0,T]$.
	
	A sequence $(x_n)_{n\in\mathbb{N}}$ in $D([0,T];S)$ converges to $x$
	in $D([0,T];S)$ if and only if there exists a sequence $(\lambda_n)$
	of homeomorphisms of $[0,T]$ such that
	$\mathlarger{\lim_n}\;\lambda_n$ is equal to identity mapping and
	$x=\mathlarger{\lim_n}\; x_n\circ\lambda_n$, both convergence being
	uniform on $[0,T]$.
\end{definition}
The Skorokhod topology relativized to $C([0,T];S)$ coincides with
the uniform topology there. It should be remarked that, if a
sequence $(x_n)$ in $D([0,T];S)$ converges to $x$ for the metric
$\delta_T$, then $\mathlarger{\lim_n}\;x_n(T)=x(T).$ The space
$D([0,T];S)$ is separable and complete. For more details see Chapter
2, M\'{e}tivier \cite{metivier} and Chapter 3, Billingsley \cite{BiP}.

\begin{remark}
	Let us define $D([0,T];V)$ as the space of all c\`{a}dl\`{a}g paths
	from $[0,T]$ into $V$, where $V$ is a Hilbert space.
\end{remark}

\begin{definition} Let $M$ be a V-valued square integrable martingale on
	$(\Omega,\mathcal{F},(\mathcal{F}_{t})_{t\geq 0},\mathbb{P})$, with
	right continuous paths. Then there exists two real right continuous
	increasing processes $[M]$ and $\triangleleft M\triangleright$ with
	$0=[M]_0=\triangleleft M\triangleright_0$ such that
	\begin{align}
	\|M_t\|_V^2=\|M_0\|_V^2+2\int_0^t(M_{s-},\d M_s)_V+[M]_t.
	\end{align}
	$\triangleleft M\triangleright$ is the unique real right continuous
	increasing predictable process such that
	\begin{align}
	\|M_t\|_V^2-\|M_0\|_V^2-\triangleleft M\triangleright_t \textrm{ is
		a martingale.}
	\end{align}
	Here $[M]$ is called the quadratic variation of $M$ and
	$\triangleleft M\triangleright$ the Meyer process of $M$.
\end{definition}
\begin{remark}
	If $M$ is continuous, then we have $\triangleleft
	M\triangleright=[M]$.
\end{remark}
\begin{example}
	Let $V$ be Hilbert spaces and let $Q:V\to V$ be a trace class
	operator. Let $$\d u(t)=\sigma(t,u)\d
	W(t)+\mathlarger{\int_Z}g(u(t-),z)\tilde{N}(\d t,\d z),$$ where
	$W(\cdot)$ is a $V$-valued Wiener process,
	$\sigma(\cdot,\cdot):[0,t)\times V\to\mathcal{L}_Q(V,V)$, $Z\in\mathscr{B}(V\backslash\{0\})$, $g(\cdot,\cdot) :V\times Z\to V$ and
	$\tilde{N}(\cdot,\cdot)$ is the compensated Poisson random measure.
 From M\'{e}tivier \cite{metivier}, Sakthivel and Sritharan \cite{sakthivel} and Manna, Manil and Sritharan \cite{MMS}, we observe that the quadratic variation process of $u$ is
	$$[u]_t=\int_0^t\|\sigma(s,u)\|_{\mathcal{L}_Q(V,V)}^2\d
	s+\int_0^t\int_Z\|g(u,z)\|_V^2 N(\d s,\d z)   $$ and the Meyer
	process of $u$ is
	$$\triangleleft u\triangleright_t=\int_0^t\|\sigma(s,u)\|_{\mathcal{L}_Q(V,V)}^2\d
	s+\int_0^t\int_Z\|g(u,z)\|_V^2 \lambda(\d z)\d s,$$ 
	and finally by the martingale property one obtains that the expectation of the quadratic variation process and that of Meyer process are same. This yields 
	\begin{align}\label{mt1}
	\mathbb{E}\left[\int_0^t\int_Z\|g(u,z)\|_V^2N(\d s,\d
	z)\right]= \mathbb{E}\left[\int_0^t\int_Z\|g(u,z)\|_V^2\lambda(\d z)\d s\right].
	\end{align}
\end{example}

\begin{lemma}(Burkholder-Davis-Gundy Inequality)\label{burk}
	Let $M$ be a $V$- valued c\`{a}dl\`{a}g martingale with
	$M_0=0$ and let $p\geq 1$ be fixed. Then for any
	$\mathcal{F}$-stopping time $\tau$, there exists constants $c_p$ and
	$C_p$ such that
	$$\mathbb{E}\left\{[M]_{\tau}^{p/2}\right\}\leq c_p\mathbb{E}\left\{\sup_{0\leq t\leq\tau}\|M_t\|_V^p\right\}\leq C_p
	\mathbb{E}\left\{[M]_{\tau}^{p/2}\right\}$$ for all $\tau$, $0\leq
	\tau\leq \infty$, where $[M]$ is the quadratic variation of process
	$M$. The constants
	are universal, they do not depend on the choice of $M$.
\end{lemma}
\noindent For proof see Theorem 1.1 of Marinelli and R\"{o}ckner
\cite{MCRM}. For real-valued c\`{a}dl\`{a}g martingales see Theorem
3.50 of Peszat and Zabczyk \cite{PZ}.

\subsection{Stochastic Tide Model}
Let us consider $U=V=\Ll$. Then with the above functional setting, we recall the stochastic tidal dynamics equation \eqref{me1}-\eqref{me3} with the L\'evy forcing
where $u_0\in L^2(\Omega; \Ll(\mathcal{O}))$, and $\z_0\in L^2(\Omega; L^2(\mathcal{O}))$. The operators $A$ and $B$ are defined through \eqref{eq3} and \eqref{eq4} respectively. $(W(t))_{t\geq 0}$ is a $\Ll(\mathcal{O})$-valued Wiener process with trace class covariance. $H(\cdot,\cdot)$ is a measurable mapping from $\h\times Z$ into $\Ll$.

\begin{assumption} \label{hypo.noi}
We assume that $\s$ and $H$ satisfy the following hypotheses:
\begin{enumerate}\label{Hyp}
\item[H.1] $\s\in C([0,T]\times\h(\mathcal{O});L_Q(\mathbb{L}^2,\Ll)), H\in\mathbb{H}^2_\lambda([0,T]\times Z;\Ll(\mathcal{O})).$
\item[H.2] For all $t\in(0,T)$, there exists a positive constant $K$ such that for all $u\in\Ll(\mathcal{O})$ 
\begin{equation*}
\|\s(t,u)\|_{L_Q(\mathbb{L}^2,\Ll)}^2+\int_Z\|H(u,z)\|_{\Ll}^2\la\leq K(1+\|u\|_{\Ll}^2).
\end{equation*}
\item[H.3] For all $t\in(0,T)$, there exists a positive constant $L$ such that for all $u,v\in\Ll(\mathcal{O})$  
\begin{equation*}
\|\s(t,u)-\s(t,v)\|_{L_Q(\mathbb{L}^2,\Ll)}^2+\int_Z \|H(u,z)-H(v,z)\|_{\Ll}^2\la 
\leq L\|u-v\|_{\Ll}^2.
\end{equation*}
\end{enumerate}
\end{assumption}
\begin{lemma}\label{Monotone lemma}
	Let $F(\cdot): \h\to\mathbb{H}^{-1}$ be a nonlinear operator defined by $F(u):=Au+B(u)-f$. Then for $u, v\in\h\mathcal{(O)}$,
	\begin{align}\label{monotone}
&	\langle F(u)-F(v), hu-hv\rangle+\Lambda\|u-v\|_{\h}^2
	\geq 0 ,
	\end{align}
	where $\Lambda:=\Lambda(\alpha, \epsilon, M)$ is a real constant and $\epsilon$ and $M$ are defined in \eqref{bddh} and $\alpha$ is defined via the definition of the operator $A.$ 
\end{lemma}
\begin{proof}
	From the definition of the operator $A$, assumptions \eqref{bddh} on $h(x)$, applying Young's and Poincar\'e inequalities, we obtain
	\begin{align*}
	\langle Au, hu\rangle &= \alpha(\nabla u, h\nabla u)_{\Ll}+\alpha(\nabla u, u\nabla h)_{\Ll}\geq \alpha\epsilon\|\nabla u\|_{\Ll}^2-\alpha M\|\nabla u\|_{\Ll}\| u\|_{\Ll}\\
	& \geq \dfrac{\alpha\epsilon}{2}\|\nabla u\|_{\Ll}^2 - \dfrac{\alpha M^2}{2\epsilon}\| u\|_{\Ll}^2 \geq  \left(\dfrac{\alpha\epsilon}{2} - \dfrac{\alpha M^2}{2\epsilon \lambda_1}\right)\| u\|_{\h}^2,
	\end{align*} where $\lambda_1$ is the first eigenvalue of $-\triangle.$ 
Hence for any $u, v\in\h(\mathcal{O})$,
$$ \langle Au-Av, hu-hv\rangle \geq \left(\dfrac{\alpha\epsilon}{2} - \dfrac{\alpha M^2}{2\epsilon \lambda_1}\right)\| u-v\|_{\h}^2.$$
From Lemma \ref{mon}, we can easily show that for $u, v\in \h(\mathcal{O})\subset\mathbb{L}^4(\mathcal{O})$
$$\langle B(u)-B(v), hu-hv\rangle \geq0.$$
Hence \eqref{monotone} is obvious with $\Lambda:=\mathlarger{\frac{\alpha}{2\epsilon }(\frac{M^2}{\lambda_1}-\epsilon^2)}.$
\end{proof}
\begin{remark}\label{monotone4} (monotonicity in $\h$):
	 If we assume $\Lambda\,<\,0$, then the operator $F(\cdot)$ is monotone on $\h(\mathcal{O}).$ 
	  Moreover, since $F(\cdot)$ is hemicontinuous operator, it is maximal monotone from $\h(\mathcal{O})$ to $\mathbb{H}^{-1}(\mathcal{O})$ (see Theorem 1.3 of Barbu \cite{VB}).  

	However, in general, if we do not assume negativity of the constant $\Lambda,$ 
	then 
the operator $F(\cdot)$ is $\Lambda$ - monotone on $\h(\mathcal{O})$
i.e.  $F+\Lambda I$ is monotone on $\h(\mathcal{O}).$
\end{remark} 
\begin{remark}\label{monotone2} (monotonicity in $\Ll$):
	Use of Young's inequality with different parameters can lead to 	\begin{align*}
	\langle Au, hu\rangle &= \alpha(\nabla u, h\nabla u)_{\Ll}+\alpha(\nabla u, u\nabla h)_{\Ll}\geq \alpha\epsilon\|\nabla u\|_{\Ll}^2-\alpha M\|\nabla u\|_{\Ll}\| u\|_{\Ll}\\
	& \geq \dfrac{\alpha\epsilon}{2}\|\nabla u\|_{\Ll}^2 - \dfrac{\alpha\epsilon}{2}\|\nabla u\|_{\Ll}^2 - \dfrac{\alpha M^2}{2\epsilon}\| u\|_{\Ll}^2, 
	\end{align*}
	which in turn yields
		 \begin{align}\label{monotone1}
		 &	\langle F(u)-F(v), hu-hv\rangle
		 + \dfrac{\alpha M^2}{2\epsilon}\|u-v\|_{\Ll}^2 \geq 0.
		\end{align}
		Hence $F(\cdot)$ is $\Lambda_1$ - monotone on $\Ll(\mathcal{O})$ (with $\Lambda_1:=\dfrac{\alpha M^2}{2\epsilon}>0$), and this interesting property will be exploited in the proof of existence theorem in the next Section.
\end{remark}


\section{Energy Estimates and Existence Result}\label{estimate}


Let  $\{e_{j}:j \in \mathbb{N}\}$ be an orthonormal basis in $\Ll(\mathcal{O})$ such that span$\{e_{j}:j \in \mathbb{N}\}$ is dense in $\h(\mathcal{O}).$ Let us fix $n \in \mathbb{N}.$ Let $P_n$ denote the projection of $\mathbb{H}^{-1}(\mathcal{O})$ onto $\Ll_n(\mathcal{O}):=$\,span\,$\{e_{1},e_{2},\cdots ,e_{n}\}$ defined by
$$ P_nu:=\sum_{i=1}^n \langle u,e_i \rangle e_i,\quad u\in \mathbb{H}^{-1}(\mathcal{O}).$$ 
Note that the restriction of $P_n$ to $\Ll(\mathcal{O}),$ still denoted by $P_n,$ is given by
$$ P_nu:=\sum_{i=1}^n (u,e_i)_{\Ll} e_i,\quad u \in \Ll(\mathcal{O}),  $$
and thus it is the $(\cdot,\cdot)_{\Ll}-$orthogonal projection onto $\Ll_n(\mathcal{O}).$

In similar manner we can denote the orthogonal projection of $L^2(\mathcal{O})$ to $L^2_n(\mathcal{O})$ by the same notation $P_n.$

\begin{lemma} \label{Pn.lem} For every $u \in \Ll(\mathcal{O}),\, v \in L^2(\mathcal{O}), w \in \h(\mathcal{O}),$ 
\begin{align*}
\lim_{n \rightarrow \infty}\|P_nw -w\|_{\h}=0,\lim_{n \rightarrow \infty}\|P_nu -u\|_{\Ll}=0, \lim_{n \rightarrow \infty}\|P_nv -v\|_{L^2}=0.
\end{align*}
\end{lemma}
\begin{proof}
See Lemma 2.4 of \cite{BM}.
\end{proof}

For every $n \in \mathbb{N},$ we consider the finite dimensional system of SDEs in variational form on $\Ll_n(\mathcal{O})$ and $L^2_n(\mathcal{O})$ respectively given by
\begin{align}
\label{eqn5}
&du^n(t)+\Big(Au^n(t)+B(u^n(t))+g\nabla\z^n(t)\Big)dt =\ada{f}(t)dt\nonumber\\&\quad+\s^n(t,u^n(t))dW^n(t)
+\int_Z H^n(u^n(t-),z) \n \\
\label{eqn6.1}
&d\z^n(t)+ Div(hu^n(t)) dt=0,\\
&u^n(0)=P_nu_0=u^n_0,\quad \z^n(0)=P_n\z_0=\z^n_0,\label{eqn6}
\end{align} 
where we have denoted $W^{n}=P_{n} W,$ $\sigma^{n}=P_{n}\sigma,$ and 
$\int_{Z}H^{n}(\cdot,z)\tilde {N}(dt,dz)=P_n\int_{Z}H(\cdot,z)$ $\tilde N(dt,dz)$ with $H^n=P_nH.$ Due to ``weak" nonlinearity of $B$, it is straightforward to show that the range of $B(u^n)$ and that of $P_nB(u^n)$ are $\Ll_n(\mathcal{O})$, and hence we write $P_nB(u^n)$ as $B(u^n)$. For simplicity, we write $P_nf$ as $f$ in equation \eqref{eqn5}, since due to Lemma \ref{Pn.lem} $P_nf \rightarrow f$ in $L^2(0,T;\Ll(\mathcal{O}))$.
\begin{remark}
We note that \eqref{e1} ensures that $B$ is locally Lipschitz and Assumption \ref{Hyp} guarantees that the noise coefficients $\sigma^{n}$ and $H^{n}$ are globally Lipschitz.
Hence, we infer that for all $n\geq1,$ there exist adapted processes $u^n \in D([0,T];\Ll_n(\mathcal{O}))$ a.s. and $\z^n \in C([0,T];L^2_n(\mathcal{O}))$ a.s. such that \eqref{eqn5}-\eqref{eqn6} are satisfied. For proof see Albeverio et. al \cite{ABW}.
\end{remark}


\begin{proposition}
\label{prop}
Under the above mathematical setting, let
\begin{equation}
\left\{\begin{array}{l}
w^0\in L^4(\Omega; L^4(0,T;\mathbb{L}^4(\mathcal{O}))),\;f\in  L^2(\Omega; L^2(0,T;\Ll(\mathcal{O}))),\\
\s\in C([0,T]\times\h(\mathcal{O});L_Q(\mathbb{L}^2,\Ll)),\;H\in\mathbb{H}^2_\lambda([0,T]\times Z;\Ll(\mathcal{O})),\\
u_0\in L^2(\Omega;\Ll(\mathcal{O})),\z_0\in  L^2(\Omega;L^2(\mathcal{O})).
\end{array}\right.
\end{equation}
If $(u^n(t),\z^n(t))$ denotes the unique strong solution of the system \eqref{eqn5}-\eqref{eqn6}, then we have the following a-priori estimates:
\begin{align}
\label{eq14}
\mathbb{E}[\|u^n(t)\|_{\Ll}^2+\|\z^n(t)\|_{L^2}^2]+2\alpha\mathbb{E}[\int_0^t\|u^n(s)\|_{\h}^2ds]\leq C_{1(2)} \;\;\forall t\in[0,T],\\
\label{eq15}
\mathbb{E}[\sup\limits_{0\leq t\leq T}(\|u^n(t)\|_{\Ll}^2 + \|\z^n(t)\|_{L^2}^2)]+2\alpha\mathbb{E}[\int_0^T\|u^n(t)\|_{\h}^2dt]\leq C_{2(2)},
\end{align}
where the constants $C_{1(2)}$ and $C_{2(2)}$ depend on the coefficients $\alpha, g,M,\mu$ and the norms $\|f\|_{ L^2(\Omega; L^2(0,T;\Ll))},$ $\|w^0\|_{L^4(\Omega; L^4(0,T;\mathbb{L}^4))},$ $\|u_0\|_{L^2(\Omega;\Ll)}$, $\|\z_0\|_{L^2(\Omega;L^2)}$ and $T$.
\end{proposition}
\begin{proof}
Define the stopping time as:
\begin{equation}
\tau_N=\inf\{t \geq 0:\|u^n(t)\|_{\Ll}^2+\|\z^n(t)\|_{L^2}^2+\int_0^t\|u^n(s)\|^2_{\mathbb{H}_0^1}ds>N\}.
\end{equation}
Applying It\^o's lemma (for reference, see Theorem 3.7.2 of Mandrekar and
R\"{u}diger \cite{MaRu1}, Theorem 4.4 of R\"{u}diger and Ziglio
\cite{RuZ}, Section 2.3, M\'{e}tivier \cite{metivier}) \ada{to the process $\|u^n(t)\|_{\Ll}^2$} 
\begin{align}
\label{eqn 7}
&\|u^n(\L)\|_{\Ll}^2+ 2\alpha\int_0^{\L} \|u^n(s)\|^2_{\mathbb{H}_0^1}ds+2\int_0^{\L}(B(u^n(s)),u^n(s))_{\Ll}ds \notag \\
& \quad +2\int_0^{\L}(g\nabla\z^n(s) ,u^n(s))_{\Ll}ds
=\|u^n_0\|_{\Ll}^2+2\int_0^{\L}(f(s),u^n(s))_{\Ll}ds\notag\\
&\quad+2\int_0^{\L}(\s^n(s,u^n(s))dW^n(s),u^n(s))_{\Ll}+\int_0^{\L}\|\s^n(s,u^n(s))\|_{L_Q(\mathbb{L}^2;\Ll)}^2ds \notag\\
& \quad+\int_0^{\L}\int_Z\| H^n(u^n(s-),z)\|_{\Ll}^2 N(ds,dz)\notag\\
& \quad +2\int_0^{\L}\int_Z (H^n(u^n(s-),z),u^n(s-))_{\Ll}\N .
\end{align}
Using the definition of the operator $B(\cdot)$ and Lemma \ref{mon}
\begin{IEEEeqnarray*}{rCl} \label{B.W0.2}
(B(u^n(t)),u^n(t))_{\Ll}&\geq -&\int_\mathcal{O}\gamma(x)|w^0(t)|w^0(t)\,u^n(t)dx\\
 &\geq & -\dfrac{r}{\epsilon}\|w^0(t)\|_{\mathbb{L}^4}^2\|u^n(t)\|_{\Ll}\\
&\geq & -\dfrac{r}{2\epsilon}[\|w^0(t)\|_{\mathbb{L}^4}^4+\|u^n(t)\|_{\Ll}^2].\IEEEyesnumber
\end{IEEEeqnarray*}
Using the Divergence Theorem and the inequality $2ab\leq \delta a^2 + \dfrac{1}{\delta}b^2$ 
we obtain,
\begin{align} \label{eq26}
 & |g(\nabla\z^n(t) ,u^n(t))_{\Ll}| =|-g(\z^n(t),Div (u^n(t)))_{L^2}| \notag \\
&\leq \dfrac{g}{2}[\dfrac{2g}{\alpha}\|\z^n(t)\|_{L^2}^2+\dfrac{\alpha}{2g}\|Div (u^n(t))\|_{L^2}^2] \notag \\
&\leq \dfrac{g}{2}[\dfrac{2g}{\alpha}\|\z^n(t)\|_{L^2}^2+\dfrac{\alpha}{2g}\| u^n(t)\|_{\mathbb{H}_0^1}^2].
\end{align}
Using Cauchy-Schwarz inequality
\begin{equation}
|(f(t),u^n(t))_{\Ll}|\leq\dfrac{1}{2}[\|f(t)\|_{\Ll}^2+\|u^n(t)\|_{\Ll}^2].
\end{equation}
Hence the energy equality \eqref{eqn 7} yields
\begin{align}
\label{eq8}
&\|u^n(\L)\|_{\Ll}^2+ 2\alpha\int_0^{\L}\|u^n(s)\|^2_{\mathbb{H}_0^1} ds \notag \\
&\leq  \int_0^{\L}\|f(s)\|_{\Ll}^2 ds 
 + \int_0^{\L}\|u^n(s)\|_{\Ll}^2 ds +\dfrac{2g^2}{\alpha}\int_0^{\L} \|\z^n(s)\|_{L^2}^2ds \notag \\
&\quad +\dfrac{r}{\epsilon}\int_0^{\L}\Big[\|w^0(s)\|_{\mathbb{L}^4}^4+\|u^n(s)\|_{\Ll}^2\Big] ds 
+\dfrac{\alpha}{2}\int_0^{\L}\| u^n(s)\|_{\mathbb{H}_0^1}^2ds \notag \\
&\quad+2\int_0^{\L}(\s^n(s,u^n(s))dW^n(s),u^n(s))_{\Ll}
+\int_0^{\L}\|\s^n(s,u^n(s))\|_{L_Q(\mathbb{L}^2;\Ll)}^2ds \notag \\
&\quad+\int_0^{\L}\int_Z\| H^n(u^n(s-),z)\|^2_{\Ll} N(ds,dz) +\|u^n_0\|_{\Ll}^2 \notag\\
&\quad+2\int_0^{\L}\int_Z (H^n(u^n(s-),z),u^n(s-))_{\Ll}\N.
\end{align}
Using equation \eqref{eqn6.1}
\begin{equation} \label{ode.z}
\dfrac{1}{2}\dfrac{d}{dt}\|\z^n(t)\|_{L^2}^2=-(Div(hu^n(t)),\z^n(t))_{L^2}.
\end{equation}
Now by using H\"{o}lder's inequality, assumptions on $h$ in \eqref{bddh} and Young's inequality we get
\begin{align} \label{div.esti}
&|(Div(hu^n(t)),\z^n(t))_{L^2}|=|(h\,Div(u^n(t)),\z^n(t))_{L^2}+(u^n(t)\cdot\nabla h,\z^n(t))_{L^2}|\notag\\
&\leq |(h\,Div(u^n(t)),\z^n(t))_{L^2}|+|(u^n(t)\cdot\nabla h,\z^n(t))_{L^2}|\notag\\
&\leq \| h\|_{L^\infty}\|Div(u^n(t))\|_{L^2}\|\z^n(t)\|_{L^2}+\|u^n(t)\|_{\Ll}\|\nabla h\|_{\mathbb L^\infty}\|\z^n(t)\|_{L^2}\notag\\
&\leq \mu\| u^n(t)\|_{\mathbb{H}_0^1}\|\z^n(t)\|_{L^2} +M\|u^n(t)\|_{\Ll}\|\z^n(t)\|_{L^2}\notag\\
&\leq \dfrac{\mu}{2}\left(\dfrac{\alpha}{2\mu}\|u^n(t)\|_{\mathbb{H}_0^1}^2 +\dfrac{2\mu}{\alpha}\|\z^n(t)\|_{L^2}^2\right)+\dfrac{M}{2}[\|u^n(t)\|_{\Ll}^2+\|\z^n(t)\|_{L^2}^2].
\end{align}
Thus integrating \eqref{ode.z} in $[0,\L]$ and using \eqref{div.esti} we get
\begin{align}\label{eq9} \|\z^n(\L)\|_{L^2}^2 &\leq \|\z^n(0)\|_{L^2}^2+ M\int_0^{\L}\|u^n(s)\|_{\Ll}^2 ds\nonumber \\
&\quad+(\dfrac{2\mu^2}{\alpha}+M)\int_0^{\L}\|\z^n(s)\|_{L^2}^2 ds+\dfrac{\alpha}{2}\int_0^{\L}\|u^n(s)\|_{\mathbb{H}_0^1}^2 ds.
\end{align}
Adding equations \eqref{eq8} and \eqref{eq9}
\begin{align}
\label{eq11}
&\| u^n  (t\wedge\tau_N)\|_{\Ll}^2+\|\z^n(t\wedge\tau_N)\|_{L^2}^2+\alpha\int_0^{t\wedge\tau_N}\|u^n(s)\|^2_{\mathbb{H}_0^1}ds\notag\\
&\leq  (1+M+\dfrac{r}{\epsilon})\int_0^{t\wedge\tau_N}\|u^n(s)\|_{\Ll}^2 ds 
+ (\dfrac{2g^2}{\alpha}+\dfrac{2\mu^2}{\alpha}+M)\int_0^{t\wedge\tau_N}\|\z^n(s)\|_{L^2}^2 ds\notag\\
&\quad+\dfrac{r}{\epsilon}\int_0^{t\wedge\tau_N}\|w^0(s)\|_{\mathbb{L}^4}^4 ds + \int_0^{t\wedge\tau_N}\|f(s)\|_{\Ll}^2 ds +\int_0^{t\wedge\tau_N}\|\s^n(s,u^n(s))\|_{L_Q(\mathbb{L}^2;\Ll)}^2 ds \notag\\
 &\quad+2\int_0^{t\wedge\tau_N}(\s^n(s,u^n(s))dW^n(s),u^n(s))_{\Ll} +\|u^n_0\|_{\Ll}^2 +\|\z^n_0\|_{L^2}^2 \notag\\
  &\quad+\int_0^{t\wedge\tau_N}\int_Z\| H^n(u^n(s-),z)\|^2_{\Ll} N(ds,dz) \notag\\
 &\quad+\int_0^{t\wedge\tau_N}2\int_Z (H^n(u^n(s-),z),u^n(s-))_{\Ll}\ns. 
\end{align}
Let $C=\max \{1+M+\dfrac{r}{\epsilon},\dfrac{2g^2}{\alpha}+\dfrac{2\mu^2}{\alpha}+M\}.$ We note that the expectation of the quadratic variation process and that of Meyer process are same, i.e.
\begin{align}
\label{eq12}
&\mathbb{E}\left[\int_0^t\int_Z \| H^n(u^n(s-),z)\|^2_{\Ll} N(ds,dz)\right]  \notag\\ 
& =\mathbb{E}\left[\int_0^t\int_Z\| H^n(u^n(s),z)\|^2_{\Ll} \la ds\right]. 
\end{align}
and the stochastic integrals
\begin{align*}
&\int_0^{t\wedge\tau_N}(\s^n(s,u^n(s))dW^n(s),u^n(s))_{\Ll},\\ &\int_0^{t\wedge\tau_N}\int_Z (H^n(u^n(s-),z),u^n(s-))_{\Ll}\ns
\end{align*}
are local martingales with zero averages. 
Furthermore, we observe that \begin{align} \label{eq.sig.n}
&\|\s^n(s,u^n(s))\|_{L_Q(\Ll;\Ll)}\leq c\|P_n\|_{\mathcal{L}(\Ll, \Ll)} \|\s(s,u^n(s))\|_{L_Q(\Ll;\Ll)} \leq c \|\s(s,u^n(s))\|_{L_Q(\Ll;\Ll)},
\end{align}
and similarly
\begin{align} \label{eq.sig.n1}
\| H^n(u^n(s),z)\|_{\Ll} \leq c\|P_n\|_{\mathcal{L}(\Ll, \Ll)} \| H(u^n(s),z)\|_{\Ll} \leq c \| H(u^n(s),z)\|_{\Ll}.
\end{align}
Hence after taking expectation of equation \eqref{eq11}, using \eqref{eq.sig.n},  \eqref{eq.sig.n1} and Assumption \ref{Hyp} we obtain the following 
\begin{align*}
&\mathbb{E}\Big[\|u^n(t\wedge\tau_N)\|_{\Ll}^2+\|\z^n(t\wedge\tau_N)\|_{L^2}^2\Big]+\alpha\mathbb{E}\left[\int_0^{t\wedge\tau_N}\|u^n(s)\|^2_{\mathbb{H}_0^1}ds\right]\\
&\leq C\mathbb{E}\left[\int_0^{t\wedge\tau_N}\Big(\|u^n(s)\|_{\Ll}^2 + \|\z^n(s)\|_{L^2}^2+ \|w^0(s)\|_{\mathbb{L}^4}^4 \Big)ds\right]
\\&\quad+\mathbb{E}\left[\int_0^{t\wedge\tau_N}\|f(s)\|_{\Ll}^2 ds\right] +\mathbb{E}\left[\int_0^{t\wedge\tau_N}\|\s^n(s,u^n(s))\|_{L_Q(\mathbb{L}^2;\Ll)}^2 ds\right] \\
  &\quad +\mathbb{E}\left[\int_0^{t\wedge\tau_N}\int_Z\| H^n(u^n(s),z)\|^2_{\Ll} \la ds\right]
+\mathbb{E}[\|u^n_0\|_{\Ll}^2+\|\z^n_0\|_{L^2}^2]\\
&\leq C\mathbb{E}\left[\int_0^{t\wedge\tau_N}\Big(\|u^n(s)\|_{\Ll}^2 + \|\z^n(s)\|_{L^2}^2+\|w^0(s)\|_{\mathbb{L}^4}^4 \Big) ds\right]\\
&\quad+\mathbb{E}\left[\int_0^{t\wedge\tau_N}\|f(s)\|_{\Ll}^2 ds\right] 
   +K\mathbb{E}\left[\int_0^{t\wedge\tau_N} (1+ \|u^n(s)\|_{\Ll}^2) ds\right]\\
  &\quad+\mathbb{E}[\|u^n_0\|_{\Ll}^2+\|\z^n_0\|_{L^2}^2]\\
   &\leq \mathbb{E}\left[\|u_0\|_{\Ll}^2+\|\z_0\|_{L^2}^2+
  C\int_0^{t\wedge\tau_N}\Big(\|w^0(s)\|_{\mathbb{L}^4}^4+\|f(s)\|_{\Ll}^2 \Big) ds +K(t\wedge \tau_N)\right]\\
&\quad +C\mathbb{E}\left[\int_0^{t\wedge\tau_N} ( \|u^n(s)\|_{\Ll}^2 + \|\z^n(s)\|_{L^2}^2) ds\right].
\end{align*}
Now Gronwall's inequality yields
\begin{align*}
&\mathbb{E}\Big[\|u^n(t\wedge\tau_N)\|_{\Ll}^2+\|\z^n(t\wedge\tau_N)\|_{L^2}^2\Big]+\alpha\mathbb{E}\left[\int_0^{t\wedge\tau_N}\|u^n(s)\|^2_{\mathbb{H}_0^1}ds\right]\\
&\leq C\mathbb{E}\left[\|u_0\|_{\Ll}^2+\|\z_0\|_{L^2}^2+
  C\int_0^{t\wedge\tau_N}\Big(\|w^0(s)\|_{\mathbb{L}^4}^4+\|f(s)\|_{\Ll}^2 \Big) ds +K(t\wedge \tau_N)\right] e^{C(t\wedge \tau_N)}.
\end{align*}
Finally, taking limit as $N\rightarrow\infty$ we have
\begin{align*}
&\mathbb{E}\Big[\|u^n(t)\|_{\Ll}^2+\|\z^n(t)\|_{L^2}^2\Big]+\alpha\mathbb{E}\left[\int_0^{t}\|u^n(s)\|^2_{\mathbb{H}_0^1}ds\right]\\
&\leq C\mathbb{E}\left[\|u_0\|_{\Ll}^2+\|\z_0\|_{L^2}^2+
  C\int_0^{t}\Big(\|w^0(s)\|_{\mathbb{L}^4}^4+\|f(s)\|_{\Ll}^2 \Big) ds +Kt\right] e^{Ct}\\
  &\leq C\mathbb{E}\left[\|u_0\|_{\Ll}^2+\|\z_0\|_{L^2}^2+
  C\int_0^{T}\Big(\|w^0(s)\|_{\mathbb{L}^4}^4+\|f(s)\|_{\Ll}^2 \Big) ds +KT\right] e^{CT},
\end{align*}
which is the desired a-priori estimate \eqref{eq14}.

Proceeding similarly, by taking supremum of equation \eqref{eq11} over time and then taking expectation we achieve
\begin{align} \label{eq17}
&\mathbb{E}\Big[\sup_{0\leq t\leq T\wedge\tau_N} (\|u^n(t)\|_{\Ll}^2+\|\z^n(t)\|_{L^2}^2)\Big]+\alpha \mathbb{E}\left[\int_0^{T\wedge\tau_N}\|u^n(s)\|^2_{\mathbb{H}_0^1}ds\right]\notag\\
&\leq  C\mathbb{E}\left[\int_0^{T\wedge\tau_N}\Big(\|u^n(s)\|_{\Ll}^2 + \|\z^n(s)\|_{L^2}^2+\|w^0(s)\|_{\mathbb{L}^4}^4 \Big) ds\right]\notag\\
&\quad+\mathbb{E}\left[\int_0^{T\wedge\tau_N}\|f(s)\|_{\Ll}^2 ds\right] +K(T\wedge\tau_N)+K\mathbb{E}\left[\int_0^{T\wedge\tau_N}\| u^n(s)\|^2_{\Ll} ds\right]\notag\\
 &\quad+2\mathbb{E}\left[\sup_{0\leq t\leq T\wedge\tau_N}\Big|\int_0^{t}(\s^n(s,u^n(s))dW^n(s),u^n(s))_{\Ll}\Big|\right]+\mathbb{E}[\|u^n_0\|_{\Ll}^2+\|\z^n_0\|_{L^2}^2] \notag\\
 &\quad+2\mathbb{E}\left[\sup_{0\leq t\leq T\wedge\tau_N}\Big|\int_0^{t}\int_Z (H^n(u^n(s-),z),u^n(s-))_{\Ll}\ns \Big|\right].
\end{align}
 Using Burkholder-Davis-Gundy inequality, \eqref{eq.sig.n}, Assumption \ref{Hyp} and Young's inequality we have
\begin{IEEEeqnarray*}{lCl}
\label{eq18}
2\mathbb{E}\left[\sup_{0\leq t\leq T\wedge\tau_N}\Big|\int_0^{t}(\s^n(s,u^n(s))dW^n(s),u^n(s))_{\Ll}\Big|\right]\\
 \leq C_3\mathbb{E}\left[\left(\int_0^{T\wedge\tau_N}\|\s^n(s,u^n(s))\|_{L_Q(\Ll,\Ll)}^2\|u^n(s)\|_{\Ll}^2 ds\right)\right]^{1/2}\\
 \leq C_3\mathbb{E}\left[\left(\int_0^{T\wedge\tau_N}\|\s(s,u^n(s))\|_{L_Q(\Ll,\Ll)}^2\|u^n(s)\|_{\Ll}^2 ds\right)\right]^{1/2}\\
\leq C_3K\mathbb{E}\left[\left(\int_0^{T\wedge\tau_N}(1+\|u^n(s)\|_{\Ll}^2)\|u^n(s)\|_{\Ll}^2 ds\right)\right]^{1/2}\\
\leq C_3K\mathbb{E}\left[\sup_{0\leq s\leq T\wedge\tau_N}\|u^n(s)\|_{\Ll}^2\left(\int_0^{T\wedge\tau_N}(1+\|u^n(s)\|_{\Ll}^2) ds\right)\right]^{1/2} \\
\leq \dfrac{1}{4}\mathbb{E}\left[\sup_{0\leq s\leq T\wedge\tau_N}\|u^n(s)\|_{\Ll}^2\right]+(C_3K)^2\mathbb{E}\left[\int_0^{T\wedge\tau_N}\|u^n(s)\|_{\Ll}^2 ds\right]\\
\quad+(C_3K)^2(T\wedge\tau_N).\IEEEyesnumber
 \end{IEEEeqnarray*}

Similarly, using Burkholder-Davis-Gundy inequality [see Ichikawa \cite{Ichi}], \eqref{eq.sig.n1}, Assumption \ref{Hyp} and Young's inequality  we have
\begin{IEEEeqnarray*}{lCl}
\label{eq19}
2\mathbb{E}\left[\sup_{0\leq t\leq T\wedge\tau_N}\Big|\int_0^{t}\int_Z (H^n(u^n(s-),z),u^n(s-))_{\Ll}\ns\Big|\right]\\
\leq C_4 \mathbb{E}\left[\left(\int_0^{T\wedge\tau_N}\int_Z \Big| (H^n(u^n(s),z),u^n(s))_{\Ll}\Big|^2\la ds\right)\right]^{1/2}\\
\leq C_4 \mathbb{E}\left[\left(\int_0^{T\wedge\tau_N}\int_Z \|H^n(u^n(s),z)\|_{\Ll}^2\| u^n(s)\|_{\Ll}^2 \la ds\right)\right]^{1/2}\\
\leq C_4 \mathbb{E}\left[\left(\int_0^{T\wedge\tau_N}\int_Z \|H(u^n(s),z)\|_{\Ll}^2\| u^n(s)\|_{\Ll}^2 \la ds\right)\right]^{1/2}\\
\leq C_4K \mathbb{E}\left[\sup_{0\leq s\leq T\wedge\tau_N}\|u^n(s)\|_{\Ll}^2\left(\int_0^{T\wedge\tau_N}(1+\|u^n(s)\|_{\Ll}^2) ds\right)\right]^{1/2}\\
\leq \dfrac{1}{4}\mathbb{E}\left[\sup_{0\leq s\leq T\wedge\tau_N}\|u^n(s)\|_{\Ll}^2\right]+(C_4K)^2\mathbb{E}\left[\int_0^{T\wedge\tau_N}\|u^n(s)\|_{\Ll}^2 ds\right]\\
\quad+(C_4K)^2(T\wedge\tau_N).\IEEEyesnumber
\end{IEEEeqnarray*}
Substituting equations \eqref{eq18} and \eqref{eq19} in equation \eqref{eq17} and rearranging the terms we have
\begin{align*}
&\mathbb{E}\Big[\sup_{0\leq t\leq T\wedge\tau_N} (\|u^n(t)\|_{\Ll}^2+\|\z^n(t)\|_{L^2}^2)\Big]+2\alpha\mathbb{E}\left[\int_0^{T\wedge\tau_N}\|u^n(s)\|^2_{\mathbb{H}_0^1}ds\right]\\
&\leq  C^\prime\mathbb{E}\left[\int_0^{T\wedge\tau_N}\sup\limits_{0\leq s\leq t}(\|u^n(s)\|_{\Ll}^2 + \|\z^n(s)\|_{L^2}^2) dt\right]\\
&\quad+C\mathbb{E}\left[\int_0^{T\wedge\tau_N}\|w^0(s)\|_{\mathbb{L}^4}^4 ds\right]+2\mathbb{E}\left[\int_0^{T\wedge\tau_N}\|f(s)\|_{\Ll}^2 ds\right] \\ 
&\quad +C'' (T\wedge\tau_N)+2\mathbb{E}[\|u^n_0\|_{\Ll}^2+\|\z^n_0\|_{L^2}^2]\\
&\leq C'''\mathbb{E}\left[\|u_0\|_{\Ll}^2+\|\z_0\|_{L^2}^2+
\int_0^{T\wedge\tau_N}\left(\|w^0(s)\|_{\mathbb{L}^4}^4+\|f(s)\|_{\Ll}^2\right) ds\right] \\ 
&\quad  +C'' (T\wedge\tau_N) +C^\prime\mathbb{E}\left[\int_0^{T\wedge\tau_N}\sup\limits_{0\leq s\leq t}(\|u^n(s)\|_{\Ll}^2 + \|\z^n(s)\|_{L^2}^2) dt\right],
\end{align*}
where $C^\prime=2[C+(C_3K)^2+(C_4K)^2+K]$, $C''=2[(C_3K)^2+(C_4K)^2+K]$ and $C'''=\max\{C, 2\}$.
Now Gronwall's inequality yields
\begin{align*}
&\mathbb{E}\Big[\sup_{0\leq t\leq T\wedge\tau_N} (\|u^n(t)\|_{\Ll}^2+\|\z^n(t)\|_{L^2}^2)\Big]+2\alpha\mathbb{E}\left[\int_0^{T\wedge\tau_N}\|u^n(s)\|^2_{\mathbb{H}_0^1}ds\right]\\
&\leq C'''\mathbb{E}\Big[\|u_0\|_{\Ll}^2+\|\z_0\|_{L^2}^2 +\int_0^{T\wedge\tau_N}\left(\|w^0(s)\|_{\mathbb{L}^4}^4+\|f(s)\|_{\Ll}^2\right) ds+C'' (T\wedge\tau_N)\Big]
e^{C^\prime(T\wedge\tau_N)}.
\end{align*}
Taking limit as $N\rightarrow\infty$ we infer that
\begin{align*}
&\mathbb{E}\Big[\sup_{0\leq t\leq T} (\|u^n(t)\|_{\Ll}^2+\|\z^n(t)\|_{L^2}^2)\Big]+2\alpha\mathbb{E}\left[\int_0^{T}\|u^n(s)\|^2_{\mathbb{H}_0^1}ds\right]\\
&\leq C'''\mathbb{E}\Big[\|u_0\|_{\Ll}^2+\|\z_0\|_{L^2}^2
+\int_0^{T}\left(\|w^0(s)\|_{\mathbb{L}^4}^4+\|f(s)\|_{\Ll}^2\right) ds+C'' T\Big]
e^{C^\prime T}.
\end{align*}
which is the desired a-priori estimate \eqref{eq15}.

\end{proof}
\subsection{$p^{\textrm{th}}$ Moment Estimate}
Let $2 \leq p<\infty$. We assume that $\s$ and $H$ satisfy Assumption \ref{Hyp} provided in Subsection \ref{Hyp}. In addition, we assume
\begin{enumerate}
\item[H.4]  for all $t\in(0,T)$, there exists a positive constant $M$ such that for all $2 \leq p<\infty$ and $u\in\Ll(\mathcal{O})$
\begin{equation*}
\int_Z\|H(u,z)\|_{\Ll}^p\la\leq M(1+\|u\|_{\Ll}^p).
\end{equation*}
\end{enumerate}
Then we have the following result.
\begin{proposition}\label{moment}
Let $2\leq p<\infty$, $p^*=2p$ and 
\begin{equation}
\left\{\begin{array}{l}
w^0\in L^{p^*}(\Omega; L^{p^*}(0,T;\mathbb{L}^{p^*}(\mathcal{O}))),\;f\in L^p(\Omega;L^p(0,T;\Ll(\mathcal{O}))),\\
\s\in C([0,T]\times\h(\mathcal{O});L_Q(\mathbb{L}^2,\Ll)),\;H\in\mathbb{H}^p_\lambda([0,T]\times Z;\Ll(\mathcal{O})),\\
u_0\in L^p(\Omega;\Ll(\mathcal{O})),\z_0\in L^p(\Omega;L^2(\mathcal{O})).
\end{array}\right.
\end{equation}
If $(u^n(t),\z^n(t))$ denotes the unique strong solution of the system \eqref{eqn5}-\eqref{eqn6}, then we have the following a-priori estimate:
\begin{IEEEeqnarray}{l}
\label{eq16}
\mathbb{E}\left[\sup\limits_{0\leq t\leq T} \Big(\|u^n(t)\|_{\Ll}^p+\|\z^n(t)\|_{L^2}^p \Big)\right]+\alpha p\mathbb{E}[\int_0^T \|u^n(t)\|_{\Ll}^{p-2}\|u^n(t)\|_{\h}^2  dt]\leq C_{1(p)},\nonumber\\
\end{IEEEeqnarray}
where the constant $C_{1(p)}$ depends on the coefficients $\alpha, g,M,\mu$ and the norms $\|f\|_{L^p(\Omega;L^p(0,T;\Ll))},$ $\|w^0\|_{L^{p^*}(\Omega; L^{p^*}(0,T;\mathbb{L}^{p^*}))}, \|u_0\|_{ L^p(\Omega;\Ll)}, \|\z_0\|_{ L^p(\Omega;L^2)}$ and $T$.
\end{proposition}
The Proposition can be proved using the same ideas exploited in Proposition \ref{prop}. However, we provide certain steps  similar to \eqref{B.W0.2} which clarifies regularity of $w^0 \in L^{p^*}(\Omega; L^{p^*}(0,T;\mathbb{L}^{p^*}(\mathcal{O}))) .$ \\
Using the definition of the operator $B(\cdot)$ and Lemma \ref{mon} we have
\begin{IEEEeqnarray*}{rCl} \label{B.W0.p}
(B(u^n(t)),u^n(t))_{\Ll} |u^n(t)|^{p-2}  &\geq -&\int_\mathcal{O}\gamma(x)  |w^0(t)|^2 \,|u^n(t)|^{p-1}dx \\
&\geq & -\dfrac{r}{\epsilon}\Big(\int_{\mathcal{O}} |w^0(t)|^{2p} dx\Big)^{\frac{1}{p}}\Big(\int_{\mathcal{O}}|u^n(t)|^{p}dx\Big)^{\frac{p-1}{p}}\\
 &\geq & -\dfrac{r}{\epsilon} \Big[\|u^n(t)\|^{p-1}_{\mathbb{L}^p}\, \|w^0(t)\|^{2}_{\mathbb{L}^{2p}} \Big]\\
 &\geq & -\dfrac{r}{\epsilon} \Big[ \dfrac{p-1}{p} \|u^n(t)\|^{p}_{\mathbb{L}^p}
 + \dfrac{1}{p} \|w^0(t)\|^{2p}_{\mathbb{L}^{2p}} \Big]
\end{IEEEeqnarray*}
\subsection{Existence Result}
\begin{definition}
A path-wise strong solution $(u, \z)$ is defined on a given filtered probability space $(\Omega,\mathcal{F},\mathcal{F}_t,P)$ as a $\Big(L^\infty(0,T;\mathbb{L}^2(\mathcal{O})) \cap L^2(0,T;\h(\mathcal{O}))\Big) \times C([0,T];L^2(\mathcal{O}))$ valued function which satisfy the stochastic tide equations \eqref{me1}-\eqref{me3} in the weak sense and also the energy inequalities in Proposition \ref{prop}.
\end{definition}
\emph{Proof of Theorem \ref{thmint1} .}
\textbf{Step I : Weak convergent subsequences}

Recall that the Galerkin approximations $(u^n, \z^n)$ satisfy the stochastic differential equations 
\begin{align}\label{SDE}
&du^n(t)+F(u^n(t))dt+g\nabla\z^n(t) dt
=\s^n(t,u^n(t))dW^n(t)\nonumber\\
&\qquad\qquad\qquad\qquad\qquad\qquad\qquad\quad+\int_Z H^n(u^n(t-),z) \n,\\
&d\z^n(t)+Div(hu^n(t)) dt =0,
\end{align}
where $F(u^n)=Au^n+B(u^n)-f$.
Then using the a-priori estimates \eqref{eq14}-\eqref{eq15}, it follows from the Banach-Alaoglu theorem that along a subsequence of $(u^n,\z^n)$, still denoted by $(u^n,\z^n)$, we obtain the following limits:
\begin{IEEEeqnarray}{rCl}
&&u^n\rightarrow u\; \text{weakly star in }L^2(\Omega ;L^\infty(0,T,\mathbb{L}^2(\mathcal{O}))\cap L^2(0,T;\h(\mathcal{O}))),\label{w.u} \\
&&u^n(T) \rightarrow \eta\; \text{weakly in }L^2(\Omega ;\mathbb{L}^2(\mathcal{O})),\label{w.unT} \\
&&\z^n\rightarrow\z\;\text{weakly in }L^2(\Omega ;L^2(0,T;L^2(\mathcal{O}))),\label{w.z} \\
&&F(u^n)\rightarrow F_0\;\text{weakly in }L^2(\Omega ;L^2(0,T;\mathbb{H}^{-1}(\mathcal{O}))), \label{w.F}\\
&&\s^n(\cdot ,u^n)\rightarrow\s_0\;\text{weakly in }L^2(\Omega ;L^2(0,T; L_Q(\Ll,\Ll))), \label{w.sig}\\
&&H^n(u^n,\cdot)\rightarrow H_0\;\text{weakly in }\mathbb{H}_\lambda^2 ([0,T]\times Z;\Ll). \label{w.H}
\end{IEEEeqnarray}
Note that boundedness of $F(u^n)$ follows from the bounds of the operators $A$ and $B$, and convergences of $\sigma^n$ and $H^n$  follow from their linear growth property and the uniform bound of $u^n$ (in $n$) in  $L^2(\Omega; L^2(0,T;\h(\mathcal{O})))$.  

\begin{claim} 
$u$ satisfies the stochastic differential equation
\begin{equation}\label{SDE1}
du(t)+F_0(t)dt+g\nabla\z(t) dt
=\s_0(t)dW(t)
+\int_Z H_0(t,z) \n ,
\end{equation}
weakly in $L^2(\Omega;L^2(0,T;\mathbb{H}^{-1}(\mathcal{O}))).$ 
\end{claim}
\begin{pf}
\end{pf}
Our proof is based on Theorem 4.4 of Brze\'zniak \emph{et al.} \cite{BHZ} and Theorem 3.2 of Suvinthra \emph{et al.} \cite{SSB}. Let us consider the function $\psi(t)\in \ada{H}^1(-\delta,T+\delta)$ with $\psi(0)=1.$ Let $\{e_j\}_{j \in \mathbb{N}}$ be an orthonormal basis in $\h(\mathcal{O}).$ We define $e_j(t)=\psi(t)e_j.$ Applying It\^o formula to the function $(u^n(t),e_j(t))_{\Ll}$ we have
\begin{align} \label{conv.sig.N}
&(u^n(T),e_j(T))_{\Ll}-\int_0^T(u^n(s),\dfrac{d e_j(s)}{ds})_{\Ll}ds+\int_0^T \langle F(u^n(s))+g\nabla\z^n(s),e_j(s)\rangle ds\notag\\
&=(u^n(0),e_j)_{\Ll}+\int_0^T(\s^n(s,u^n(s))dW^n(s),e_j(s))_{\Ll}+\int_0^T \int_Z( H^n(u^n(s-),z),e_j(s))_{\Ll}\N.
\end{align}
 We now fix the integer $j$. Let $\mathcal{P}^1_T$ denote the class of all predictable processes with values in $L^2(\Omega\times [0,T],L_Q(\Ll,\Ll))$ with the inner product defined by $(\xi, \eta)_{\mathcal{P}^1_T} = \E\left[\mathlarger{\int_0^T} \Tr(\xi(t)Q\eta^\star(t)) dt\right]$ for all $\xi, \eta\in\mathcal{P}^1_T$.
 Define $J_1:\mathcal{P}^1_T \rightarrow L^2(\Omega\times [0,T])$ by $J_1(G_1)=\mathlarger{\int_0^t}(G_1(s)dW(s),e_j(s))_{\Ll}.$ Then $J_1$ is linear and continuous. Therefore, in view of weak convergence [as in Chow \cite{Ch}, Chapter 6.7, Proof of Theorem 7.5]  of $\s^n(\cdot,u^n(\cdot))$ to $\sigma_0(\cdot)$, we have $$J_1(\s^n(t,u^n(t))) =\mathlarger{\int_0^T}(\s^n(s,u^n(s)) dW^n(s),e_j(s))_{\Ll} \rightarrow \mathlarger{\int_0^t}(\s_0(s)dW(s),e_j(s))_{\Ll} \quad\mbox{as}\quad n \rightarrow \infty.$$
Again, let $\mathcal{P}^2_T$ be the class of all predictable process with values in $L^2(\Omega\times [0,T],\mathbb{H}^2_\lambda([0,T]\times Z;\Ll(\mathcal{O})))$ with the inner product defined by $(\zeta, \gamma)_{\mathcal{P}^2_T} = \E\left[\mathlarger{\int_0^T\int_Z} (\zeta(t), \gamma(t))_{\Ll} \lambda(dz) dt\right]$. Define $J_2:\mathcal{P}^2_T \rightarrow L^2(\Omega\times [0,T])$ by $J_2(G_2)=\mathlarger{\int_0^t\int_Z}(G_2(s,z),e_j(s))_{\Ll}\N.$ Then $J_2$ is linear and continuous (infact an isometry), hence it is continuous with respect to the weak topologies, [see Theorem 4.4 of Brze\'zniak \emph{et al.} \cite{BHZ} and references therein].
 Again in view of weak convergence of $H^n(u^n(\cdot),\cdot)$ to $H_0$ we have,
 $$J_2(H^n(u^n(t), z))=\mathlarger{\int_0^t\int_Z}(H^n(u^n(s-), z),e_j(s))_{\Ll}\N\rightarrow \mathlarger{\int_0^t}\mathlarger{\int_Z}( H_0(s,z),e_j(s))_{\Ll}\N,$$ as $n \rightarrow \infty.$ 
 Passing to the limit in \eqref{conv.sig.N} using weak convergence we have
 \begin{align*}
 &(\eta,e_j)_{\Ll}\psi(T)-\int_0^T(u(s),\dfrac{d \psi(s)}{ds}e_j)_{\Ll}ds+\int_0^T\langle F_0(s)+g\nabla\z(s),e_j \rangle ds\notag\\
&=(u(0),e_j)_{\Ll}+\int_0^T(\s_0(s)dW(s),\psi(s)e_j)_{\Ll}+\int_0^T \int_Z( H_0(s,z),z),\psi(s)e_j)_{\Ll}\N.
 \end{align*}
 We now choose a subsequence of functions $\{\psi_k\} \subset \ada{H}^1(-\delta,T+\delta) $ such that $\psi_k(0)=1,\,k \in \mathbb{N},$ and letting $k \rightarrow \infty, \,\, \psi_k(s)$ converges to the Heaviside function $H(t-s)$ which is one for $s \leq t$ and zero otherwise.
 Replacing $\psi_k$ instead of $\psi$ as $k \rightarrow \infty$ we have,
  \begin{align*}
 &(u(t),e_j)_\gr{{\Ll}}
 +\int_0^t\langle F_0(s)+g\nabla\z(s),e_j \rangle ds\notag\\&\quad
=(u(0),e_j)_\gr{{\Ll}}+\int_0^t(\s_0(s)dW(s),e_j)_\gr{{\Ll}} +\int_0^t \int_Z( H_0(s,z),z),e_j)_\gr{{\Ll}}\N
 \end{align*}
 with $(u(T),e_j)_\gr{{\Ll}}=(\eta,e_j)_\gr{{\Ll}}\,\,;j=1,\ldots.$
 This gives
 \begin{equation*}
u(t)+\int_0^t(F_0(s)+g\nabla\z(s)) ds
=u(0)+\int_0^t\s_0(s)dW(s)
+\int_0^t\int_Z H_0(s,z) \N ,
\end{equation*} with $u(T)=\eta.$
 This proves the claim.
Hence $u$ satisfies \eqref{SDE1} weakly in $L^2(\Omega;L^2(0,T;\mathbb{H}^{-1}(\mathcal{O})))$ and $\z$ satisfies
 \begin{equation}
d\z(t)+Div(hu(t)) dt =0
\end{equation} 
weakly in $L^2(\Omega;L^2(0,T;L^2(\mathcal{O}))).$ 
 

\noindent
\textbf{Step II : Energy inequality of weak limits }

       Let us take $\tilde{\Lambda}:=2 \Lambda_1 +L$ where $\Lambda_1$ is given in Remark \ref{monotone2}. Let us apply It\^o's lemma to \ada{the process $e^{-\tilde{\Lambda} t}\|\sqrt{h}u(t)\|_{\Ll}^2$} to get
\begin{align}\label{itou}
&e^{-\tilde{\Lambda} T}\|\sqrt{h}u(T)\|_{\Ll}^2 - \|\sqrt{h}u(0)\|_{\Ll}^2 + e^{-\tilde{\Lambda} T}g\|\z(t)\|_{L^2}^2 -g\|\z(0)\|_{L^2}^2 \nonumber\\
&=-\tilde{\Lambda}\int_0^T e^{-\tilde{\Lambda} t}g\|\z(t)\|_{\Ll}^2 dt-\tilde{\Lambda}\int_0^T e^{-\tilde{\Lambda} t}\|\sqrt{h}u(t)\|_{\Ll}^2 dt\nonumber\\
&\quad -\int_0^Te^{-\tilde{\Lambda} t}\langle 2F_0(t),hu(t) \rangle dt+2\int_0^Te^{-\tilde{\Lambda} t}(\s_0(t)dW(t),hu(t))_{\Ll}\nonumber\\
&\quad+\int_0^Te^{-\tilde{\Lambda} t}\|\sqrt{h}\s_0(t)\|_{L_Q(\mathbb{L}^2,\Ll)}^2dt
+2\int_0^T e^{-\tilde{\Lambda} t}\int_Z (H_0(t,z),hu(t-))_{\Ll}\n \nonumber \\
&\quad+\int_0^Te^{-\tilde{\Lambda} t}\int_Z\|\sqrt{h} H_0(t,z)\|^2_{\Ll} N(dt,dz).
\end{align}
Similarly applying It\^o's lemma to the \ada{process $e^{-\tilde{\Lambda} t}\|\sqrt{h}u^n(t)\|_{\Ll}^2$} and then taking expectation we obtain
\begin{align*}
&\mathbb{E}[e^{-\tilde{\Lambda} T}\|\sqrt{h}u^n(T)\|_{\Ll}^2 + e^{-\tilde{\Lambda} T}g\|\z^n(T)\|_{L^2}^2 - \|\sqrt{h}u^n(0)\|_{\Ll}^2 - g\|\z^n(0)\|_{L^2}^2]\\
&=-\mathbb{E}[\int_0^T\tilde{\Lambda} e^{-\tilde{\Lambda} t}g\|\z^n(t)\|_{\Ll}^2 dt]-\mathbb{E}[\int_0^T \tilde{\Lambda} e^{-\tilde{\Lambda} t}\|\sqrt{h}u^n(t)\|_{\Ll}^2 dt]\\ 
&\quad-2\mathbb{E}[\int_0^T e^{-\tilde{\Lambda} t}\langle \gr{F}(u^n(t)),hu^n(t)\rangle dt]
+\mathbb{E}[\int_0^T e^{-\tilde{\Lambda} t}\|\sqrt{h}\s^n(t,u^n(t))\|_{L_{Q}(\mathbb{L}^2,\Ll)}^2dt]\\
&\quad+\mathbb{E}[\int_0^T 2e^{-\tilde{\Lambda} t}(\s^n(t,u^n(t))dW^n(t),hu^n(t))_{\Ll}]\\
&\quad+\mathbb{E}[\int_0^T e^{-\tilde{\Lambda} t}\int_Z\|\sqrt{h} H^n(u^n(t-),z)\|^2_{\Ll} N(dt,dz)]\\
&\quad+\mathbb{E}[\int_0^T 2e^{-\tilde{\Lambda} t}\int_Z (H^n(u^n(t-),z),hu^n(t-))_{\Ll}\n].
\end{align*}
Using the facts that
\begin{equation*}
\int_0^T 2e^{-\tilde{\Lambda} t}(\s^n(t,u^n(t))dW^n(t),hu^n(t))_{\Ll},
\end{equation*}
and
\begin{equation*}
\int_0^T 2e^{-\tilde{\Lambda} t}\int_Z (H^n(u^n(t-),z),hu^n(t-))_{\Ll}\n,
\end{equation*}
are local martingales with zero averages, and the expectation of quadratic variation process and that of Meyer process are same, i.e.,
\begin{align*}
&\mathbb{E}[\int_0^T e^{-\tilde{\Lambda} t}\int_Z\|\sqrt{h} H^n(u^n(t-),z)\|^2_{\Ll} N(dt,dz)]\\
&=\mathbb{E}[\int_0^T e^{-\tilde{\Lambda} t}\int_Z\|\sqrt{h} H^n(u^n(t),z)\|^2_{\Ll} \la dt],
\end{align*}
we have
\begin{IEEEeqnarray*}{lCl}
\mathbb{E}[e^{-\tilde{\Lambda} T}\|\sqrt{h}u^n(T)\|_{\Ll}^2 + e^{-\tilde{\Lambda} T}g\|\z^n(T)\|_{L^2}^2 - \|\sqrt{h}u^n(0)\|_{\Ll}^2 - g\|\z^n(0)\|_{L^2}^2]\\
=-\mathbb{E}[\int_0^T\tilde{\Lambda} e^{-\tilde{\Lambda} t}g\|\z^n(t)\|_{\Ll}^2 dt]-\mathbb{E}[\int_0^T \tilde{\Lambda} e^{-\tilde{\Lambda} t}\|\sqrt{h}u^n(t)\|_{\Ll}^2 dt] \\
\quad-2\mathbb{E}[\int_0^T e^{-\tilde{\Lambda} t} \langle \gr{F}(u^n(t)),hu^n(t) \rangle dt]
+\mathbb{E}[\int_0^T e^{-\tilde{\Lambda} t}\|\sqrt{h}\s^n(t,u^n(t))\|_{{L_{Q}(\mathbb{L}^2,\Ll)}}^2dt] \\
\quad+\mathbb{E}[\int_0^T e^{-\tilde{\Lambda} t}\int_Z\|\sqrt{h} H^n(u^n(t),z)\|^2_{\Ll} \la dt].
\end{IEEEeqnarray*}
Using the lower semi-continuity property of the $\Ll$-norm, strong convergence of the initial data and applying It\^o's lemma to \ada{the process $e^{-\tilde{\Lambda} t}\|\sqrt{h}u(t)\|_{\Ll}^2$} (see \eqref{itou}), we infer that
\begin{IEEEeqnarray*}{lCl}
\label{eq25}
\liminf_n \Big\{-\mathbb{E}[\int_0^T \tilde{\Lambda} e^{-\tilde{\Lambda} t}\|\sqrt{h}u^n(t)\|_{\Ll}^2 dt] 
-2\mathbb{E}[\int_0^T e^{-\tilde{\Lambda} t}\langle \gr{F}(u^n(t)),hu^n(t)\rangle dt]\\
\quad+\mathbb{E}[\int_0^T e^{-\tilde{\Lambda} t}\|\sqrt{h}\s^n(t,u^n(t))\|_{L_{Q}(\mathbb{L}^2,\Ll)}^2dt] \\
\quad+\mathbb{E}[\int_0^T e^{-\tilde{\Lambda} t}\int_Z\|\sqrt{h} H^n(u^n(t),z)\|^2_{\Ll} \la dt]\Big\}\\
=\liminf_n \Big\{\mathbb{E}[e^{-\tilde{\Lambda} T}\|\sqrt{h}u^n(T)\|_{\Ll}^2 + e^{-\tilde{\Lambda} T}g\|\z^n(T)\|_{L^2}^2 - \|\sqrt{h}u^n(0)\|_{\Ll}^2 \\ \quad- g\|\z^n(0)\|_{L^2}^2] +\mathbb{E}[\int_0^T\tilde{\Lambda} e^{-\tilde{\Lambda} t}g\|\z^n(t)\|_{\Ll}^2 dt]\gr{\Big\}}\\
\geq \mathbb{E}[e^{-\tilde{\Lambda} T}\|\sqrt{h}u(T)\|_{\Ll}^2 + e^{-\tilde{\Lambda} T}g\|\z(T)\|_{L^2}^2 - \|\sqrt{h}u(0)\|_{\Ll}^2 - g\|\z(0)\|_{L^2}^2]\\ \quad+\mathbb{E}[\int_0^T\tilde{\Lambda} e^{-\tilde{\Lambda} t}g\|\z(t)\|_{\Ll}^2 dt]\\
=-\mathbb{E}[\int_0^T \tilde{\Lambda} e^{-\tilde{\Lambda} t}\|\sqrt{h}u(t)\|_{\Ll}^2 dt]-2\mathbb{E}[\int_0^T e^{-\tilde{\Lambda} t}\langle F_0(t),hu(t)\rangle dt]\\
\quad+\mathbb{E}[\int_0^T e^{-\tilde{\Lambda} t}\|\sqrt{h}\s_0(t)\|_{L_{Q}(\mathbb{L}^2,\Ll)}^2dt] 
+\mathbb{E}[\int_0^T e^{-\tilde{\Lambda} t}\int_Z\|\sqrt{h} H_0(t,z)\|^2_{\Ll} \la dt].\\\IEEEyesnumber
\end{IEEEeqnarray*}
\textbf{Step III : Consequences of weak convergence}
\gr{\begin{claim} For any $v \in L^2(\Omega ;L^{\infty}(0,T;\mathbb{L}^2(\mathcal{O}))\cap L^2(0,T;\h(\mathcal{O}))),$ using the weak convergences \eqref{w.u}-\eqref{w.H}, we have the following:
\begin{align}
(i)\,&\,\mathbb{E}[\int_0^T e^{-\tilde{\Lambda} t}\langle F(v(t)),hu^n(t)-hv(t)\rangle dt] \notag \\
&\rightarrow \,\mathbb{E}[\int_0^T e^{-\tilde{\Lambda} t}\langle F(v(t)),hu(t)-hv(t)\rangle dt] \quad \mbox{as} \quad n \rightarrow \infty.\label{weak.F}\\
(ii)\,&\,\mathbb{E}[\int_0^T e^{-\tilde{\Lambda} t}(\s^n(t,u^n(t)),\s^n(t,v(t)))_{L_{Q}(\mathbb{L}^2,\Ll)}dt] \notag\\
 &\rightarrow \mathbb{E}[\int_0^T e^{-\tilde{\Lambda} t}(\s_0(t),\s(t,v(t)))_{L_{Q}(\mathbb{L}^2,\Ll)}dt] \quad \mbox{as} \quad n \rightarrow \infty.\label{weak.sig}\\
(iii)\, &\,\mathbb{E}[\int_0^T e^{-\tilde{\Lambda} t}\int_Z (H^n(u^n(t),z),H^n(v(t),z))_{\Ll} \la dt] \notag \\
& \rightarrow \mathbb{E}[\int_0^T e^{-\tilde{\Lambda} t}\int_Z\ (H_0(t,z),H(v(t),z))_{\Ll} \la dt] \quad \mbox{as} \quad n \rightarrow \infty.\label{weak.H}\\
(iv)\, &\,\mathbb{E}[\int_0^T  e^{-\tilde{\Lambda} t}(u^n(t),v(t))_{\Ll} dt] 
 \rightarrow \mathbb{E}[\int_0^T  e^{-\tilde{\Lambda} t}(u(t),v(t))_{\Ll} dt] \notag\\
 &\quad \mbox{as} \quad n \rightarrow \infty.\label{weak.u}\\
(v)\,&\,\mathbb{E}[\int_0^T e^{-\tilde{\Lambda} t}\|\s^n(t,v(t))\|_{L_{Q}(\mathbb{L}^2,\Ll)}^2dt] 
\rightarrow \mathbb{E}[\int_0^T e^{-\tilde{\Lambda} t}\|\s(t,v(t))\|_{L_{Q}(\mathbb{L}^2,\Ll)}^2dt]\notag\\ &\quad \mbox{as} \quad n \rightarrow \infty.\label{weak.nor.sig}\\(vi)
\,&\,\mathbb{E}[\int_0^T e^{-\tilde{\Lambda} t}\int_Z\|H^n(v(t),z)\|^2_{\Ll} \la dt] \notag \\ &\rightarrow \mathbb{E}[\int_0^T e^{-\tilde{\Lambda} t}\int_Z\|H(v(t),z)\|^2_{\Ll} \la dt] \quad \mbox{as} \quad n \rightarrow \infty. \label{weak.norm.H}
\end{align}
\end{claim}}

\begin{pf}
\end{pf}
\begin{itemize}
\item[{\textit{(i)}}]
Note, as $v \in L^2(\Omega ;L^{\infty}(0,T;\mathbb{L}^2(\mathcal{O}))\cap L^2(0,T;\h(\mathcal{O}))),$  $F(v(t))\in  L^2(\Omega; L^2(0,T; \mathbb{H}^{-1}(\mathcal{O}))).$ 
Since $u^n\rightarrow u$ weakly in $L^2(\Omega ; L^2(0,T;\h(\mathcal{O})))$, we infer that
\begin{align*} 
&|\mathbb{E}[\int_0^T e^{-\tilde{\Lambda} t}\Big(\langle F(v(t)),hu^n(t)\rangle  -\langle F(v(t)),hu(t)\rangle \Big) dt]|\\
&\leq  \mathbb{E}[\int_0^T e^{-\tilde{\Lambda} t}|\langle F(v(t)),hu^n(t)-hu(t) \rangle| dt]\rightarrow 0 \quad \mbox{as}\quad n \rightarrow \infty.
\end{align*}
This proves {\textit{(i)}}.
\item[{\textit{(ii)}}] 
To prove \eqref{weak.sig}, we first note
\begin{align} \label{w.ii}
&|\mathbb{E}[\int_0^T e^{-\tilde{\Lambda} t}(\s^n(t,u^n(t)),\s^n(t,v(t)))_{L_{Q}(\mathbb{L}^2,\Ll)}dt]\notag\\&\quad-\mathbb{E}[\int_0^T e^{-\tilde{\Lambda} t}(\s_0(t),\s(t,v(t)))_{L_{Q}(\mathbb{L}^2,\Ll)}dt]|\nonumber\\
&=|\mathbb{E}[\int_0^T e^{-\tilde{\Lambda} t}(\s^n(t,u^n(t))-\s_0(t),\s(t,v(t)))_{L_{Q}(\mathbb{L}^2,\Ll)}dt]\nonumber\\&\quad +\mathbb{E}[\int_0^T e^{-\tilde{\Lambda} t}(\s^n(t,u^n(t)),\s^n(t,v(t))-\s(t,v(t)))_{L_{Q}(\mathbb{L}^2,\Ll)}dt]|\nonumber\\
& \leq |\mathbb{E}[\int_0^T e^{-\tilde{\Lambda} t}(\s^n(t,u^n(t))-\s_0(t),\s(t,v(t)))_{L_{Q}(\mathbb{L}^2,\Ll)}dt]| \nonumber\\&\quad +\mathbb{E}[\int_0^T e^{-\tilde{\Lambda} t}\|(\s^n(t,u^n(t))\|_{L_{Q}(\mathbb{L}^2,\Ll)}\|\s^n(t,v(t))-\s(t,v(t))\|_{L_{Q}(\mathbb{L}^2,\Ll)}dt]\nonumber\\
& \leq |\mathbb{E}[\int_0^T e^{-\tilde{\Lambda} t}(\s^n(t,u^n(t))-\s_0(t),\s(t,v(t)))_{L_{Q}(\mathbb{L}^2,\Ll)}dt]|\nonumber \\&\quad +\Big(\mathbb{E}[\int_0^T e^{-2\tilde{\Lambda} t}\|(\s^n(t,u^n(t))\|^2_{L_{Q}(\mathbb{L}^2,\Ll)} dt] \Big)^{\frac{1}{2}}\notag\\& \quad\quad \times \Big(\mathbb{E}[\int_0^T \|\s^n(t,v(t))-\s(t,v(t))\|^2_{L_{Q}(\mathbb{L}^2,\Ll)}dt]\Big)^{\frac{1}{2}}\nonumber\\
& \leq |\mathbb{E}[\int_0^T e^{-\tilde{\Lambda} t}(\s^n(t,u^n(t))-\s_0(t),\s(t,v(t)))_{L_{Q}(\mathbb{L}^2,\Ll)}dt]| \nonumber\\&\quad 
+c\Big(\mathbb{E}[\int_0^T \|\s^n(t,v(t))-\s(t,v(t))\|^2_{L_{Q}(\mathbb{L}^2,\Ll)}dt]\Big)^{\frac{1}{2}},
\end{align}
where by using Assumption \ref{hypo.noi}, and Proposition \ref{prop} we observe that
\begin{align*}
c:&=\sup_{n \in \mathbb{N}}\Big(\mathbb{E}[\int_0^T e^{-2\tilde{\Lambda} t}\|(\s^n(t,u^n(t))\|^2_{L_{Q}(\mathbb{L}^2,\Ll)} dt] \Big)\\& 
\leq \sup_{n \in \mathbb{N}}\Big(\mathbb{E}[\int_0^T e^{-2\tilde{\Lambda} t}\|(\s(t,u^n(t))\|^2_{L_{Q}(\mathbb{L}^2,\Ll)} dt] \Big)
\\& \leq K \sup_{n \in \mathbb{N}}\Big(\mathbb{E}[\int_0^T e^{-2\tilde{\Lambda} t}(1+\|u^n(t))\|^2_{\Ll}) dt] \Big) \\ & \leq C_T + C_T\sup_{n \in \mathbb{N}}\mathbb{E}\Big(\sup_{t \in [0,T]}\|u^n(t))\|^2_{\Ll} \Big)<\infty.
\end{align*}
Thus using \eqref{w.sig}, the first term on the right hand side of \eqref{w.ii} tends to zero as $n\to\infty$. Since $P_n \s(\cdot,v(\cdot))= \s^n(\cdot,v(\cdot)),$ we now prove $\s^n(\cdot,v(\cdot)) \rightarrow \s(\cdot,v(\cdot))$ strongly in $L^2(\Omega ;L^2(0,T;L_Q(\mathbb{L}^2,\Ll))).$ \\
Let us fix $t \in [0,T]$ and $\omega \in \Omega$ and define the double sequence $a_{n,k}:=\|(\s^n(t,v(t))-\s(t,v(t)))Q^{\frac{1}{2}}\psi_k\|^2_{\Ll}$, where $\{\psi_k\}_{k \geq 1}$ is an orthonormal basis in $\mathbb{L}^2.$ Define $b_k:=4\|\s(t,v(t)) Q^{\frac{1}{2}}\psi_k\|^2_{\Ll}$. Note, by Minkowski inequality and \eqref{eq.sig.n}, $a_{n,k} \leq b_k$ for every $n$ and $k$. Moreover, $\sum_{k=1}^{\infty}b_k=4\sum_{k=1}^{\infty}\|\s(t,v(t)) Q^{\frac{1}{2}}\psi_k\|^2_{\Ll}=4 \|\s(t,v(t))\|^2_{L_{Q}(\mathbb{L}^2,\Ll)}< \infty.$ Now using Lemma \ref{Pn.lem} for every $k$, we achieve 
\begin{align*}
&\lim_{n \rightarrow \infty} a_n(k)=\lim_{n \rightarrow \infty} \|(\s^n(t,v(t))-\s(t,v(t)))Q^{\frac{1}{2}}\psi_k\|^2_{\Ll}\\&=\lim_{n \rightarrow \infty} \|\s^n(t,v(t)) Q^{\frac{1}{2}}\psi_k-\s(t,v(t)) Q^{\frac{1}{2}}\psi_k\|^2_{\Ll}=0.
\end{align*}
 Hence by Lebesgue Dominated Convergence Theorem (for double sequence), $$\lim_{n \rightarrow \infty}\sum_{k=1}^{\infty} a_{n,k}=0.$$ Thus for almost all $t \in [0,T],\,\,P-$a.e. $\omega \in \Omega$ we obtain
\begin{align} \label{conv.sig.Q}
\lim_{n \rightarrow \infty}\|\s^n(t,v(t))-\s(t,v(t))\|^2_{L_{Q}(\mathbb{L}^2,\Ll)}=0.
\end{align}  
Furthermore, $ \|\s^n(t,v(t))-\s(t,v(t))\|^2_{L_{Q}(\mathbb{L}^2,\Ll)} \leq 4 \|\s(t,v(t))\|^2_{L_{Q}(\mathbb{L}^2,\Ll)}$ for almost all $t \in [0,T],\,\,P-$a.e. $\omega \in \Omega$. 
Thus applying Lebesgue Dominated Convergence Theorem again we achieve,
\begin{align} \label{fin.conv.si}
\lim_{n\to\infty}\mathbb{E} \Big[ \int_0^T \|\s^n(t,v(t))-\s(t,v(t))\|^2_{L_{Q}(\mathbb{L}^2,\Ll)} dt\Big] = 0,
\end{align}
 proving that $\s^n(\cdot,v(\cdot)) \rightarrow \s(\cdot,v(\cdot))$ strongly in $L^2(\Omega ;L^2(0,T;L_Q(\mathbb{L}^2,\Ll))).$ 
Hence, second term on the right hand side of \eqref{w.ii} goes to zero as $n\to\infty$. Thus we have \eqref{weak.sig}.

\item[{\textit{(iii)}}]
Let us consider
\begin{align} \label{w.iv}
&|\mathbb{E}[\int_0^T e^{-\tilde{\Lambda} t}\int_Z (H^n(u^n(t),z),H^n(v(t),z))_{\Ll} \la dt]\notag\\&\quad-\mathbb{E}[\int_0^T e^{-\tilde{\Lambda} t}\int_Z\ (H_0(t,z),H(v(t),z))_{\Ll} \la dt]|\nonumber\\
&=|\mathbb{E}[\int_0^T e^{-\tilde{\Lambda} t} \int_Z (H^n(u^n(t),z)-H_0(t,z),H(v(t),z))_{\Ll} \la dt]\nonumber\\&\quad +\mathbb{E}[\int_0^T e^{-\tilde{\Lambda} t} \int_Z (H^n(u^n(t),z),H^n(v(t),z)-H(v(t),z))_{\Ll} \la dt]|\nonumber\\
& \leq |\mathbb{E}[\int_0^T e^{-\tilde{\Lambda} t} \int_Z (H^n(u^n(t),z)-H_0(t,z),H(v(t),z))_{\Ll} \la dt]| \nonumber\\&\quad +\mathbb{E}[\int_0^T e^{-\tilde{\Lambda} t} \int_Z \|(H^n(u^n(t),z)\|_{\Ll}\|H^n(v(t),z)-H(v(t),z)\|_{\Ll} \la dt]\nonumber\\
& \leq |\mathbb{E}[\int_0^T e^{-\tilde{\Lambda} t} \int_Z (H^n(u^n(t),z)-H_0(t,z),H(t,v(t)))_{\Ll} \la dt]|\nonumber \\&\quad +\Big(\mathbb{E}[\int_0^T e^{-2\tilde{\Lambda} t} \int_Z \|(H^n(t,u^n(t))\|^2_{\Ll} \la dt] \Big)^{\frac{1}{2}} \notag\\
&\quad \quad \times \Big(\mathbb{E}[\int_0^T \int_Z \|H^n(v(t),z)-H(v(t),z)\|^2_{\Ll} \la dt]\Big)^{\frac{1}{2}}\nonumber\\
& \leq |\mathbb{E}[\int_0^T e^{-\tilde{\Lambda} t}\int_Z (H^n(u^n(t),z)-H_0(t,z),H(v(t),z))_{\Ll} \la dt]| \nonumber\\&\quad 
+c\Big(\mathbb{E}[\int_0^T \int_Z \|H^n(v(t),z)-H(v(t),z)\|^2_{\Ll} \la dt]\Big)^{\frac{1}{2}}
\end{align}
where by using Assumption \ref{hypo.noi}, Proposition \ref{prop} we observe that
\begin{align*}
c:&=\sup_{n \in \mathbb{N}}\Big(\mathbb{E}[\int_0^T e^{-2\tilde{\Lambda} t} \int_Z \|(H^n(u^n(t),z)\|^2_{\Ll} \la dt] \Big)\\& 
\leq \sup_{n \in \mathbb{N}}\Big(\mathbb{E}[\int_0^T e^{-2\tilde{\Lambda} t} \int_Z \|(H(u^n(t),z)\|^2_{\Ll} \la dt] \Big)\\& 
\leq K \sup_{n \in \mathbb{N}}\Big(\mathbb{E}[\int_0^T e^{-2\tilde{\Lambda} t}(1+\|u^n(t))\|^2_{\Ll}) dt] \Big) \\ & \leq C_T + C_T\sup_{n \in \mathbb{N}}\mathbb{E}\Big(\sup_{t \in [0,T]}\|u^n(t))\|^2_{\Ll} \Big)<\infty.
\end{align*}
Thus using \eqref{w.H} first term on the right hand side of \eqref{w.iv} tends to zero. Recalling $P_n H(v(\cdot),\cdot)= H^n(v(\cdot),\cdot),$ 
 using Lemma \ref{Pn.lem} and \eqref{eq.sig.n1}, we have for almost all $t \in [0,T]$ and $P-$a.e. $\omega \in \Omega,$ all $z \in Z,$
\begin{align} \label{conv.H}
&\lim_{n \rightarrow \infty} \|H^n(v(t),z)-H(v(t),z)\|_{\Ll}=0\quad \mbox{and} \notag \\
 &\|H^n(v(t),z)-H(v(t),z)\|_{\Ll} \leq 2 \|H(v(t),z)\|_{\Ll}.
\end{align}
Hence by applying Lebesgue Dominated Convergence Theorem, we have
\begin{align} \label{fin.conv.H}
\lim_{n\to\infty}\mathbb{E} \Big[\int_0^T \int_Z \|H^n(v(t),z)-H(v(t),z)\|_{\Ll}^2 \la dt \Big]  = 0,
\end{align}
proving that $H^n(v(\cdot),\cdot) \rightarrow H(v(\cdot),\cdot)$ strongly in $\mathbb{H}_\lambda^2 ([0,T]\times Z;\Ll),$ which also implies that the second term on the right hand side of \eqref{w.iv} goes to zero, ensuring \eqref{weak.H}. 
\item[{\textit{(iv)}}] Using \eqref{w.u} we directly have \eqref{weak.u}.
\item[{\textit{(v)}}] It directly follows from \eqref{fin.conv.si}.
\item[{\textit{(vi)}}] It directly follows from \eqref{fin.conv.H}.
\end{itemize}

\noindent
\textbf{Step IV : Consequences of monotonicity argument }
\ada{Using the monotonicity property of $F(\cdot)$ (see Remark \ref{monotone2} and \eqref{monotone1}) and Assumption \ref{Hyp}}, we have for $v\in \gr{L^2(\Omega; L^\infty(0,T;\mathbb{L}^2(\mathcal{O}))}\cap L^2(0,T;\h(\mathcal{O})))$
\gr{\begin{IEEEeqnarray*}{lCl}
2\mathbb{E}[\int_0^T e^{-\tilde{\Lambda} t}\langle F(u^n(t))-F(v(t)),hu^n(t)-hv(t)\rangle dt]\\
-\mathbb{E}[\int_0^T e^{-\tilde{\Lambda} t}\|\s^n(t,u^n(t))-\s^n(t,v(t))\|_{L_Q(\mathbb{L}^2,\Ll)}^2dt] \\
-\mathbb{E}[\int_0^T e^{-\tilde{\Lambda} t}\int_Z\|H^n(u^n(t),z)-H^n(v(t),z)\|^2_{\Ll} \la dt]\\
+\mathbb{E}[\int_0^T \tilde{\Lambda} e^{-\tilde{\Lambda} t}\|u^n(t)-v(t)\|_{\Ll}^2 dt] \ \geq 0.\IEEEyesnumber
\end{IEEEeqnarray*}
Splitting each of the inner products and norms and then on rearranging we obtain,
\begin{IEEEeqnarray}{lCl}
&-&2\mathbb{E}[\int_0^T e^{-\tilde{\Lambda} t}\langle F(u^n(t)),hu^n(t)-hv(t)\rangle dt]
+\mathbb{E}[\int_0^T e^{-\tilde{\Lambda} t}\|\s^n(t,u^n(t))\|_{L_Q(\mathbb{L}^2,\Ll)}^2dt] \notag \\
&+&\mathbb{E}[\int_0^T e^{-\tilde{\Lambda} t}\int_Z\|H^n(u^n(t),z)\|^2_{\Ll} \la dt]
-\mathbb{E}[\int_0^T \tilde{\Lambda} e^{-\tilde{\Lambda} t}\|u^n(t)\|_{\Ll}^2 dt] \notag \\
&\leq & -2\mathbb{E}[\int_0^T e^{-\tilde{\Lambda} t}\langle F(v(t)),hu^n(t)-hv(t)\rangle dt]
-\mathbb{E}[\int_0^T e^{-\tilde{\Lambda} t}\|\s^n(t,v(t))\|_{L_Q(\mathbb{L}^2,\Ll)}^2dt] \notag \\
&&+2\mathbb{E}[\int_0^T e^{-\tilde{\Lambda} t}(\s^n(u^n(t)),\s^n(v(t)))_{L_Q(\mathbb{L}^2,\Ll)}dt] \notag \\
&&
-\mathbb{E}[\int_0^T e^{-\tilde{\Lambda} t}\int_Z\|H^n(v(t),z)\|^2_{\Ll} \la dt] \notag \\
&&+2\mathbb{E}[\int_0^T e^{-\tilde{\Lambda} t}\int_Z\ (H^n(u^n(t),z),H^n(v(t),z))_{\Ll} \la dt] \notag \\
&&+\mathbb{E}[\int_0^T \tilde{\Lambda} e^{-\tilde{\Lambda} t}\|v(t)\|_{\Ll}^2 dt]
-2\mathbb{E}[\int_0^T \tilde{\Lambda} e^{-\tilde{\Lambda} t}(u^n(t),v(t))_{\Ll} dt].
\end{IEEEeqnarray}


In view of \eqref{weak.F}, \eqref{weak.sig}-\eqref{weak.norm.H}, and taking limit in $n$ and using \eqref{eq25} we conclude that
\begin{align*}
&-2\mathbb{E}[\int_0^T e^{-\tilde{\Lambda} t}\langle F_0(t),hu(t)-hv(t)\rangle dt]
+\mathbb{E}[\int_0^T e^{-\tilde{\Lambda} t}\| \s_0(t)\|_{L_Q(\mathbb{L}^2,\Ll)}^2dt] \\
&+\mathbb{E}[\int_0^T e^{-\tilde{\Lambda} t}\int_Z\| H_0(t,z)\|^2_{\Ll} \la dt]
-\mathbb{E}[\int_0^T \tilde{\Lambda} e^{-\tilde{\Lambda} t}\| u(t)\|_{\Ll}^2 dt]\\
&\leq -2\mathbb{E}[\int_0^T e^{-\tilde{\Lambda} t}\langle F(v(t)),hu(t)-hv(t)\rangle dt]-\mathbb{E}[\int_0^T e^{-\tilde{\Lambda} t}\| \s(t,v(t))\|_{L_Q(\mathbb{L}^2,\Ll)}^2dt] \\
&\quad+2\mathbb{E}[\int_0^T e^{-\tilde{\Lambda} t}( \s_0(t), \s(v(t)))_{L_Q(\mathbb{L}^2,\Ll)}dt]-\mathbb{E}[\int_0^T e^{-\tilde{\Lambda} t}\int_Z\| H(v(t),z)\|^2_{\Ll} \la dt]\\
&\quad+2\mathbb{E}[\int_0^T e^{-\tilde{\Lambda} t}\int_Z\ ( H_0(t,z), H(v(t),z))_{\Ll} \la dt]\\
&\quad+\mathbb{E}[\int_0^T \tilde{\Lambda} e^{-\tilde{\Lambda} t}\|v(t)\|_{\Ll}^2 dt]
-2\mathbb{E}[\int_0^T \tilde{\Lambda} e^{-\tilde{\Lambda} t}(u(t),v(t))_{\Ll} dt].
\end{align*}}
Rearranging the terms
\begin{IEEEeqnarray*}{lCl}
\gr{-2\mathbb{E}[\int_0^T e^{-\tilde{\Lambda} t}\langle F_0(t)-F(v(t)),hu(t)-hv(t)\rangle dt]}\\
+\mathbb{E}[\int_0^T e^{-\tilde{\Lambda} t}\|\s_0(t)-\s(t,v(t))\|_{L_Q(\mathbb{L}^2,\Ll)}^2dt] \\
+\mathbb{E}[\int_0^T e^{-\tilde{\Lambda} t}\int_Z\|H_0(t,z)-H(v(t),z)\|^2_{\Ll} \la dt]\\
-\mathbb{E}[\int_0^T \tilde{\Lambda} e^{-\tilde{\Lambda} t}\|u(t)-v(t)\|_{\Ll}^2 dt]\ \leq 0. \yesnumber\label{limiteqn}
\end{IEEEeqnarray*}
\gr{The above estimate holds for any $v\in L^2(\Omega; L^\infty(0,T;\mathbb{L}^2(\mathcal{O}))\cap L^2(0,T;\h(\mathcal{O})))$. Choosing $v(\cdot)=u(\cdot)$, we can immediately observe  $\s_0(\cdot)=\s(\cdot,u(\cdot))$ and $H_0(\cdot,\cdot)=H(u(\cdot),\cdot)$.}
Now we take $v=u+\lambda w$ in \eqref{limiteqn}, where $\lambda>0$ and $w$ is an adapted process in $L^2(\Omega; L^\infty(0,T;\mathbb{L}^2(\mathcal{O}))\cap L^2(0,T;\h(\mathcal{O})))$. Then we have 
\gr{\begin{align*}
&\lambda\mathbb{E}[\int_0^T e^{-\tilde{\Lambda} t}\langle F(u(t)+\lambda w(t)),hw(t)\rangle dt]+ \lambda^2 \tilde{\Lambda} \mathbb{E}[\int_0^T e^{-\tilde{\Lambda} t}\|w(t)\|_{\Ll}^2 dt]\\& \geq \lambda\mathbb{E}[\int_0^T e^{-\tilde{\Lambda} t}\langle F_0(t),hw(t)\rangle dt].
\end{align*}}
Dividing by $\lambda$ on both sides of the inequality above and letting $\lambda$ to go to 0, we have by the hemicontinuity of $F(\cdot)$
\begin{equation*}
\mathbb{E}[\int_0^T e^{-\tilde{\Lambda} t}\langle F(u(t))-F_0(t),hw(t)\rangle dt]\geq 0.
\end{equation*}
Since $w$ is arbitrary and $h$ is  a positive, bounded, continuously differentiable function,  $F_0(t)=F(u(t))$. This proves the existence of a strong solution.\\\\
\noindent
\textbf{Step V : Uniqueness}:\\
If the pairs $(u,\z)$ and $(v,\tilde{z})$ are two solutions to \eqref{me1}-\eqref{me3}, then $w(t):=u(t)-v(t)$ and $\zeta(t):=\z(t)-\tz(t)$ solves the stochastic integral equation
\begin{align} \label{uniq.u}
&dw(t)+Aw(t)dt+g\nabla\zeta(t)dt \notag\\
&=[B(v(t))-B(u(t))]dt+(\s(t,u(t))-\s(t,v(t)))dW(t)\notag \\
&\quad+\int_Z (H(u(t-),z)-H(v(t-),z))\n .
\end{align}
and 
\begin{equation} \label{uniq.z}
d\zeta(t)+Div(hw(t)) dt=0.
\end{equation}
Define the stopping time as:
\begin{equation}
\tau_N=\inf\{t \geq 0:\|w(t)\|_{\Ll}^2+\|\zeta(t)\|_{L^2}^2+\int_0^t\|w(s)\|^2_{\mathbb{H}_0^1}ds>N\}.
\end{equation}
Applying It\^o's lemma to the \ada{process $\|w(t)\|_{\Ll}^2$} 
\begin{align*}
&\|w(\L)\|_{\Ll}^2+2\alpha\int_0^{\L}\|w(s)\|_{\h}^2 ds+2g\int_0^{\L}\gr{\langle \nabla\zeta(s),w(s)\rangle ds}\\
&=\|w(0)\|_{\Ll}^2+2\int_0^{\L}[(B(v(s))-B(u(s)),u(s)-v(s))_{\Ll}]ds\\
&\quad+2\int_0^{\L}(\s(s,u(s))-\s(s,v(s))dW(s),w(s))_{\Ll} \\
&\quad+\int_0^{\L}\|\s(s,u(s))-\s(s,v(s))\|_{L_Q(\mathbb{L}^2,\Ll)}^2 ds\\
&\quad+\int_0^{\L}\int_Z \|H(u(s-),z)-H(v(s-),z)\|^2_{\Ll} N(ds,dz) \\
&\quad+2\int_0^{\L}\int_Z \left((H(u(s-),z)-H(v(s-),z)),w(s-)\right)_{\Ll}\ns .
\end{align*}
Using the result from Lemma \ref{mon}, inequality \eqref{eq26} \gr{and $w(0)=0,$}
\begin{align}
\label{eq27}
&\|w(\L)\|_{\Ll}^2+2\alpha\int_0^{\L}\|w(s)\|_{\h}^2 ds\notag\\
&\leq \dfrac{2g^2}{\alpha}\int_0^{\L}\|\zeta(s)\|_{L^2}^2 ds+\dfrac{\alpha}{2}\int_0^{\L}\|w(s)\|_{\h}^2 ds\notag\\
&\quad+2\int_0^{\L}(\s(s,u(s))-\s(s,v(s)) dW(s),w(s))_{\Ll}\notag\\&\quad+\int_0^{\L}\|\s(s,u(s))-\s(s,v(s))\|_{L_Q(\mathbb{L}^2,\Ll)}^2 ds\notag\\
&\quad+\int_0^{\L}\int_Z \|H(u(s-),z)-H(v(s-),z)\|^2_{\Ll} N(ds,dz)\notag \\
&\quad+2\int_0^{\L}\int_Z \left((H(u(s-),z)-H(v(s-),z)),w(s-)\right)_{\Ll}\ns .
\end{align}

Taking \eqref{uniq.z} inner product with $\zeta(t):=\z(t)-\tz(t)$ \gr{and using $\zeta(0)=0$}, we have as in equation \eqref{eq9}
\begin{align}
\label{eq28}
\|\zeta(\L)\|_{L^2}^2
&\leq M\,\int_0^{\L}\|w(s)\|_{\Ll}^2 ds + \left(\dfrac{2\mu^2}{\alpha}+M\right)\int_0^{\L}\|\zeta(s)\|_{L^2}^2 ds \notag\\&\quad+ \dfrac{\alpha}{2}\int_0^{\L}\|w(s)\|_{\h}^2 ds.
\end{align}
Let
$C=\dfrac{2g^2}{\alpha}+\dfrac{2\mu^2}{\alpha}+M.$
Adding equations \eqref{eq27} and \eqref{eq28}, then taking expectation \gr{and using It\^{o}-L\'evy isometry we achieve}
\begin{align*}
&\mathbb{E}[\|w(t\wedge\tau_N)\|_{\Ll}^2+\|\zeta(t\wedge\tau_N)\|_{L^2}^2]+\mathbb{E}[\int_0^{t\wedge\tau_N} \alpha\|w(s)\|_{\h}^2 ds]\\
&\leq C\mathbb{E}[\int_0^{t\wedge\tau_N}(\|w(s)\|_{\Ll}^2+\|\zeta(s)\|_{L^2}^2) ds]\\
&\quad+\mathbb{E}[\int_0^{t\wedge\tau_N} \|\s(s,u(s))-\s(s,v(s))\|_{L_Q(\mathbb{L}^2,\Ll)}^2 ds]\\
&\quad+\mathbb{E}[\int_0^{t\wedge\tau_N} \int_Z \|H(u(s),z)-H(v(s),z)\|^2_{\Ll} \la ds].
\end{align*}
Applications of Assumption \ref{Hyp} and Gronwall's inequality yield
\begin{align*}
\mathbb{E}[\|w(t\wedge\tau_N)\|_{\Ll}^2+\|\zeta(t\wedge\tau_N)\|_{L^2}^2]+\alpha\mathbb{E}[\int_0^{t\wedge\tau_N} \|w(s)\|_{\h}^2 ds]\leq 0.
\end{align*}
As $N\to\infty, t\wedge\tau_N \to t$ a.s., and hence the uniqueness of pathwise strong solution follows.

\section{Stochastic Optimal Control}\label{stochastic}
\subsection{Preliminaries}
In this Subsection we provide some  definitions and known results borrowed from M\'{e}tivier \cite{metivier} and Aldous \cite{ald}.
\begin{definition}
Let $(\mathbb{S},\rho)$ be a separable and complete metric space. Let $u\in D([0,T];\mathbb{S})$ and let $\delta>0$ be given. A modulus of $u$ is defined by
\begin{equation}
w_{[0,T],\mathbb{S}}(u,\delta):=\inf_{\Pi_\delta}\max_{t_i\in\overline{\omega}}\sup_{t_i\leq s<t<t_{i+1}\leq T}\rho(u(t),u(s)),
\end{equation}
where $\Pi_\delta$ is the set of all increasing sequences $\overline{\omega}=\{0=t_0<t_1<\ldots<t_n=T\}$ with the following property
\begin{equation*}
t_{i+1}-t_{i}\geq \delta,\qquad i=0,\ldots ,n-1.
\end{equation*}
\end{definition}
\begin{theorem}
\label{cmpct}
A set $A\subset D([0,T];\mathbb{S})$ has compact closure iff it satisfies the following two conditions:\\
\begin{itemize}
\item[(a)] there exists a dense subset $I\subset [0,T]$ such that for every $t\in I$ the set $\{u(t),u\in A\}$ has compact closure in $\mathbb{S}$,
\item[(b)]$\lim_{\delta\rightarrow 0}\sup_{u\in A}w_{[0,T],\mathbb{S}}(u,\delta)=0$.
\end{itemize}
 \end{theorem}
Let us consider the ball $\mathbb{B}:=\{x\in \Ll(\mathcal{O}): \ada{\|x\|_{\Ll}\leq r}\}$. If $\mathbb{B}_w$ denote the ball $\mathbb{B}$ endowed with the weak topology of $\Ll$, then it is well-known that $\mathbb{B}_w$ is metrizable. Let $D([0,T]; \mathbb{B}_w)$ be the space of weakly c\`adl\`ag functions $u:[0, T]\to \Ll$ such that $\sup_{t\in [0, T]} \ada{\|u(t)\|_{\Ll}}\leq r$. Note that, the space $D([0,T]; \mathbb{B}_w)$ is metrizable as well. Then we have the following lemma from Lemma 2 of \cite{motyl2013stochastic}.
\begin{lemma}
	Let $u_n: [0, T]\to \Ll, \ n\in\mathbb{N}$ be functions such that 
	\begin{enumerate}
		\item $\sup_{n\in\mathbb{N}} \sup_{s\in [0, T]} \ada{\|u_n(s)\|_{\Ll}\leq r}$,
		\item $u_n\to u$ in $D([0,T]; \mathbb{H}^{-1})$.
	\end{enumerate}
	Then $u, u_n\in D([0,T]; \mathbb{B}_w)$ and $u_n\to u$ in $ D([0,T]; \mathbb{B}_w)$ as $n\to\infty$.
\end{lemma}

We now deal with the following functional spaces endowed with the respective topologies:
$$D([0,T];\mathbb{H}^{-1}(\mathcal{O}))_J:=\,\, \mbox{the  the space of c\`adl\`ag functions} \,\, u:[0, T]\to \mathbb{H}^{-1}(\mathcal{O})$$
with the extended Skorokhod topology $\tau_1,$ 
$$L^2 _w(0,T;\h(\mathcal{O})):=\,\,\mbox{the space}\,\, L^2 (0,T;\h(\mathcal{O}))\, \, \mbox{with the weak topology}\,\, \tau_2,$$
$ D([0,T];\Ll_w(\mathcal{O})):=\,\,\mbox{the  the space of all weakly c\`adl\`ag functions} \,\, u:[0, T]\to$ \\ $\Ll(\mathcal{O})$ with the weakest topology $\tau_3$ such that for all $h\in \Ll(\mathcal{O})$ the mappings
$$D([0,T];\Ll_w(\mathcal{O})) \ni u \mapsto (u(\cdot),h)_{\Ll} \in D([0,T];\mathbb{R})$$ are continuous.
$$L^2(0,T;\Ll(\mathcal{O}))\,\,\mbox{is endowed with its strong topology }\,\,\tau_4.$$

We take the path space $\mathcal{Z} = D([0,T];\mathbb{H}^{-1}(\mathcal{O}))_J\cap D([0,T];\Ll_w(\mathcal{O}))\cap L^2_w(0,T;\h(\mathcal{O}))\cap L^2(0,T;\Ll(\mathcal{O}))$ and $\tau$ be the supremum of the corresponding topologies.
\begin{theorem}
\label{cpt}
A set $\mathcal{K}\subset \mathcal{Z}$ is $\tau$-relatively compact if the following three conditions hold:
\begin{itemize}
\item[(a)]$\forall u\in \mathcal{K}$ and all $t\in [0,T],u(t)\in\Ll(\mathcal{O})$ and $\sup_{u\in \mathcal{K}}\sup_{s\in [0,T]}\|u(s)\|_{\Ll}<\infty$,
\item[(b)]$\sup_{u\in \mathcal{K}}\int_{0}^T\|u(s)\|_{\h}^2 ds<\infty$, i.e. $\mathcal{K}$ is bounded in $L^2(0, T; \h(\mathcal{O}))$,
\item[(c)]$\lim_{\delta\rightarrow 0}\sup_{u\in \mathcal{K}} w_{[0,T],\mathbb{H}^{-1}(\mathcal{O})}(u,\delta)=0$.
\end{itemize}
\end{theorem}
\noindent For proof see Lemma 3.3 in \cite{BM}, Lemma 4.1 in \cite{motyl14}, Theorem 2 of \cite{motyl2013stochastic}, Lemma 2.7 in \cite{MRo}
\begin{definition}
Let $(\mathbb{S},\rho)$ be a separable and complete metric space. Let $(\Omega ,\mathcal{F},P)$ be a probability space with the filtration $F:=(\mathcal{F}_t)_{t\in[0,T]}$ satisfying the usual hypotheses, and let $(X_n)_{n\in\mathbb{N}}$ be a sequence of c\`adl\`ag, $F$-adapted and $\mathbb{S}$-valued processes. $(X_n)_{n\in\mathbb{N}}$ is said to satisfy the Aldous condition iff $\forall\,\epsilon>0,\;\forall\,\eta>0,\;\exists\,\delta>0$ such that for every sequence $(\tau_n)_{n\in\mathbb{N}}$ of stopping times with $\tau_n\leq T$
\begin{equation*}
\sup_{n\in\mathbb{N}}\sup_{0\leq\theta\leq\delta}P\{\rho(X_n(\tau_n+\theta),X_n(\tau_n))\geq\eta\}\leq\epsilon.
\end{equation*}
\end{definition}
\begin{lemma}
\label{aldt}
Let $(X_n)_{n\in\mathbb{N}}$ satisfies the Aldous condition. Let $\mathbb{P}^n$ be the law of $X_n$ on $D([0,T];\mathbb{S}),n\in\mathbb{N}$. Then for every $\epsilon>0$ there exists a subset $A_\epsilon\subset D([0,T];\mathbb{S})$ such that 
\begin{equation*}
\sup_{n\in\mathbb{N}}\mathbb{P}^n(A_\epsilon)\geq 1-\epsilon\quad \mbox{and}\quad
\lim_{\delta\rightarrow 0}\sup_{u\in A_{\epsilon}}w_{[0,T],\mathbb{S}}(u,\delta)=0.
\end{equation*}
\end{lemma}
\noindent
We will use the following lemma given in \cite{metivier} and \cite{motyl2013stochastic}.
\begin{lemma}
\label{ald}
Let $(E,\|\cdot\|_{E})$ be a separable Banach space and let $(X_n)_{n\in\mathbb{N}}$ be a sequence of $E$-valued random variables. Assume that for every sequence $(\tau_n)_{n\in\mathbb{N}}$ of $F$-stopping times with $\tau_n\leq T$ and for every $n\in\mathbb{N}$ and $\theta\geq 0$  the following condition holds
\begin{equation}
\label{sc5}
\mathbb{E}[\|X_n(\tau_n+\theta)-X_n(\tau_n)\|_{E}^\alpha]\leq C\theta^{\beta},
\end{equation}
for some $\alpha,\beta>0$ and some constant $C>0$. Then the sequence $(X_n)_{n\in\mathbb{N}}$ satisfies the Aldous condition in the space $E$.
\end{lemma}
We use the tightness condition for the Prokhorov-Varadarajan theorem which states that a sequence of measures $(\tilde{\mathbb{P}}^n)_{n\in\mathbb{N}}$ is tight on a topological space $E$ if for every $\epsilon>0$ there exists a compact set $K_\epsilon\subset E$ such that $\sup_n \tilde{\mathbb{P}}^n(E\setminus K_{\epsilon})\leq\epsilon$. Hence the tightness of measure in $\mathcal{Z}$ is given by the following theorem.
\begin{theorem}
\label{sdcthm1}
Let  $(X_n)_{n\in\mathbb{N}}$ be a sequence of c\`adl\`ag $F$-adapted $\mathbb{H}^{-1}(\mathcal{O})$-valued processes such that
\begin{itemize}
\item[(a)] there exists a positive constant $C_1$ such that
\begin{equation*}
\sup_{n\in\mathbb{N}}\mathbb{E}[\sup_{s\in[0,T]}\|X_n(s)\|_{\Ll}]\leq C_1,
\end{equation*}
\item[(b)] there exists a positive constant $C_2$ such that
\begin{equation*}
\sup_{n\in\mathbb{N}}\mathbb{E}[\int_0^T \|X_n(s)\|_{\h}^2 ds]\leq C_2,
\end{equation*}
\item[(c)] $(X_n)_{n\in\mathbb{N}}$ satisfies the Aldous condition in $\mathbb{H}^{-1}(\mathcal{O})$.
\end{itemize}
Let $\tilde{\mathbb{P}}^n$ be the law of $X_n$ on $\mathcal{Z}$. Then for every $\epsilon>0$ there exists a compact subset $K_\epsilon$ of $\mathcal{Z}$ such that
\begin{equation}
\tilde{\mathbb{P}}^n(K_\epsilon)\geq 1-\epsilon,
\end{equation}
and the sequence of measures $\{\tilde{\mathbb{P}}^n,n\in\mathbb{N}\}$ is said to be tight on $(\mathcal{Z},\tau)$.
\end{theorem}
\noindent For proof see Corollary 1, \cite{motyl2013stochastic}.
\subsection{Martingale Problem}
We now consider the stochastic tidal dynamics equation with L\'evy forcing as defined in Section \ref{setting} with initial value control as
\begin{align}
&du(t) + [Au(t)+B(u(t))+g\nabla\z(t)]dt
=f(t)dt+\s(t,u(t))dW(t)\notag\\& \quad+\int_Z H(u(t-),z) \n,  \label{sc1}\\
& d\z(t)+Div(hu(t))dt=0, \label{sc2}\\
&u(0)=u_0+U,\qquad \z(0)=\z_0,\label{sc3} 
\end{align}
where $f\in L^2(\Omega;L^2(0,T;\Ll(\mathcal{O})))$, $u_0,U\in  L^2(\Omega;\Ll(\mathcal{O}))$ and $\z_0\in  L^2(\Omega;L^2(\mathcal{O}))$.  We assume that $\s$ and $H$ satisfy Assumption \ref{Hyp}.
\begin{definition} \label{defi.mart}
A martingale solution of \eqref{sc1}-\eqref{sc3} is a system \\
$(\overline{\Omega},\overline{\mathcal{F}},\overline{F},\overline{P},\overline{u},\overline{z},\overline{U},\overline{N},\overline{W})$, where
\begin{itemize}
\item $(\overline{\Omega},\overline{\mathcal{F}},\overline{F},\overline{P})$ is a filtered probability space with a filtration $\overline{F}=\{\overline{\mathcal{F}}_t\}_{t\geq 0}$,
\item $\overline{N}$ is a time homogeneous Poisson random measure over $(\overline{\Omega},\overline{\mathcal{F}},\overline{F},\overline{P})$ with the intensity measure $\lambda$,
\item $\overline{W}$ is a cylindrical Wiener process over $(\overline{\Omega},\overline{\mathcal{F}},\overline{F},\overline{P})$,
\item $\overline{U}$ is measurable with $\overline{P}$- a.e. $\omega \in \overline{\Omega},$ $\overline{U}(\omega)\in \Ll(\mathcal{O})$,\item \gr{$\overline{u}, \overline{z}$ are progressively measurable processes with $\overline{P}$- a.e. $\omega \in \overline{\Omega},$ the paths
\begin{align*}
&\overline{u}(\cdot, \omega)\in D([0,T];\mathbb{H}^{-1}(\mathcal{O}))\cap D([0,T];\Ll_w(\mathcal{O})) \cap L^2_w(0,T;\h(\mathcal{O}))\cap L^2(0,T;\Ll(\mathcal{O})),\,\,\\
&\overline{z}(\cdot, \omega)\in L^2(0,T;L^2(\mathcal{O})) \cap C([0,T];H^{-1}(\mathcal{O})),
\end{align*}
}
 such that for all $t\in [0, T]$, for all $v\in\h(\mathcal{O})$ and for all $w\in L^2(\mathcal{O})$, the following identities hold $\overline{P}$-a.s.
\begin{IEEEeqnarray}{llr}
(&\overline{u}(t),v)_{\Ll}+\int_0^t (A\ou(s),v)_{\Ll}ds+\int_0^t(B(\ou(s)),v)_{\Ll}ds+\int_0^t\gr{\langle g\nabla\overline{z}(s) ,v \rangle} ds&\nonumber\\
&=(u_0,v)_{\Ll}+(\overline{U},v)_{\Ll}+\int_0^t(f(s),v)_{\Ll}ds\nonumber\\
&\quad+\int_0^t(\s(s,\ou(s))d \overline{W}(s),v)_{\Ll}
+\int_0^t\int_Z (H(\ou(s-),z),v)_{\Ll} \tilde{\overline{N}}(ds,dz) ,\\
&(\overline{z}(t),w)_{L^2}+\int_0^t (Div(h\ou(s)),w)_{L^2} ds=(z_0,w)_{L^2},
\end{IEEEeqnarray}
\end{itemize}
\gr{where $\langle \cdot,\cdot \rangle$ denotes the duality pairing between $\mathbb{H}^{1}_{0}(\mathcal{O})$ and $\mathbb{H}^{-1}(\mathcal{O}).$}
\end{definition}
Equivalent infinite dimensional martingale formulations are available in the literature [e.g. see Theorem 9 of Sritharan \cite{sritharan}]. Similar formulations, in the finite dimensional case, are known due to Stroock and Varadhan \cite{var} and Viot \cite{viot}. Equivalence between infinite dimensional version of Stroock-Varadhan martingale formulations and weak formulations as in the spirit of Definition \ref{defi.mart} can be found in Theorem 10 of Sritharan \cite{sritharan}.  
\subsection{Existence of Martingale Solution}
We will prove the existence of a martingale solution using the Galerkin approximations as explained in Section \ref{estimate}. We write
\begin{IEEEeqnarray*}{lrl}
\label{sc4}
u^n(t)&=&u^n_0+U^n-\int_0^t (Au^n(s)+\gr{B(u^n(s))}+g\nabla \z^n(s)-\gr{f(s))}ds\\
&&+\int_0^t \s^n(s,u^n(s))dW^n(s)+\int_0^t \int_Z H^n(u^n(s-),z)\ns .\yesnumber\\
\z^n(t)&=&\z^n_0 -\int_0^tDiv(hu^n(s))ds .\yesnumber\label{sc41}
\end{IEEEeqnarray*}
For each $n\in\mathbb{N}$, we use the measures $\mathcal{L}(u^n)$ and $\mathcal{L}(\z^n)$ defined on $(\mathcal{Z},\tau)$ and \gr{$L^2(0,T;L^2(\mathcal{O}))$} respectively by the solution $(u^n, \z^n)$ of the Galerkin equations \eqref{sc4}-\eqref{sc41}.
\begin{lemma}
\label{tight}
The set of measures $\{\mathcal{L}(u^n),n\in\mathbb{N}\}$ is tight on $(\mathcal{Z},\tau).$
\end{lemma}
\begin{proof}
We will prove the tightness of the measures using the tightness criterion given in Theorem \ref{sdcthm1}. \gr{From Proposition \ref{prop}, conditions (a) and (b) of Theorem \ref{sdcthm1} are satisfied. Now we need to verify that condition (c) of Theorem \ref{sdcthm1} is also satisfied. Using Lemma \ref{ald} we will show that sequence $(u^n)_{n\in\mathbb{N}}$ satisfy Aldous condition in the space $\mathbb{H}^{-1}(\mathcal{O}).$} Let $\theta>0$. Let $(\tau_n)_{n\in\mathbb{N}}$ be a sequence of stopping times where $0\leq\tau_n\leq \tau_n+\theta \leq T$. \gr{Let us recall \eqref{sc4} as:}
\begin{IEEEeqnarray*}{lrl}
u^n(t)=u^n_0+U^n-\int_0^t Au^n(s)ds-\int_0^t B(u^n(s)) ds-\int_0^t g\nabla\z^n(s) ds\\
+\int_0^t f(s)ds+\int_0^t \s^n(s,u^n(s))dW^n(s)+\int_0^t \int_Z H^n(u^n(s-),z)\ns\\
=\gr{J^n_1+J^n_2}+J^n_3(t)+J^n_4(t)+J^n_5(t)+J^n_6(t)+J^n_7(t)+J^n_8(t),\qquad t\in [0,T].
\end{IEEEeqnarray*}
Now we will show that \eqref{sc5} holds for each $J^n_i,\; i\in\{1,2,\ldots,8\}$. \newline
First note, since the terms $J^n_1$ and $J_2^n$ are independent of time, clearly \eqref{sc5} is satisfied for any $\alpha,\beta>0.$
\newline
Now consider the term $J^n_3(t)$. Since $A:\h(\mathcal{O})\rightarrow\mathbb{H}^{-1}(\mathcal{O})$, therefore for all $v\in\h(\mathcal{O})$
\begin{equation*}
\langle Au,v\rangle =(Au,v)_{\Ll}
\leq C_1\|u\|_{\h}\|v\|_{\h}.
\end{equation*}
Hence, using the above inequality 
\begin{equation} \label{eq.A}
\| Au\|_{\mathbb{H}^{-1}}\leq C_1\|u\|_{\h}.
\end{equation}
\gr{Therefore by using \eqref{eq.A}}, H\"older's inequality and \eqref{eq15}, $J^n_3$ can be estimated as
\begin{align*}
\mathbb{E}[\| &J^n_3(\tau_n+\theta)-J^n_3(\tau_n)\|_{\mathbb{H}^{-1}}]=\mathbb{E}[\|\int_{\tau_n}^{\tau_n+\theta}Au^n(s)ds\|_{\mathbb{H}^{-1}}]\\
&\quad\leq C_1\mathbb{E}[\int_{\tau_n}^{\tau_n+\theta}\| u^n(s)\|_{\h}ds] \leq C_1\mathbb{E}[\theta^{1/2}\left(\int_0^T \|u^n(s)\|^2_{\h} ds\right)^{1/2}] \\
&\quad\gr{\leq C_1 \theta^{1/2}\Big[\mathbb{E}\left(\int_0^T \|u^n(s)\|^2_{\h} ds\right)\Big]^{1/2}}  \leq c_2\theta^{1/2}.
\end{align*}
Thus $J^n_3$ satisfies $\eqref{sc5}$ with $\alpha=1$ and $\beta=\frac{1}{2}$.\\
We next consider the term $J^n_4(t)$.
\gr{In Proposition \ref{moment} for $p=4$ we note that $w^0\in L^{8}(\Omega; L^{8}\\(0,T;\mathbb{L}^8(\mathcal{O})))
\subset  L^4(\Omega; L^4(0,T;\mathbb{L}^4(\mathcal{O}))).$  
 Again by Proposition \ref{moment}, for $p=4$, we obtain \\
 $\mathbb{E}[\int_0^T \|u^n(s)\|_{\mathbb{L}^2}^2 \|u^n(s)\|_{\h}^2ds] \leq C_{1(4)}$. 
By the estimate \eqref{eq13} we have $\mathbb{E}[\int_0^T \|u^n(s)\|_{\mathbb{L}^4}^4 ds] \leq C_{1(4)}$.
Using the embedding $\h(\mathcal{O})\hookrightarrow\mathbb{L}^4(\mathcal{O})\hookrightarrow\Ll(\mathcal{O})
\hookrightarrow\mathbb{H}^{-1}(\mathcal{O})$, property of the operator B in \eqref{e2}, Minskowskii and H\"older's inequalities and Proposition \ref{moment} (for $p=4$) we estimate $J^n_4$ as
\begin{align*}\label{lrl}
&\mathbb{E}[\| J^n_4(\tau_n+\theta)-J^n_4(\tau_n)\|_{\mathbb{H}^{-1}}]
=\mathbb{E}[\|\int_{\tau_n}^{\tau_n+\theta}B(u^n(s))ds\|_{\mathbb{H}^{-1}}]\\
&\quad \leq \mathbb{E}[\int_{\tau_n}^{\tau_n+\theta}\|B(u^n(s))\|_{\mathbb{H}^{-1}}ds]
 \leq c\mathbb{E}[\int_{\tau_n}^{\tau_n+\theta}\|B(u^n(s))\|_{\Ll}ds]\\
 &\quad \leq cC_2\mathbb{E}[\int_{\tau_n}^{\tau_n+\theta}\|u^n(s)+w^0(s)\|_{\mathbb{L}^4}^2 ds]\\
& \quad \leq 2cC_2 \mathbb{E}[\int_{\tau_n}^{\tau_n+\theta} \left(\|u^n(s)\|_{\mathbb{L}^4}^2 + \|w^0(s)\|_{\mathbb{L}^4}^2\right)ds]\\
&\quad\leq c_3\theta^{1/2}  \left(\mathbb{E}\left[\int_{\tau_n}^{\tau_n+\theta} \|u^n(s)\|_{\mathbb{L}^4}^4 ds\right]\right)^{1/2} + c_3\theta^{1/2} 
\left(\mathbb{E}\left[\int_{\tau_n}^{\tau_n+\theta} \|w^0(s)\|_{\mathbb{L}^4}^4 ds\right]\right)^{1/2}\\
&\quad\leq c_3\theta^{1/2}  \left(\mathbb{E}\left[\int_0^T \|u^n(s)\|_{\mathbb{L}^4}^4 ds\right]\right)^{1/2} 
+ c_3\theta^{1/2} \left(\mathbb{E}\left[\int_0^T \|w^0(s)\|_{\mathbb{L}^4}^4 ds\right]\right)^{1/2}\\
&\quad\leq c_4\theta^{1/2}.
\end{align*}
Thus $J^n_4$ satisfies $\eqref{sc5}$ with $\alpha=1$ and $\beta=\frac{1}{2}$.\\
}
Next consider the term $J^n_5(t)$.
Consider the operator $C:L^2(\mathcal{O})\rightarrow\mathbb{H}^{-1}(\mathcal{O})$ defined by $C(\z)=g\nabla \z.$
For all $v\in\h(\mathcal{O}),$ we have $\gr{|\langle C(\z),v\rangle|
= |-g(\z,Div(v))_{L^2}}|\leq g\|\z\|_{L^2}\|v\|_{\h}.$
Hence
\begin{equation} \label{eq.C}
\| C(\z)\|_{\mathbb{H}^{-1}}\leq g\|\z\|_{L^2}.
\end{equation}
\gr{So by using \eqref{eq.C}, H\"older's inequality and \eqref{eq15}, we have
\begin{align*}
\mathbb{E}[\| &J^n_5(\tau_n+\theta)-J^n_5(\tau_n)\|_{\mathbb{H}^{-1}}]=\mathbb{E}[\|\int_{\tau_n}^{\tau_n+\theta} g\nabla \z^n(s)ds\|_{\mathbb{H}^{-1}}]\\
&\quad\leq g\mathbb{E}[\int_{\tau_n}^{\tau_n+\theta}\|\z^n(s)\|_{L^2}ds] \leq g\theta^{1/2} \Big[\mathbb{E}\left(\int_{\tau_n}^{\tau_n+\theta}\|\z^n(s)\|_{L^2}^2 ds\right)\Big]^{1/2}\\
& \quad\leq g\theta^{1/2} \Big[\mathbb{E}\left(\int_{0}^{T}\|\z^n(s)\|_{L^2}^2 ds\right)\Big]^{1/2}\leq g\theta^{1/2}\Big[\mathbb{E}\left(T\sup_{0\leq s\leq T}\|\z^n(s)\|_{L^2}^2\right)\Big]^{1/2}\\
&\quad \leq g\theta^{1/2}T^{1/2}\left(\mathbb{E}\left[\sup_{0\leq s\leq T}\|\z^n(s)\|_{L^2}^2\right]\right)^{1/2}\leq  c_5\theta^{1/2}.
\end{align*}}
Thus $J^n_5$ satisfies $\eqref{sc5}$ with $\alpha=1$ and $\beta=\frac{1}{2}$.\\
We next consider the term $J^n_6(t)$.
Since $\Ll(\mathcal{O})\hookrightarrow\mathbb{H}^{-1}(\mathcal{O})$ and \gr{$f\in L^2(\Omega;L^2(0,T;\Ll(\mathcal{O})))$, by H\"older inequality, we have
\begin{align*}
&\mathbb{E}[\| J^n_6(\tau_n+\theta)-J^n_6(\tau_n)\|_{\mathbb{H}^{-1}}]=\mathbb{E}[\|\int_{\tau_n}^{\tau_n+\theta} f(s)ds\|_{\mathbb{H}^{-1}}]\\
&\quad\leq \mathbb{E}[\int_{\tau_n}^{\tau_n+\theta}\|f(s)\|_{\mathbb{H}^{-1}}ds]
\leq c\mathbb{E}[\int_{\tau_n}^{\tau_n+\theta}\|f(s)\|_{\Ll}ds]\\
& \quad\leq c\theta^{1/2}\left(\mathbb{E}\left[\int_{\tau_n}^{\tau_n+\theta}\|f(s)\|_{\Ll}^2 ds\right]\right)^{1/2}\\
&\quad\leq c\theta^{1/2}\left(\mathbb{E}\left[\int_0^T\|f(s)\|_{\Ll}^2 ds\right]\right)^{1/2}\leq c_6\theta^{1/2}.
\end{align*}}
Thus $J^n_6$ satisfies $\eqref{sc5}$ with $\alpha=1$ and $\beta=\frac{1}{2}$.\\
Now we consider the term $J^n_7(t)$.
\gr{Using the embedding $\Ll(\mathcal{O})\hookrightarrow\mathbb{H}^{-1}(\mathcal{O}),$  It\^o isometry, \eqref{eq.sig.n}, Assumption \ref{Hyp} and inequality \eqref{eq15},} we obtain
\begin{align*}
&\mathbb{E}[\| J^n_7(\tau_n+\theta)-J^n_7(\tau_n)\|^2_{\mathbb{H}^{-1}}]=\mathbb{E}[\|\int_{\tau_n}^{\tau_n+\theta} \s^n(s,u^n(s)) dW^n(s)\|^2_{\mathbb{H}^{-1}}]\\
&\quad\leq c\mathbb{E}[\|\int_{\tau_n}^{\tau_n+\theta} \s^n(s,u^n(s)) dW^n(s)\|_{\Ll}^2]=c \mathbb{E}[\int_{\tau_n}^{\tau_n+\theta} \|\s^n(s,u^n(s))\|_{L_Q(\mathbb{L}^2,\Ll)}^2 ds]\\
& \quad \gr{\leq c \mathbb{E}[\int_{\tau_n}^{\tau_n+\theta} \|\s(s,u^n(s))\|_{L_Q(\mathbb{L}^2,\Ll)}^2 ds]}
\leq cK \mathbb{E}[\int_{\tau_n}^{\tau_n+\theta} (1+\| u^n(s)\|_{\Ll}^2) ds]\\& \quad \leq  cK\theta(1+\mathbb{E}[\sup_{0\leq s\leq T}\|u^n(s)\|_{\Ll}^2])\leq c_6\theta.
\end{align*}
Thus $J^n_7$ satisfies $\eqref{sc5}$ with $\alpha=2$ and $\beta=1$.\\
We finally consider the term $J^n_8(t)$.
\gr{Using the embedding $\Ll(\mathcal{O})\hookrightarrow\mathbb{H}^{-1}(\mathcal{O}),$ L\'evy-It\^o isometry,  \eqref{eq.sig.n1}, Assumption \ref{Hyp} and inequality \eqref{eq15},} we obtain
\gr{\begin{IEEEeqnarray*}{lrl}
\mathbb{E}[\| J^n_8(\tau_n+\theta)-J^n_8(\tau_n)\|^2_{\mathbb{H}^{-1}}]&=&\mathbb{E}[\|\int_{\tau_n}^{\tau_n+\theta} \int_Z H^n(u^n(s-),z)\ns\|_{\mathbb{H}^{-1}}^2]\\
&\leq & c\mathbb{E}[\|\int_{\tau_n}^{\tau_n+\theta}\int_Z H^n(u^n(s-),z)\ns\|_{\Ll}^2]\\
&=& c\mathbb{E}[\int_{\tau_n}^{\tau_n+\theta}\int_Z \| H^n(u^n(s),z)\|_{\Ll}^2\lambda(dz)ds]\\
& \leq & c\mathbb{E}[\int_{\tau_n}^{\tau_n+\theta}\int_Z \| H(u^n(s),z)\|_{\Ll}^2\lambda(dz)ds]\\
& \leq & cK\,\mathbb{E}[\int_{\tau_n}^{\tau_n+\theta} (1+\| u^n(s))\|_{\Ll}^2) ds]\\
& \leq & cK\,\theta\,(1+\mathbb{E}[\sup_{0\leq s\leq T}\|u^n(s)\|_{\Ll}^2])\leq c_6\theta.
\end{IEEEeqnarray*}}
Thus $J^n_8$ satisfies $\eqref{sc5}$ with $\alpha=2$ and $\beta=1$.\\
\gr{Finally combining estimates of each $J^n_i;i=1,\cdots,8$,
we have,
\begin{align*}
&\mathbb{E}[\| u^n(\tau_n+\theta)-u^n(\tau_n)\|_{\mathbb{H}^{-1}}] = \mathbb{E}\Big[\| \sum_{i=1}^{8}(J^n_i(\tau_n+\theta)-J^n_i(\tau_n)\|_{\mathbb{H}^{-1}}\Big]\\
& \leq  \sum_{i=1}^{8} \mathbb{E}\Big[\|(J^n_i(\tau_n+\theta)-J^n_i(\tau_n)\|_{\mathbb{H}^{-1}}\Big] =  \sum_{i=1}^{6} \mathbb{E}\Big[\|(J^n_i(\tau_n+\theta)-J^n_i(\tau_n)\|_{\mathbb{H}^{-1}}\Big]\\&\quad +  \sum_{i=7}^{8} \mathbb{E}\Big[\|(J^n_i(\tau_n+\theta)-J^n_i(\tau_n)\|_{\mathbb{H}^{-1}}\Big]\\& \leq C_1\theta^{1/2}+  \sum_{i=7}^{8} \Big[\mathbb{E}(\|(J^n_i(\tau_n+\theta)-J^n_i(\tau_n)\|^2_{\mathbb{H}^{-1}}\Big]^{1/2}\\& \leq C_1\theta^{1/2} + C_2\theta^{1/2}:= C\theta^{1/2}.
\end{align*}
Hence $u^n$ satisfies Aldous condition in the space $\mathbb{H}^{-1}(\mathcal{O})$ with $\alpha=1$ and $\beta=\frac{1}{2},$ which completes the proof.}
\end{proof}
\begin{lemma}
\label{tight2}
The set of measures $\{\mathcal{L}(\z^n),n\in\mathbb{N}\}$ is tight on $L^2(0,T;L^2(\mathcal{O}))\cap C([0,T]; H^{-1}(\mathcal{O})).$ 
\end{lemma}
\begin{proof}
First note that, due to lack of $H^1$ estimate for $\z^n$,  we can not apply Theorem \ref{sdcthm1}, and hence we may not be able prove tightness following the method of the previous Lemma. However, one could possibly consider more regular initial data and apply Proposition \ref{PropH1} to obtain $H^1$ estimate for $\z^n$ and then apply Theorem \ref{sdcthm1} to obtain tightness. But we do not proceed in this direction.
\par\noindent
To prove the tightness in $L^2(0,T;L^2(\mathcal{O}))$, we follow the classical argument of Chow and Khasminskii \cite{CK}. 
\par\noindent
By Proposition \ref{prop}, $\E[\sup_{0\leq t\leq T}\|\z^n(t)\|_{L^2}^2]\leq C_1$, hence using Fubini's theorem we have
\begin{equation*}
\E[\|\z^n\|_{L^2(0,T;L^2)}^2]\leq C_1T.
\end{equation*} 
By the Chebychev inequality, we see that for any $r>0$
\begin{equation*}
P(\|\z^n\|_{L^2(0,T;L^2)}>r)\leq \dfrac{\mathbb{E}[\|\z^n\|_{L^2(0,T;L^2)}^2]}{r^2}\leq \dfrac{C_1T}{r^2}.
\end{equation*}
Let $\epsilon>0$ be given. Let $R_1>0$ be such that $\frac{C_1T}{R_1^2}< \epsilon$. Then
$P(\|\z^n\|_{L^2(0,T;L^2)}>R_1)< \epsilon.$
Denote
\begin{equation*}
B_\epsilon := \{\z^n\in L^2(0,T;L^2(\mathcal{O})): \|\z^n\|_{L^2(0,T;L^2)}\leq R_1\}.
\end{equation*}
Then it is clear that $P(B_\epsilon)\geq 1-\epsilon$. Hence for every $\epsilon >0$, there exists a compact subset $B_\epsilon$ of $L^2(0,T;L^2(\mathcal{O}))$ such that $P(B_\epsilon)\geq 1-\epsilon$. Thus the set of measures $\{\mathcal{L}(\z^n),n\in\mathbb{N}\}$ is tight on $L^2(0,T;L^2(\mathcal{O}))$.
\par\noindent
To prove the tightness in $C([0,T]; H^{-1}(\mathcal{O}))$, we follow the method due to Flandoli and Gatarek \cite{FG} (see Theorem 3.1, Step 2). 
\par\noindent
We decompose $\z^n$ as
\begin{align}
\z^n(t)&=\z^n_0 -\int_0^tDiv(hu^n(s))ds\nonumber\\
&= J_9^n + J_{10}^n(t).
\end{align}
We already have $\E\|J_9^n\|_{L^2}^2 \leq C_1.$ 
Also, by Proposition \ref{prop} we have
\begin{align*}
&\E\|J_{10}^n(t)\|_{L^2(0,T;L^2)}^2 \leq \E\int_0^T \int_0^t\| Div(hu^n(s))\|_{L^2}^2 ds\, dt\\ 
&\leq 2\,T \E\int_0^T \big(\|\nabla h\|_{\mathbb L^\infty}^2 \|u^n(t)\|_{\mathbb L^2}^2 + \|h\|_{L^\infty}^2 \|\nabla u^n(t)\|_{\mathbb L^2}^2\big) dt\\
&\leq 2\,T^2 \|\nabla h\|_{\mathbb L^\infty}^2 \E[\sup_{0\leq t\leq T} \|u^n(t)\|_{\mathbb L^2}^2]  + 2\,T\,\|h\|_{L^\infty}^2\E\int_0^T \|\nabla u^n(t)\|_{\mathbb L^2}^2 dt\\
&\leq C_2.
\end{align*}
Thus we have $\E\|\z^n\|_{L^2(0,T;L^2)}^2\leq C_3$, and hence $\E\|\z^n\|_{W^{1,2}(0,T;L^2)}^2\leq C_4$, where for a generic Banach space $B$ and a real number $p\geq 1$, $W^{1,p}(0, T; B)$ denotes the space of all $v\in L^p(0, T; B)$ such that $\dfrac{dv}{dt}\in L^p(0, T; B)$. Since $W^{1,2}(0,T;L^2(\mathcal{O}))\subset W^{\alpha, 2}(0, T; L^2(\mathcal{O}))$ for all $\alpha\in (0,1)$, we infer $$\E\|\z^n\|_{W^{\alpha,2}(0,T;L^2)}^2\leq C_4.$$ 
\noindent
Since $L^2(\mathcal{O})$ is compactly embedded in $H^{-1}(\mathcal{O})$, we can apply Theorem 2.2 of Flandoli and Gatarek \cite{FG} to infer that for all real numbers $\alpha\in (\frac{1}{2},1)$, the space $W^{\alpha,2}(0,T;L^2(\mathcal{O}))$  is compactly embedded into $C([0, T]; H^{-1}(\mathcal{O}))$. Hence the set of measures $\{\mathcal{L}(\z^n),n\in\mathbb{N}\}$ is tight on $C([0, T]; H^{-1}(\mathcal{O}))$.
\end{proof}
\begin{lemma}
\label{tight3}
The set of measures $\{\mathcal{L}(U^n),n\in\mathbb{N}\}$ is tight on $\Ll(\mathcal{O})$.
\end{lemma}
\begin{proof}
Using the assumption $\E[\|U^n\|_{\Ll}^2]\leq c$,
 and using the Chebychev inequality, we see that for any $r>0$
\begin{equation*}
P(\|U^n\|_{\Ll}>r)\leq \dfrac{\mathbb{E}[\|U^n\|_{\Ll}^2]}{r^2}\leq \dfrac{c}{r^2}.
\end{equation*}
Let $\epsilon>0$ be given. Let $R_2>0$ be such that $\frac{c}{R_2^2}< \epsilon$. Then
\begin{equation*}
P(\|U^n\|_{\Ll}>R_2)< \epsilon.
\end{equation*}
Denote
\begin{equation*}
B_\epsilon = \{U^n\in \Ll(\mathcal{O}): \|U^n\|_{\Ll}\leq R_2\}.
\end{equation*}
Then $P(B_\epsilon)\geq 1-\epsilon$. Hence for every $\epsilon >0$, there exists a compact subset $B_\epsilon$ of $ \Ll(\mathcal{O})$ such that $P(B_\epsilon)\geq 1-\epsilon$. This proves the Lemma.
\end{proof}
\begin{theorem}
\label{martexist}
There exists a martingale solution of \eqref{sc1}-\eqref{sc3} provided the Assumptions \ref{Hyp} are satisfied.
\end{theorem}
\begin{proof}
\textbf{Step I}:\\
\noindent
\gr{By Lemmas \ref{tight}, \ref{tight2} and \ref{tight3}, the set of measures $\{\mathcal{L}(u^n),n\in\mathbb{N}\},$ $\{\mathcal{L}(\z^n),n\in\mathbb{N}\}$ and $\{\mathcal{L}(U^n),n\in\mathbb{N}\}$ are tight on the spaces $(\mathcal{Z},\tau),$ $L^2(0,T;L^2(\mathcal{O}))\cap C([0,T]; H^{-1}(\mathcal{O}))$, and $\Ll(\mathcal{O})$ respectively.} Define $N^n=N,\;\forall n\in\mathbb{N}$. Then the set of measures $\{\mathcal{L}(N^n),n\in\mathbb{N}\}$ is tight on the space \gr{$M_{\bar{\mathbb{N}}}([0,T]\times Z)$}, where $\bar{\mathbb{N}}:=\mathbb N\cup\{\infty\}$ and $M_{\bar{\mathbb{N}}}(S)$ denotes the set of all $\bar{\mathbb{N}}-$valued measures on the measurable space $(S, \mathscr{B}(S))$ (see Motyl \cite{motyl2013stochastic},\cite{motyl2011} for more details). Define $W^n=W,\;\forall n\in\mathbb{N}$. Then the set of measures $\{\mathcal{L}(W^n),n\in\mathbb{N}\}$ is tight on the space \gr{$C([0,T];\mathbb{R})$ of continuous function from $[0, T]$ to $\mathbb{R}$ with standard supremum norm.} \\
Thus the set $\{\mathcal{L}(u^n,\z^n,U^n,N^n, W^n),n\in\mathbb{N}\}$ is tight on \gr{$\mathcal{Z} \times \Big(L^2(0,T;L^2(\mathcal{O}))\cap C([0,T]; H^{-1}(\mathcal{O}))\Big) \times \Ll(\mathcal{O}) \times M_{\bar{\mathbb{N}}}([0,T]\times Z) \times C([0,T];\mathbb{R}).$}
\gr{By the Skorokhod theorem \cite{motyl2011}}, there exists a subsequence $(n_k)_{k\in\mathbb{N}}$, a probability space $(\overline{\Omega},\overline{\mathcal{F}},\overline{P})$, and, on this space, random variables $(u^*,z^*,U^*,N^*,W^*)$, $\{(\ou^k,\oz^k,\overline{U}^k,\on^k,\ow^k),k\in\mathbb{N}\}$ such that
\begin{itemize}
\item[(i)] $\mathcal{L}((\ou^k,\oz^k,\overline{U}^k,\on^k,\ow^k))=\mathcal{L}((u^{n_k},\z^{n_k},U^{n_k},N^{n_{k}},W^{n_k}))$ for all $k\in\mathbb{N}$,
\item[(ii)]$(\ou^k,\oz^k,\overline{U}^k,\on^k,\ow^k)\rightarrow (u^*,z^*,U^*,N^*,W^*)$ in \gr{$\mathcal{Z} \times \Big(L^2(0,T;L^2(\mathcal{O}))\cap C([0,T]; H^{-1}(\mathcal{O}))\Big) \times \Ll(\mathcal{O}) \times M_{\bar{\mathbb{N}}}([0,T]\times Z) \times C([0,T];\mathbb{R})$} with probability 1 on\\ $(\overline{\Omega},\overline{\mathcal{F}},\overline{P})$ as $k\rightarrow\infty$,
\item[(iii)]$(\on^k(\overline{\omega}),\ow^k(\overline{\omega}))=(N^*(\overline{\omega}),W^*(\overline{\omega}))$ for all $\overline{\omega}\in\overline{\Omega}$.
\end{itemize}
We denote these sequences again by $((u^n,\z^n,U^n,N^n,W^n))_{n\in\mathbb{N}}$ and \\
$((\ou^n,\oz^n,\overline{U}^n,\on^n,\ow^n))_{n\in\mathbb{N}}$. Using the definition of the space $\mathcal{Z}$, we have $\overline{P}-a.s.$
\begin{align} \label{conv.u*}
\ou^n\rightarrow u^*\;\text{in}\;L^2_{w}(0,T;\h(\mathcal{O}))\cap L^2(0,T;\Ll(\mathcal{O}))\cap D([0,T];\mathbb{H}^{-1}(\mathcal{O}))\nonumber\\ \cap D([0,T];\Ll_w(\mathcal{O})),
\end{align}
\gr{\begin{equation}
\oz^n\rightarrow z^*\quad\text{in}\quad L^2(0,T;L^2(\mathcal{O}))\cap C([0,T]; H^{-1}(\mathcal{O})),
\end{equation}
and
\begin{equation}
\overline{U}^n\rightarrow U^*\quad\text{in}\quad \Ll(\mathcal{O})
\end{equation}}
\gr{Note that, since $D([0,T];\mathbb{L}^2_n(\mathcal{O}))$ is a Polish space and $L^2_{w}(0,T;\h(\mathcal{O}))\cap L^2(0,T;\Ll(\mathcal{O}))\cap D([0,T];\mathbb{H}^{-1}(\mathcal{O}))\cap D([0,T];\Ll_w(\mathcal{O}))$ is a separable metric space, due to Kuratowski theorem  (see Theorem 1.1, Chapter 1 of \cite{Vakhania}), Borel subsets of $D([0,T]; \Ll_n(\mathcal{O}))$ are Borel subsets of  $L^2_{w}(0,T;\h(\mathcal{O}))\cap L^2(0,T;\Ll(\mathcal{O}))\cap D([0,T];\mathbb{H}^{-1}(\mathcal{O}))\cap D([0,T];\Ll_w(\mathcal{O}))$, and $P\{u^n\in D([0,T];\\ \Ll_n(\mathcal{O}))\}=1$. Hence, we may assume that $\overline{u}^n$ takes values in $\Ll_n(\mathcal{O})$ and that the laws on $D([0,T]; \Ll_n(\mathcal{O}))$ of $u^n$ and $\ou^n$ are equal.
\par\noindent
In view of the above, it is straightforward to show that the sequence $(\ou^n)_{n\in\mathbb{N}}$ satisfies the same estimates as the original sequence $(u^n)_{n\in\mathbb{N}}$. In particular, for any $p \geq 2,$ we have}
\begin{equation}
\label{ineq1}
\sup_{n\geq 1}\overline{\E}[\sup_{0\leq s\leq T}\|\ou^n(s)\|_{\Ll}^p]\leq C_{1(p)},
\end{equation}
and
\begin{equation}
\label{ineq2}
\sup_{n\geq 1}\overline{\E}[\int_0^T\|\ou^n(s)\|_{\h}^2 ds]\leq C_{2(2)}.
\end{equation}

\gr{Again repeating the same arguments (as above) 
for the random variables $\oz^n$ and $\z^n,$ it is obvious to show that the sequence $(\oz^n)_{n\in\mathbb{N}}$ satisfies the same estimates as the original sequence $(\z^n)_{n\in\mathbb{N}}$. Hence we have,}
\begin{equation}
\label{ineq3}
\sup_{n\geq 1}\overline{\E}[\sup_{0\leq s\leq T}\|\oz^n(s)\|_{L^2}^2]\leq C_{1(2)}.
\end{equation}
\gr{Now, using the assumption $\E[\|U\|_{\Ll}^2] \leq c$, we have
\begin{equation} 
\sup_{n\geq 1}\overline{\E}[\|\overline{U}^n\|_{\Ll}^2]\leq c.
\end{equation}}

\gr{
\begin{proposition} \label{X*r}
Let $u^*$ be the limiting process defined above. Then 
 \begin{align} \label{ener.*}
 \oE[\int_0^T \|u^*(s)\|_{\h}^2 ds] \leq C,
\end{align}
and for $r \geq 2,$ 
\begin{align} \label{X*r.1}
\oE \Big[\sup_{s \in [0,T]}\|u^*(s)\|_{\Ll}^r \Big]\leq C_r.
\end{align}
\end{proposition}

\begin{proof} We begin the proof by establishing our Claim.
\begin{claim} \label{conv.int.*}  $\ou^n \xrightarrow{w} u^* \,\, \mbox{ in}\,\, L^2(\overline{\Omega};L^{2}(0,T;\h(\mathcal{O}))),$ i.e.,
\begin{align} \label{X*r.H}
\oE\Big[\int_0^T\langle \ou^n(t,\omega),\phi(t,\omega)\rangle \d t \Big] \rightarrow \oE\Big[\int_0^T \langle u^*(t,\omega),\phi(t,\omega)\rangle \d t \Big]\notag\\
 \quad \forall\,\, \phi\,\, \in L^{2}(\overline{\Omega};L^{2}(0,T;\mathbb{H}^{-1}(\mathcal{O}))).
\end{align} 
\end{claim}
\begin{pf}
\end{pf}
Let $1<s<2.$ Then $\dfrac{2s}{2-s}>2.$ Let $\phi\,\, \in L^{\frac{2s}{2-s}}(\overline{\Omega};L^{2}(0,T;\mathbb{H}^{-1}(\mathcal{O}))).$ Then $\phi(\cdot,\omega)\,\, \in L^{2}(0,T;\mathbb{H}^{-1}(\mathcal{O})))\,\,\overline{P}-a.s.$ \\By \eqref{conv.u*} we have $\ou^n\rightarrow u^*\;\text{in}\;L^2_{w}(0,T;\h(\mathcal{O}))\,\,\overline{P}-a.s.$ Hence,
\begin{align} \label{conv.u.P}
\int_0^T\langle \ou^n(t,\omega),\phi(t,\omega)\rangle \d t  \rightarrow \int_0^T \langle u^*(t,\omega),\phi(t,\omega)\rangle \d t \,\,\overline{P}-\mbox{a.s.}
\end{align}
We note that by H\"{o}lder's inequality with $\frac{1}{2}+\frac{1}{2}=1$ and $\frac{s}{2}+\frac{2-s}{2}=1$ and using \eqref{ineq2} we achieve
\begin{align} \label{unif.int.u}
&\oE\Big[|\int_0^T\langle \ou^n(t,\omega),\phi(t,\omega)\rangle \d t|^s \Big] \leq \oE\Big[\Big(\int_0^T| \langle \ou^n(t,\omega),\phi(t,\omega)\rangle| \d t\Big)^s \Big]\notag\\
&\leq \oE\Big[\Big(\int_0^T| \| \ou^n(t)\|_{\h} \|\phi(t)\|_{\mathbb{H}^{-1}} \d t\Big)^s \Big]\notag\\& \leq \oE\Big[\Big(\int_0^T| \| \ou^n(t)\|^2_{\h} dt\Big)^{\frac{s}{2}} \Big(\int_0^T\|\phi(t)\|^2_{\mathbb{H}^{-1}} \d t \Big)^{\frac{s}{2}} \Big]\notag\\
& \leq \Big[\oE\Big(\int_0^T| \| \ou^n(t)\|^2_{\h} dt\Big)\Big]^{\frac{s}{2}} \Big[\oE \Big(\int_0^T\|\phi(t)\|^2_{\mathbb{H}^{-1}} \d t \Big)^{\frac{s}{2-s}} \Big]^{\frac{2-s}{2}}\notag\\& \leq C\, \Big[\oE \Big(\|\phi\|^{\frac{2s}{2-s}}_{L^2(0,T;\mathbb{H}^{-1})} \Big)\Big]^{\frac{2-s}{2}} \leq C.
\end{align}
Using \eqref{conv.u.P}, \eqref{unif.int.u} and by Vitali Theorem we have
 \begin{align}
\oE\Big[\int_0^T\langle \ou^n(t,\omega),\phi(t,\omega)\rangle \d t \Big] \rightarrow \oE\Big[\int_0^T \langle u^*(t,\omega),\phi(t,\omega)\rangle \d t \Big]\notag\\
 \quad \forall\,\, \phi\,\, \in L^{\frac{2s}{2-s}}(\overline{\Omega};L^{2}(0,T;\mathbb{H}^{-1}(\mathcal{O}))).
\end{align}
As  $1<s<2,$ so $\dfrac{2s}{2-s}>2.$ Hence  $L^{\frac{2s}{2-s}}(\overline{\Omega};L^{2}(0,T;\mathbb{H}^{-1}(\mathcal{O})))$ is dense in $L^{2}(\overline{\Omega};L^{2}(0,T;\mathbb{H}^{-1}(\mathcal{O}))).$
Therefore,
\begin{align*}
\oE\Big[\int_0^T\langle \ou^n(t,\omega),\phi(t,\omega)\rangle \d t \Big] \rightarrow \oE\Big[\int_0^T \langle u^*(t,\omega),\phi(t,\omega)\rangle \d t \Big]\notag\\
 \quad \forall\,\, \phi\,\, \in L^{2}(\overline{\Omega};L^{2}(0,T;\mathbb{H}^{-1}(\mathcal{O}))).
\end{align*} This proves our claim.

Using Claim \ref{conv.int.*} it can be directly observed that $u^* \in L^2(\overline{\Omega};L^2(0,T;\h(\mathcal{O})),$ i.e., \eqref{ener.*} is established. \newline
Now we will prove \eqref{X*r.1}.
By \eqref{ineq1} we have that $\{\ou^n\}_{n \geq 1}$ is uniformly bounded in $L^r(\overline{\Omega};\\L^{\infty}(0,T;\Ll(\mathcal{O})))$ for $r \geq 2.$ Since $L^r(\overline{\Omega};L^{\infty}(0,T;\Ll(\mathcal{O})))$ is isomorphic to the space $(L^\frac{r}{r-1}(\overline{\Omega};\\L^{1}(0,T;\Ll(\mathcal{O}))))^{*},$ by Banach Alaoglu Theorem, there exists a subsequence still denoted by $\{\ou^n\}_{n \geq 1}$ and $\zeta \in L^r(\overline{\Omega};L^{\infty}(0,T;\Ll(\mathcal{O})))$ such that 
$$ \ou^n \xrightarrow{w^{*}} \zeta \quad \mbox{ in}\quad L^r(\overline{\Omega};L^{\infty}(0,T;\Ll(\mathcal{O})))$$ i.e.,
\begin{align} \label{X*r1}
\oE\Big[\int_0^T(\ou^n(t,\omega),\phi(t,\omega))_{\Ll}\d t \Big] \rightarrow \oE\Big[\int_0^T(\zeta(t,\omega),\phi(t,\omega))_{\Ll}\d t \Big]\notag\\
 \quad \forall\,\, \phi\,\, \in L^{\frac{r}{r-1}}(\overline{\Omega};L^{1}(0,T;\Ll(\mathcal{O}))).
\end{align}
Employing Claim \ref{conv.int.*} and Gelfand triple $\h(\mathcal{O}) \subset \Ll(\mathcal{O}) \subset \mathbb{H}^{-1}(\mathcal{O})$ we have
\begin{align} \label{X*r2}
\oE\Big[\int_0^T(\ou^n(t,\omega),\phi(t,\omega))_{\Ll}\d t \Big] \rightarrow \oE\Big[\int_0^T(u^*(t,\omega),\phi(t,\omega))_{\Ll}\d t \Big]\notag\\
 \quad \forall\,\, \phi\,\, \in L^{2}(\overline{\Omega};L^{2}(0,T;\Ll(\mathcal{O}))).
\end{align} 
Again for $r \geq 2,\,L^{2}(\overline{\Omega};L^{2}(0,T;\Ll(\mathcal{O}))) $ is dense subspace of $L^{\frac{r}{r-1}}(\overline{\Omega};L^{1}(0,T;\Ll(\mathcal{O}))).$ 
 \eqref{X*r1} and \eqref{X*r2} jointly produces
 \begin{align*}
\oE\Big[\int_0^T(\zeta(t,\omega),\phi(t,\omega))_{\Ll}\d t \Big] =\oE\Big[\int_0^T(u^*(t,\omega),\phi(t,\omega))_{\Ll}\d t \Big]\notag\\
 \quad \forall\,\, \phi\,\, \in L^{2}(\overline{\Omega};L^{2}(0,T;\Ll(\mathcal{O}))).
 \end{align*}
\gr{Thus we have, $u^*(t,\omega)=\zeta(t,\omega)$ for almost every $t \in [0,T]$ and $\omega \in \overline{\Omega}.$ Since $\zeta \in L^{r}(\overline{\Omega};L^{\infty}(0,T;\Ll(\mathcal{O}))),$ we infer that $u^* \in L^{r}(\overline{\Omega};L^{\infty}(0,T;\Ll(\mathcal{O}))),$ i.e.,
$$\oE \Big[\sup_{s \in [0,T]}\|u^*(s)\|_{\Ll}^r \Big]\leq C_r$$ for some constant $C_r$ (depending on r).}
\end{proof}}

\textbf{Step II}:
\gr{\begin{lemma} \label{conv.each} For all $v\in\h(\mathcal{O})$ and for all $w\in \Hh(\mathcal{O}),$
\begin{itemize}
\item[(i)] $\lim_{n\rightarrow\infty}\overline{\E}[\int_0^T |(\ou^n(t)-u^*(t),v)_{\Ll}| dt]=0.$
\item[(ii)] $\lim_{n\rightarrow\infty}\overline{\E}[|(\ou^n(0)-u^*(0),v)_{\Ll}|]=0.$
\item[(ii)] $\lim_{n\rightarrow\infty}\overline{\E}[|(\overline{U}^n-U^*,v)_{\Ll}|]=0.$
\item[(iv)] $\lim_{n\rightarrow\infty}\overline{\E}[\int_0^T |\int_0^t (A\ou^n(s)-Au^*(s),v)_{\Ll} ds|dt]=0.$
\item[(v)] $\lim_{n\rightarrow\infty} \overline{\E}[\int_0^T|\int_0^t (B(\ou^n(s))-B(u^*(s)),v)_{\Ll} ds|dt]=0.$
\item[(vi)] \gr{$\lim_{n\rightarrow\infty}\overline{\E}[\int_0^T|\int_0^t\langle g\nabla(\oz^n(s)-z^*(s)),v \rangle ds|dt]=0.$}
\item[(vii)] $\lim_{n\rightarrow\infty}\overline{\E}[\int_0^T\|(\int_0^t (\s^n(s,\ou^n(s))-\s(s,u^*(s)))dW^*(s),v)_{\Ll}\|_{L_Q(\mathbb{L}^2;\mathbb{R})}^2]dt=0.$
\item[(viii)] $\lim_{n\rightarrow\infty}\overline{\E}[\int_0^T|\int_0^t \int_Z (H^n(\ou^n(s),z)-H(u^*(s),z),v)_{\Ll} \lambda(dz)ds|^2dt]=0.$
\item[(ix)] $\lim_{n\rightarrow\infty}\overline{\E}[\int_0^T|\int_0^t \int_Z(H^n(\ou^n(s-),z)-H(u^*(s-),z),v)_{\Ll}\tilde{N}^*(ds,dz)|^2 dt]=0.$
\item[(x)] $\lim_{n\rightarrow\infty}\overline{\E}[\int_0^T |(\oz^n(t)-z^*(t),w)_{L^2}| dt]=0.$
\item[(xi)] $\lim_{n\rightarrow\infty}\overline{\E}[|(\oz^n_0-z^*_0,w)_{L^2}|]=0.$
\item[(xii)] $\lim_{n\rightarrow\infty}\overline{\E}[\int_0^T |\int_0^t (Div(h\ou^n(s))-Div(hu^*(s)),w)_{L^2} ds|dt]=0.$
\end{itemize}
\end{lemma}}

\begin{proof} Let $v \in \h(\mathcal{O})$ and $w \in \Hh(\mathcal{O})$ be fixed.\begin{itemize}
\item[(i)]  We have  $\ou^n\rightarrow u^*\;\,\text{in}\; D([0,T];\Ll_w(\mathcal{O})),\,\, \overline{P}-\mbox{a.s.}$
 i.e., $(\ou^n(t)-u^*(t), v)_{\Ll} \rightarrow 0 \,\,\text{in}\,\, D([0,T];\mathbb{R}),\,\, \overline{P}-\textit{a.s.}$
Hence, in particular for almost all $t \in[0,T],$
\gr{\begin{align} \label{conv.1}
&\lim_{n \rightarrow \infty} (\ou^n(t),v)_{\Ll}=(u^*(t),v)_{\Ll}  \quad \overline{P}-\mbox{a.s.}
\end{align}
Employing H\"{o}lder's inequality  and \eqref{ineq1} we obtain
\begin{align} \label{uni.1}
&\overline{\mathbb{E}}\Big[\mathlarger{\int_0^T}|(\ou^n(t)-u^*(t), v)_{\Ll}|^2 d t\Big] \leq \overline{\mathbb{E}}\Big[\mathlarger{\int_0^T}|\|\ou^n(t)-u^*(t)\|_{\Ll}^2 \|v\|_{\Ll}^2 d t\Big]\notag \\ & \leq \|v\|_{\Ll}^2\,\overline{\mathbb{E}}\Big[\mathlarger{\int_0^T}|\|\ou^n(t)-u^*(t)\|_{\Ll}^2  d t\Big] \leq  2\|v\|_{\Ll}^2\,\overline{\mathbb{E}}\Big[\mathlarger{\int_0^T}(\|\ou^n(t)\|^2_{\Ll}+\|u^*(t)\|_{\Ll}^2) d t\Big] \notag\\
& \leq 2cT \|v\|_{\h}^2\,\overline{\mathbb{E}}\Big[\sup_{0 \leq s \leq T}(\|\ou^n(s)\|^2_{\Ll}+\|u^*(s)\|_{\Ll}^2)\Big] \leq C.
\end{align}
Hence, by employing \eqref{conv.1}, \eqref{uni.1} and by Vitali Theorem we have assertion (i).}
\item[(ii)] 
We have $\ou^n\rightarrow u^*$ in $D([0,T];\Ll_w(\mathcal{O})),\,\, \overline{P}$-a.s. and $u^*$ is right continuous at $t=0,$ we infer that 
\gr{\begin{align} \label{conv.u_n0}
(\ou^n(0),v)_{\Ll} \rightarrow (u^*(0),v)_{\Ll}\,\,\overline{P}-\mbox{a.s.}
\end{align}
Using \eqref{ineq1},
\begin{align} \label{uni.u_n0}
\overline{\E}[|(\ou^n(0),v)_{\Ll}|^2] \leq \|v\|_{\Ll}^2 \overline{\E}[\|\ou^n(0)\|_{\Ll}^2] \leq c\|v\|_{\h}^2 \overline{\E}[\sup_{0\leq s\leq T}\|\ou^n(s)\|_{\Ll}^2] \leq C,
\end{align}
\eqref{conv.u_n0}, \eqref{uni.u_n0} and Vitali Theorem gives   
\begin{equation*}
\gr{\lim_{n\rightarrow\infty}\overline{\E}[|(\ou^n(0)-u^*(0),v)_{\Ll}|]=0.}
\end{equation*}}

\gr{
\item[(iii)]
Since $\overline{U}^n\rightarrow U^*$ in $\Ll_w(\mathcal{O})$,  $\overline{P}$-a.s and as $v \in \h(\mathcal{O}) \subset \Ll(\mathcal{O})$ we have
\begin{equation*}
(\overline{U}^n,v)_{\Ll}\rightarrow (U^*,v)_{\Ll},\quad \overline{P}-\mbox{a.s.}
\end{equation*}
Since $\overline{\E}[|(\overline{U}^n,v)_{\Ll}|^2]\leq \|v\|^2_{\Ll} \overline{\E}\|\overline{U}^n\|_{\Ll}^2 \leq C$, by the Vitali theorem
\begin{equation*}
\lim_{n\rightarrow\infty}\overline{\E}[|(\overline{U}^n-U^*,v)_{\Ll}|]=0.
\end{equation*}}
\gr{
\item[(iv)] We see from \eqref{prop.A} that $u \mapsto \int_0^t (Au(s),v)_{\Ll}ds$ from $L^2(0,T;\h(\mathcal{O}))$ to $\mathbb{R}$ is linear and continuous.
Since $\ou^n\rightarrow u^*$ in $L^2_w(0,T;\h(\mathcal{O}))$, $\overline{P}$-a.s. we have,
\begin{align} \label{conv.A1}
\lim_{n\rightarrow\infty}\int_0^t (A\ou^n(s),v)_{\Ll}ds=\int_0^t (Au^*(s),v)_{\Ll}ds \quad \overline{P}-\mbox{a.s.}
\end{align}
Hence by H\"{o}lder's inequality and \eqref{ineq2}, we achieve for all $t \in [0,T]$ and $n \in \mathbb{N},$
\begin{align} \label{conv.A2}
\overline{\E}[|\int_0^t (A\ou^n(s),v)_{\Ll} ds|^2]
\leq C_1\|v\|^2_{\h}\overline{\E}[\int_0^t \|\ou^n(s)\|^2_{\h}ds] \leq c\,C_{2(2)}
\end{align} for some constant $c>0.$
Therefore by \eqref{conv.A1}, \eqref{conv.A2} and Vitali theorem, for all $t\in [0,T]$
\begin{equation*}
\lim_{n\rightarrow\infty}\overline{\E}[|\int_0^t (A\ou^n(s)-Au^*(s),v)_{\Ll} ds|]=0.
\end{equation*}
Hence by the dominated convergence theorem
\begin{equation*}
\lim_{n\rightarrow\infty}\int_0^T \overline{\E}[|\int_0^t (A\ou^n(s)-Au^*(s),v)_{\Ll} ds|]=0.
\end{equation*}}


\item[(v)] First we note that $\overline{P}$-a.s. $\|w_0\|_{L^4(0,T;\mathbb{L}^4)}\leq C$ and using \eqref{conv.u*} we have $\overline{P}$-a.s. $\|\ou^n\|_{L^2(0,T;\h)} \leq C$ and $\|u^*\|_{L^2(0,T;\h)} \leq C$ for some $C>0.$ Exploiting \eqref{eq4}, \eqref{e1}, H\"olders inequality with $\frac{1}{2}+\frac{1}{4}+\frac{1}{4}=1,$ and then using $\ou^n \rightarrow u^*$ in $L^2(0,T;\Ll(\mathcal{O}))\,\,\overline{P}$-a.s., we infer that $\overline{P}$-a.s.
\begin{align} \label{eq.B1}
&\Big|\int_0^t (B(\ou^n(s))-B(u^*(s)),v)_{\Ll} ds\Big|\notag\\&=\gamma \Big|\int_0^t \Big(|\ou^n(s) +w_0(s)|(\ou^n(s) +w_0(s))-|u^*(s) +w_0(s)|(u^*(s) +w_0(s)),v\Big)_{\Ll} ds\Big|\notag\\
&=\gamma \Big|\int_0^t \Big(|\ou^n(s) +w_0(s)|(\ou^n(s) +w_0(s) -(u^*(s) +w_0(s))),v\Big)_{\Ll} ds\notag\\& \quad +\int_0^t \Big( (|\ou^n(s) +w_0(s)|- |u^*(s) +w_0(s)|)(u^*(s) +w_0(s)),v\Big)_{\Ll} ds\Big|\notag\\
&\leq \gamma \int_0^t \Big|\Big(|\ou^n(s) +w_0(s)|(\ou^n(s)-u^*(s)),v\Big)_{\Ll} ds\Big| \notag\\ &\quad + \gamma\int_0^t \Big| \Big( (|\ou^n(s) +w_0(s)|- |u^*(s) +w_0(s)|)(u^*(s) +w_0(s)),v\Big)_{\Ll} ds\Big|\notag\\
& \leq \gamma \int_0^t \|\ou^n(s) +w_0(s)\|_{\mathbb{L}^4} \|\ou^n(s)-u^*(s)\|_{\Ll} \|v\|_{\mathbb{L}^4} ds\notag\\&\quad + \gamma \int_0^t \||\ou^n(s) +w_0(s)|- |u^*(s) +w_0(s)|\|_{\Ll}\|u^*(s) +w_0(s)\|_{\mathbb{L}^4} \|v\|_{\mathbb{L}^4} ds \notag\\
& \leq 2 \gamma \Big(\int_0^t \|\ou^n(s)-u^*(s)\|_{\Ll}\|v\|_{\mathbb{L}^4} (\|\ou^n(s) +w_0(s)\|_{\mathbb{L}^4}+\|u^*(s) +w_0(s)\|_{\mathbb{L}^4}) ds \Big)\notag \\
& \leq 2 \gamma \|v\|_{\mathbb{L}^4} \Big[\Big(\int_0^T \|\ou^n(s)-u^*(s)\|_{\Ll}\|\ou^n(s) +w_0(s)\|_{\mathbb{L}^4}ds\Big)\notag \\&\quad + \Big(\int_0^T \|\ou^n(s)-u^*(s)\|_{\Ll}\|u^*(s) +w_0(s)\|_{\mathbb{L}^4}ds\Big)\Big]\notag \\
& \leq c \|v\|_{\h} \Big[\|\ou^n +w_0\|_{L^2(0,T;\mathbb{L}^4)} + \|u^* +w_0\|_{L^2(0,T;\mathbb{L}^4)} \Big]\notag\\& \quad\quad \times \|\ou^n-u^*\|_{L^2(0,T;\Ll)}\notag\\
& \leq c \|v\|_{\h} \Big[\|\ou^n\|_{L^2(0,T;\mathbb{L}^4)} +2\|w_0\|_{L^2(0,T;\mathbb{L}^4)} + \|u^*\|_{L^2(0,T;\mathbb{L}^4)} \Big] \|\ou^n-u^*\|_{L^2(0,T;\Ll)}\notag\\
&\leq c \|v\|_{\h} \Big[\|\ou^n\|_{L^2(0,T;\h)} +2\|w_0\|_{L^4(0,T;\mathbb{L}^4)} + \|u^*\|_{L^2(0,T;\h)} \Big] \|\ou^n-u^*\|_{L^2(0,T;\Ll)}\notag\\
& \leq C \|\ou^n-u^*\|_{L^2(0,T;\Ll)} \rightarrow 0 \quad \mbox{as} \quad n \rightarrow \infty.
\end{align}
Employing \eqref{ineq1}, \eqref{ineq2}, \eqref{e2}, Lemma \ref{lemw}, Proposition \ref{moment}, we observe that for every $t\in [0,T]$ and $1\leq r < 2,$ and for every $n \in \mathbb{N},$
\begin{align} \label{eq.B2}
&\overline{\E}[|\int_0^t(B(\ou^n(s)),v)_{\Ll} ds|^r]\leq c\,\|v\|^{r}_{\Ll} T^{r-1}\,\overline{\E}[\int_0^t \|B(\ou^n(s))\|^{r}_{\Ll}ds]\notag\\
&\leq  c\,C_2\|v\|^{r}_{\Ll}\,T^{r-1} \overline{\E}[\int_0^t \|\ou^n(s) +w^0(s)\|^{2r}_{\mathbb{L}^4} ds]\notag\\
&\leq C \|v\|^{r}_{\Ll}\,T^{r-1} \left\{\overline{\E}\left[\int_0^T \left(\|\overline{u}^n(s)\|_{\mathbb{L}^4}^{2r} + \|w^0(s)\|_{\mathbb{L}^4}^{2r} \right) ds\right]\right\} \leq C,
\end{align}
since, 
\begin{align*}
&\overline{\E}[\int_0^T \|\overline{u}^n(s)\|_{\mathbb{L}^4}^{2r} ds] \leq 2^{\frac{r}{2}}\oE[\int_0^T \|\overline{u}^n(s)\|_{\Ll}^r\, \|\overline{u}^n(s)\|_{\h}^{r} ds]\\ 
&\leq C_r \left(\oE[ \int_0^T \|\overline{u}^n(s)\|_{\h}^{2}ds] +\oE[ T\, \sup_{0 \leq s \leq T} \|\overline{u}^n(s)\|_{\Ll}^{\frac{2r}{2-r}}]\right) \leq C,
\end{align*} 
provided $\frac{2r}{2-r}\geq 2.$
Hence \eqref{eq.B1},\eqref{eq.B2} and by Vitali theorem, for all $t\in [0,T],$ we have
\begin{equation*}
\lim_{n\rightarrow\infty}\overline{\E}[|\int_0^t (B(\ou^n(s))-B(u^*(s)),v)_{\Ll} ds|]=0.
\end{equation*}
Hence by the dominated convergence theorem
\begin{equation*}
\lim_{n\rightarrow\infty}\int_0^T \overline{\E}[|\int_0^t (B(\ou^n(s))-B(u^*(s)),v)_{\Ll} ds|]dt=0.
\end{equation*}
\gr{\item[(vi)]
Since $\oz^n\rightarrow z^*$ in $L^2(0,T;L^2(\mathcal{O}))\,\,\overline{P}-$a.s., so for any $\tilde{\phi} \in L^2(0,T;\h(\mathcal{O})),$ using H\"older inequality we have $\overline{P}-$a.s.,
\begin{align} \label{esti.z.1.1}
&|\int_0^T \langle g\nabla(\oz^n(s)-z^*(s)),\tilde{\phi}(s) \rangle ds| =|-\int_0^T (g(\oz^n(s)-z^*(s)),\nabla\cdot \tilde{\phi}(s))_{L^2} ds|\notag\\ &\leq \int_0^T |(g(\oz^n(s)-z^*(s)),\nabla\cdot \tilde{\phi}(s))_{L^2}| ds \notag\\& \leq g \Big(\int_0^T \|\oz^n(s)-z^*(s)\|_{L^2} ds\Big)^{1/2} \Big(\int_0^T \|\nabla\cdot \tilde{\phi}(s)\|^2_{L^2} ds \Big)^{1/2}
\rightarrow 0 \quad \mbox{as} \quad n \rightarrow \infty.
\end{align}
In particular, we choose $\tilde{\phi}(s)=\chi_{(0,t)}(s)v$, where $v$ is a fixed element of $\mathbb{H}^1_0(\mathcal{O})$. Hence, from \eqref{esti.z.1.1}, we have for all $t\in[0,T]\,\, \overline{P}-$a.s.
\begin{align} \label{esti.z.1}
&\lim_{n\rightarrow\infty}\int_0^t \langle g\nabla(\oz^n(s)-z^*(s)),v \rangle ds =0.
\end{align}
Now by \eqref{ineq3}
\begin{align} \label{esti.z.2}
&\overline{\E}[|\int_0^t \langle g\nabla(\oz^n(s),v \rangle ds|^2] \leq g\overline{\E}[\int_0^t |(\oz^n(s),\nabla\cdot v)_{L^2}|^2 ds]\notag\\
& \leq g\|v\|_{\h}^2\overline{\E}[\int_0^t \|\oz^n(s)\|_{L^2}^2 ds]
 \leq gT\|v\|_{\h}^2\overline{\E}[\sup_{0\leq s\leq T} \|\oz^n(s)\|_{L^2}^2]\notag\\
& \leq C.
\end{align}
Hence from \eqref{esti.z.1}, \eqref{esti.z.2} and applying Vitali's theorem we obtain, for all $t\in [0,T],$ 
\begin{equation} \label{esti.z.3}
\lim_{n\rightarrow\infty}\overline{\E}[|\int_0^t \langle g\nabla(\oz^n(s)-z^*(s)),v \rangle ds|]=0.
\end{equation}
Again using \eqref{ineq3} and employing Cauchy-Schwartz inequality twice, we achieve 
\begin{align} \label{esti.z.4}
&|\overline{\E}[\int_0^t \langle g\nabla \oz^n(s),v \rangle ds]| \leq g\overline{\E}[\int_0^t |(\oz^n(s),\nabla\cdot v)_{L^2}| ds]\notag\\
& \leq g\|v\|_{\h} \overline{\E}[\int_0^t \|\oz^n(s)\|_{L^2} ds] \leq g\|v\|_{\h} T^{\frac{1}{2}}\Big[\overline{\E}\int_0^t \|\oz^n(s)\|^2_{L^2} ds\Big]^{\frac{1}{2}} \notag\\
 &\leq g\|v\|_{\h} T^{\frac{3}{2}} \overline{\E}[\sup_{0\leq s\leq T} \|\oz^n(s)\|_{L^2}^2]^{\frac{1}{2}}\leq C.
\end{align}
Hence by \eqref{esti.z.3}, \eqref{esti.z.4} and by the Dominated Convergence theorem we have
\begin{equation*}
\lim_{n\rightarrow\infty}\int_0^T\overline{\E}[|\int_0^t \langle g\nabla(\oz^n(s)-z^*(s)),v \rangle ds|]dt=0.
\end{equation*}}

\item[(vii)]
\gr{Using the Assumption \ref{Hyp} H.3, since $\ou^n\rightarrow u^*$ in $L^2(0,T;\Ll(\mathcal{O}))$ $\overline{P}$-a.s., we have $\overline{P}-$a.s.
\begin{align} \label{esti.sig}
&\lim_{n\rightarrow\infty}\int_0^t \|(\s(s,\ou^n(s))-\s(s,u^*(s)),v)_{\Ll}\|_{L_Q(\mathbb{L}^2;\mathbb{R})}^2 ds \notag\\
&\leq \lim_{n\rightarrow\infty}\|v\|_{\Ll}^2 \int_0^t \|\s(s,\ou^n(s))-\s(s,u^*(s))\|_{L_Q(\mathbb{L}^2;\mathbb{R})}^2 ds\notag\\
&\leq L\|v\|_{\Ll}^2 \lim_{n\rightarrow\infty}  \int_0^T \|\ou^n(s)-u^*(s)\|_{\Ll}^2 ds
 =0.
\end{align}
Using Assumption \ref{Hyp} H.2, \eqref{X*r.1} (in Proposition \ref{X*r}) and \eqref{ineq1} we observe that for every $t\in [0,T]$ and $r > 1$ and for every $n \in \mathbb{N},$
\begin{IEEEeqnarray}{lrl} \label{ener1.sig}
\overline{\E}[|\int_0^t \|(\s(s,\ou^n(s))-\s(s,u^*(s)),v)_{\Ll}\|_{L_Q(\mathbb{L}^2;\mathbb{R})}^2 ds|^r]\nonumber \\
\leq \|v\|^{2r}_{\Ll}\,T^{r-1}\,2^{2r-1} \overline{\E}[\int_0^t \big( \|\s(s,\ou^n(s))\|^{2r}_{L_Q(\mathbb{L}^2;\Ll)}+\|\s(s,u^*(s))\|^{2r}_{L_Q(\mathbb{L}^2;\Ll)}\big) ds]\nonumber \\
\leq K\|v\|^{2r}_{\Ll}\,T^{r-1}\,2^{2r-1} \overline{\E}[\int_0^t (2+\|\ou^n(s)\|^{2r}_{\Ll}+\| u^*(s)\|^{2r}_{\Ll}) ds]\nonumber \\
\leq c_r \overline{\E}[ (2+\sup_{0\leq s\leq T}\|\ou^n(s)\|^{2r}_{\Ll}+\sup_{0\leq s\leq T}\| u^*(s)\|^{2r}_{\Ll})]\leq\tilde{c}_r,
\end{IEEEeqnarray} for some constant $\tilde{c}_r>0$ (depending upon $r$).
Thus by employing \eqref{esti.sig} and \eqref{ener1.sig} and Vitali's theorem we obtain
\begin{equation} \label{conv.sig.n}
\lim_{n\rightarrow\infty}\overline{\E}[\int_0^t \|(\s(s,\ou^n(s))-\s(s,u^*(s)),v)_{\Ll}\|_{L_Q(\mathbb{L}^2;\mathbb{R})}^2 ds]=0 \quad \forall v \in \h(\mathcal{O}).
\end{equation}
For every $v \in \h(\mathcal{O})$ and every $s \in [0,T]$ we have
\begin{align*}
&( \s^n(s,\ou^n(s))-\s(s,u^*(s)),v )_{\Ll}=  ( \s(s,\ou^n(s)),P_n v)_{{\Ll}}-(\s(s,u^*(s)),v )_{\Ll}\\
&= ( \s(s,\ou^n(s)),P_n v-v)_{{\Ll}}+(\s(s,\ou^n(s))-\s(s,u^*(s)),v )_{\Ll}\\
& \leq \|\s(s,\ou^n(s))\|_{L_Q(\mathbb{L}^2;\Ll)} \|P_n v-v \|_{\Ll}+ (\s(s,\ou^n(s))-\s(s,u^*(s)),v )_{\Ll}\\
& \leq c\|\s(s,\ou^n(s))\|_{L_Q(\mathbb{L}^2;\Ll)} \|P_n v-v \|_{\h}+ (\s(s,\ou^n(s))-\s(s,u^*(s)),v )_{\Ll}.
\end{align*}
Then by Assumption \ref{Hyp} H.2 and by \eqref{ineq1} we obtain
\begin{align*}
&\oE \Big[\int_0^t \|( \s^n(s,\ou^n(s))-\s(s,u^*(s)),v )_{\Ll}\|^2_{L_Q(\mathbb{L}^2;\mathbb{R})} ds   \Big]\\
&\leq 2c\,K_2 \|P_n v-v\|^2_{\h} \,\oE \Big[\int_0^T (1+\|\ou^n(s)\|_{\Ll}^{2}) ds \Big]\\
&\quad+ \oE \Big[\int_0^t \|( \s(s,\ou^n(s))-(\s(s,u^*(s)),v )_{\Ll}\|^2_{L_Q(\mathbb{L}^2;\mathbb{R})} ds   \Big]\\
& \leq 2 \tilde{c} \|P_n v -v\|_{\h}^2 + \oE \Big[\int_0^t \|( \s(s,\ou^n(s))-(\s(s,u^*(s)),v )_{\Ll}\|^2_{L_Q(\mathbb{L}^2;\mathbb{R})} ds\Big].
\end{align*}
Thus by Lemma \ref{Pn.lem} and by  \eqref{conv.sig.n} we conclude that 
\begin{align} \label{ener1.sig.11}
\lim_{n \rightarrow \infty} \oE \Big[\int_0^t \|( \s^n(s,\ou^n(s))-\s(s,u^*(s)),v )_{\Ll}\|^2_{L_Q(\mathbb{L}^2;\mathbb{R})} ds   \Big]=0 \quad \forall \,v \in \h(\mathcal{O}).
\end{align}
Using \ada{$\ow^n=W^*,$} It\^o isometry, \eqref{ener1.sig.11} and $\forall \,v \in \h(\mathcal{O})$ we have
\begin{align}\label{ito.W}
&\lim_{n\rightarrow\infty}\overline{\E}[\|(\int_0^t (\s^n(s,\ou^n(s))d\ow^n(s)-\s(s,u^*(s)))dW^*(s),v)_{\Ll}\|_{L_Q(\mathbb{L}^2;\mathbb{R})}^2]\notag\\
=&\lim_{n\rightarrow\infty}\overline{\E}[\|(\int_0^t (\s^n(s,\ou^n(s))-\s(s,u^*(s)))dW^*(s),v)_{\Ll}\|_{L_Q(\mathbb{L}^2;\mathbb{R})}^2]\notag\\ 
&=\lim_{n\rightarrow\infty}\overline{\E}[\int_0^t \|(\s^n(s,\ou^n(s))-\s(s,u^*(s)),v)_{\Ll}\|_{L_Q(\mathbb{L}^2;\mathbb{R})}^2 ds] =0
\end{align} }
\gr{Again by the It\^o isometry, \eqref{eq.sig.n}, Assumption \ref{Hyp} and \eqref{ineq1} we have for all $t \in [0,T]$ and all $n \in \mathbb{N},$
\begin{align} \label{conv.W}
&\oE \Big[\int_0^t \|([\s^n(s,\ou^n(s))-\s(s,u^*(s))]dW^*(s),v )_{\Ll} \|^2_{L_Q(\mathbb{L}^2;\mathbb{R})} ds   \Big] \notag\\
&=\oE \Big[\int_0^t \|( \s^n(s,\ou^n(s))-\s(s,u^*(s)),v )_{\Ll}\|^2_{L_Q(\mathbb{L}^2;\mathbb{R})} ds \Big] \notag\\
&\leq 2cK_2\,T^{2}\,\|v\|_{\h}^{2} \,\oE \Big[ (2+\sup_{0\leq s\leq T}\|\ou^n(s)\|_{\Ll}^{2}+\sup_{0\leq s\leq T}\|u^*(s)\|_{\Ll}^{2})\Big]\leq \tilde{C_2}.
\end{align} 
Hence by \eqref{ito.W} and \eqref{conv.W} and the Dominated Convergence Theorem, we have the assertion (vii).}
\gr{
\item[(viii)] 
Using Assumption \ref{Hyp} H.3 and that $\ou^n\rightarrow u^*$ in $L^2(0,T;\Ll(\mathcal{O}))$ $\overline{P}$-a.s., we have $\overline{P}-$a.s.
\begin{align} \label{esti.H.lim}
&\int_0^t \int_Z |(H(\ou^n(s),z)-H(u^*(s),z),v)_{\Ll}|^2 \lambda(dz)ds \notag \\
&\leq \|v\|_{\Ll}^2 \int_0^t\int_Z \|(H(\ou^n(s),z)-H(u^*(s),z)\|_{\Ll}^2 \lambda(dz)ds \notag \\
&\leq  L\|v\|_{\Ll}^2 \int_0^T \|\ou^n(s)-u^*(s)\|_{\Ll}^2 ds \rightarrow 0 \,\,\mbox{as}\,\, {n\rightarrow\infty}.
\end{align}
Furthermore, using Assumption \ref{Hyp} H.2, \eqref{X*r.1} (in Proposition \ref{X*r}) and \eqref{ineq1}, for every $t\in [0,T],\,r > 1$ and $n \in \mathbb{N},$ we have the following inequality 
\begin{align} \label{esti.H.uni}
&\overline{\E}\Big[\Big|\int_0^t \int_Z \Big|(H(\ou^n(s),z)-H(u^*(s),z),v)_{\Ll}|^2 \lambda(dz)ds \Big|^r\Big] \notag\\
&\leq 2^r \|v\|_{\Ll}^{2r}\, \overline{\E}\Big[\Big|\int_0^t \int_Z \Big(\|H(\ou^n(s),z)\|_{\Ll}^{2}+\| H(u^*(s),z)\|_{\Ll}^{2} \Big)\lambda(dz)ds\Big|^r\Big] \notag\\
&\leq 2^r K^{2r}\,\|v\|_{\Ll}^{2r}\, \overline{\E}\Big[\Big|\int_0^t \Big(2+\|\ou^n(s)\|_{\Ll}^{2}+\| u^*(s)\|_{\Ll}^2\Big) ds \Big|^r\Big]\notag\\
&\leq 2^r K^{2r}\,\|v\|_{\Ll}^{2r}\,T^r \,\Big(1+ \overline{\E}\Big[\sup_{0\leq s\leq T}\|\ou^n(s)\|_{\Ll}^{2r}\Big]+\overline{\E}\Big[\sup_{0\leq s\leq T}\| u^*(s)\|_{\Ll}^{2r})\Big]\Big)\notag\\
&\leq \tilde{c}_r
\end{align} for some constant $\tilde{c}_r>0$ (depending upon $r$).
Thus by \eqref{esti.H.lim} and \eqref{esti.H.uni} and by Vitali's theorem for every $t\in [0,T],\,\,\forall\,\, v \in \h(\mathcal{O}),$
\begin{equation}
\lim_{n\rightarrow\infty}\overline{\E}[\int_0^t \int_Z |(H(\ou^n(s),z)-H(u^*(s),z),v)_{\Ll}|^2 \lambda(dz)ds]=0.
\end{equation}
Since the restriction of $P_n$ to $\Ll(\mathcal{O}),$ is the $(\cdot,\cdot)_{\Ll}-$orthogonal projection onto $\Ll_n(\mathcal{O}),$ we conclude $\forall\,\, v \in \Ll(\mathcal{O}),$
\begin{align} \label{eq.H}
&\lim_{n\rightarrow\infty}\overline{\E}[\int_0^t \int_Z |(H^n(\ou^n(s),z)-H(u^*(s),z),v)_{\Ll}|^2 \lambda(dz)ds]=0.
\end{align}
Since $\h(\mathcal{O}) \subset \Ll(\mathcal{O}),$ \eqref{eq.H} holds for all $v \in \h(\mathcal{O}).$
\par\noindent
Moreover, Assumption \ref{Hyp} and \eqref{ineq1} yield the following inequality
\begin{align} \label{ener4.g}
&\oE\Big[\mathlarger{\int_0^t} \mathlarger{\int_Z} \Big| (H^n(\ou^n(s),z)-H(u^*(s),z),v)_{\Ll} \Big|^2 \lambda(\d z)\d s \Big]
\leq \tilde{C}_2.
\end{align}
Now \eqref{eq.H}, \eqref{ener4.g} and the Dominated Convergence Theorem assures assertion (viii).

\item[(ix)]
Employing $\on^n=N^*$, It\^o-L\'evy isometry and \eqref{eq.H}, we have $\forall\, v \in \h(\mathcal{O})$ 
\begin{align}\label{0006}
&\lim_{n\rightarrow\infty}\overline{\E}\left[\left|\int_0^t \int_Z(H^n(\ou^n(s-),z)-H(u^*(s-),z),v)_{\Ll}\tilde{\on^n}(ds,dz)\right|^2\right]\nonumber\\
&=\lim_{n\rightarrow\infty}\overline{\E}\left[\left|\int_0^t \int_Z(H^n(\ou^n(s-),z)-H(u^*(s-),z),v)_{\Ll}\tilde{N}^*(ds,dz)\right|^2\right]\nonumber\\
&=\lim_{n\rightarrow\infty}\overline{\E}[\int_0^t \int_Z |(H^n(\ou^n(s),z)-H(u^*(s),z),v)_{\Ll}|^2 \lambda(dz)ds] = 0. 
\end{align}
Moreover from \eqref{ener4.g}, we have $\forall\, v \in \h(\mathcal{O})$ 
\begin{align}\label{0007}
&\oE \left[\left|\mathlarger{\int_0^t} \mathlarger{\int_Z} (H^n(\ou^n(s-),z)-H(u^*(s-),z),v)_{\Ll}\tilde{N}^*(\d s,\d z) \right|^2\right]\leq 
\tilde{C}_2.
\end{align}
Hence by \eqref{0006}, \eqref{0007} and the Dominated convergence theorem we obtain (ix).}

\gr{\item[(x)] 
Since $\oz^n\rightarrow z^*$ in $C([0,T];H^{-1}(\mathcal{O}))\,\,\overline{P}-$a.s., we see that $\overline{P}$-a.s. and for all $t\in[0,T]$
\begin{align} \label{eq.z.1}
&|\langle \oz^n(t)-z^*(t),w\rangle | \leq \|\oz^n(t)-z^*(t)\|_{H^{-1}} \|w\|_{\Hh}\notag\\& \leq \|w\|_{\Hh} \sup_{t \in [0,T]}\|\oz^n(t)-z^*(t)\|_{H^{-1}} \rightarrow 0 \quad \mbox{as} \quad n \rightarrow \infty.
\end{align}
Also from \eqref{ineq3} we get
\begin{align} \label{eq.z.2}
\oE[\int_0^T |\langle \oz^n(t),w \rangle|^2 dt]&\leq \|w\|_{\Hh}^2 \oE[\int_0^T \|\oz^n(t)\|_{H^{-1}}^2 dt]\notag\\
&\leq c\, \|w\|_{\Hh}^2\,T\,\oE[\sup_{0\leq t\leq T} \|\oz^n(t)\|_{L^2}^2 ] \leq C.
\end{align}
By \eqref{eq.z.1}, \eqref{eq.z.2} and then by the Vitali theorem we have (x).}


\gr{
\item[(xi)]
Since $\oz^n\rightarrow z^*$ in $C([0,T];H^{-1}(\mathcal{O}))\,\,\overline{P}-$a.s., using similar arguments as in \eqref{eq.z.1} in particular for $t=0$ we have $\overline{P}$-a.s.
\begin{equation} \label{est.z0.1}
\langle \oz^n_0,w\rangle \rightarrow \langle z^*_0,w \rangle.
\end{equation}
Hence  by \eqref{ineq3} we achieve
\begin{align} \label{est.z0.2}
\oE[|\langle \oz^n_0,w \rangle|^2] \leq c\, \|w\|^2_{\Hh} \oE[\|\oz_0^n\|_{L^2}^2] \leq c\, \|w\|^2_{\Hh} \oE[\|\oz_0\|_{L^2}^2] \leq C.
\end{align}
Then by \eqref{est.z0.1} and \eqref{est.z0.2} and using Vitali theorem
\begin{equation*}
\lim_{n\rightarrow\infty}\overline{\E}[|(\oz^n_0-z^*_0,w)_{L^2}|]=0.
\end{equation*}}
\gr{
\item[(xii)]
Since $\ou^n\rightarrow u^*$ in $L_{w}^2(0,T;\h(\mathcal{O}))\,\,\overline{P}-$a.s., so for any $\tilde{\phi} \in L^2(0,T;L^2(\mathcal{O}))$ we have $\overline{P}-$a.s.,
\begin{align} \label{esti.z}
\lim_{n\rightarrow\infty}\int_0^T (Div(h\ou^n(s)),\tilde{\phi}(s))_{L^2}ds
=\int_0^T (Div(hu^*(s)),\tilde{\phi}(s))_{L^2}ds.
\end{align}
Let $t \in [0,T]$ be fixed. Let us choose $\tilde{\phi}(s)=\chi_{(0,t)}(s)w$  in \eqref{esti.z}. Also we note that $\tilde{\phi} \in L^2(0,T;L^2(\mathcal{O})).$ Hence,
\begin{align} \label{esti.div.z}
\lim_{n\rightarrow\infty}\int_0^t (Div(h\ou^n(s)),w)_{L^2}ds
=\int_0^t (Div(hu^*(s)),w)_{L^2}ds.
\end{align}
By \eqref{ineq2}, Minkowskii and H\"older inequalities, and by \eqref{bddh} we have 
\begin{align} \label{esti.vit.z}
&\overline{\E}[|\int_0^t (Div(h\ou^n(s)),w)_{L^2} ds|^2]\leq T\,\|w\|_{L^2}^2\overline{\E}[\int_0^t \|Div(h\ou^n(s))\|_{L^2}^2 ds]\notag\\
&\quad\leq c\,T\,\|w\|_{L^2}^2\overline{\E}[\int_0^t \left(\|h \ Div\ou^n(s)\||_{L^2}^2+\|\nabla h\cdot\ou^n(s)\|_{L^2}^2\right)ds]\notag\\
&\quad\leq c\,T\,\|w\|_{\Hh}^2\overline{\E}[\int_0^t \left(\|h\|_{L^{\infty}}^2 \|\nabla\ou^n(s)\|_{\Ll}^2+\|\nabla h\|_{\mathbb L^{\infty}}^2\|\ou^n(s)\|_{\Ll}^2\right)ds]\notag\\
&\quad\leq c\,T\,\|w\|_{\Hh}^2(\mu^2\overline{\E}\int_0^t\|\ou^n(s)\|_{\h}^2 ds + M^2 T\overline{\E}\sup_{0\leq t\leq T} \|\ou^n(s)\|_{\Ll}^2)\leq C.
\end{align}
Therefore by \eqref{esti.div.z} and \eqref{esti.vit.z} and by Vitali theorem, for all $t\in [0,T],$
\begin{equation} \label{lim.z}
\lim_{n\rightarrow\infty}\overline{\E}[|\int_0^t \Big((Div(h\ou^n(s)),w)_{L^2}-(Div(hu^*(s)),w)_{L^2}\Big) ds|]=0.
\end{equation}
Also repeating the same arguments as for \eqref{esti.vit.z} we achieve,
\begin{align} \label{vit.z}
\int_0^T|\overline{\E}[\int_0^t (Div(h\ou^n(s)),w)_{L^2} ds]|^2dt
\leq C\,T.
\end{align}
Hence by \eqref{lim.z} and \eqref{vit.z} and by the virtue of Vitali theorem, we have
\begin{equation*}
\lim_{n\rightarrow\infty}\int_0^T \overline{\E}[|\int_0^t \Big((Div(h\ou^n(s)),w)_{L^2}-(Div(hu^*(s)),w)_{L^2}\Big) ds|]=0.
\end{equation*}}
\end{itemize}
\end{proof}

\textbf{Step III}:

Define for all $v\in\h(\mathcal{O})$
\gr{\begin{align} \label{K1.n}
K^1(\ou^n,\oz^n,\overline{U}^n,\on^n,\ow^n,v)(t)&=(\ou^n_0,v)_{\Ll}+(\overline{U}^n,v)_{\Ll}-\int_0^t (A\ou^n(s),v)_{\Ll}ds\notag\\
&\quad-\int_0^t (B(\ou^n(s)),v)_{\Ll}ds-\gr{\int_0^t \langle g\nabla \oz^n(s),v \rangle ds}\notag\\
&\quad+\int_0^t (\s^n(s,\ou^n(s))d\ow^n(s),v)_{\Ll}+\gr{\int_0^t(f(s),v)_{\Ll}ds} \notag\\
&\quad+\int_0^t\int_Z (H^n(\ou^n(s-),z),v)_{\Ll}\tilde{\on}^n(ds,dz),
\end{align}}
and for all \gr{$w\in \Hh(\mathcal{O})$}
\begin{align} \label{K2.n}
K^2(\ou^n,\oz^n,w)(t)=(\oz^n_0,w)_{L^2}-\int_0^t (Div(h\ou^n(s)),w)_{L^2} ds.
\end{align}
Hence for all $v\in\h(\mathcal{O})$
\gr{\begin{align} \label{K1.u.star}
K^1(u^*,z^*,U^*,N^*,W^*,v)(t)&=(u^*_0,v)_{\Ll}+(U^*,v)_{\Ll}-\int_0^t (Au^*(s),v)_{\Ll}ds\notag\\
&\quad-\int_0^t (B(u^*(s)),v)_{\Ll}ds-\gr{\int_0^t \langle g\nabla z^*(s),v \rangle ds}\notag\\
&\quad+\int_0^t (\s(s,u^*(s))dW^*(s),v)_{\Ll}+\int_0^t(f(s),v)_{\Ll}ds \notag\\
&\quad+\int_0^t\int_Z (H(u^*(s-),z),v)_{\Ll}\tilde{N}^*(ds,dz),
\end{align}}
and for all \gr{$w\in \Hh(\mathcal{O})$}
\begin{align} \label{K2.u.star}
K^2(u^*,z^*,w)(t)=(z^*_0,w)_{L^2}-\int_0^t (Div(hu^*(s)),w)_{L^2} ds.
\end{align}
\begin{claim}  
\begin{itemize}
\item[1.] For all $v\in\h(\mathcal{O})$
\begin{equation}
\label{pf1}
\lim_{n\rightarrow\infty}\|(\ou^n(\cdot),v)_{\Ll}-(u^*(\cdot),v)_{\Ll}\|_{L^1([0,T]\times \overline{\Omega})}=0,
\end{equation}
and
\begin{equation}
\label{pf2}
\lim_{n\rightarrow\infty}\|K^1(\ou^n,\oz^n,\overline{U}^n,\on^n,\ow^n,v)-K^1(u^*,z^*,U^*,N^*,W^*,v)\|_{L^1([0,T]\times \overline{\Omega})}=0.
\end{equation}
\item[2.]  For all \gr{$w\in \Hh(\mathcal{O})$}
\begin{equation}
\label{pf3}
\lim_{n\rightarrow\infty}\|(\oz^n(\cdot),w)_{L^2}-(z^*(\cdot),w)_{L^2}\|_{L^1([0,T]\times \overline{\Omega})}=0,
\end{equation}
and 
\begin{equation}
\label{pf4}
\lim_{n\rightarrow\infty}\|K^2(\ou^n,\oz^n,w)-K^2(u^*,z^*,w)\|_{L^1([0,T]\times \overline{\Omega})}=0.
\end{equation}
\end{itemize}
\end{claim}

\begin{pf}
\end{pf}
\begin{itemize}
\gr{\item[1.] We note that \eqref{pf1} follows from Lemma \ref{conv.each} (i). Now we see by Fubini's theorem,}
 \begin{align} \label{k1}
&\|K^1(\ou^n,\oz^n,\overline{U}^n,\on^n,\ow^n,v)-K^1(u^*,z^*,U^*,N^*,W^*,v)\|_{L^1([0,T]\times \overline{\Omega})}\notag\\
&=\overline{\E}[\int_0^T|K^1(\ou^n,\oz^n,\overline{U}^n,\on^n,\ow^n,v)-K^1(u^*,z^*,U^*,N^*,W^*,v)| dt]\notag\\
&=\int_0^T \overline{\E}[|K^1(\ou^n,\oz^n,\overline{U}^n,\on^n,\ow^n,v)-K^1(u^*,z^*,U^*,N^*,W^*,v)|] dt.
\end{align}
\gr{Lemma \ref{conv.each} (ii)-(ix) ensure that each term on the right hand side of \eqref{K1.n} converges to the right hand side of corresponding term in \eqref{K1.u.star} in $L^1([0,T]\times \overline{\Omega})$ which further assures that right hand side of \eqref{k1} goes to zero as $n \rightarrow \infty.$ This verifies \eqref{pf2}. 
\item[2.] \gr{We note that \eqref{pf3} follows from Lemma \ref{conv.each} (x).} Now we see by Fubini's theorem,}
\begin{align} \label{k2}
&\|K^2(\ou^n,\oz^n,w)-K^2(u^*,z^*,w)\|_{L^1([0,T]\times \overline{\Omega})}\notag\\
&=\overline{\E}[\int_0^T|K^2(\ou^n,\oz^n,w)-K^2(u^*,z^*,w)| dt]\notag\\
&=\int_0^T \overline{\E}[|K^2(\ou^n,\oz^n,w)-K^2(u^*,z^*,w)|] dt.
\end{align}
\gr{Lemma \ref{conv.each} (xi)-(xii) ensure that each term on the right hand side of \eqref{K2.n} converges to the right hand side of corresponding term in \eqref{K2.u.star} in $L^1([0,T]\times \overline{\Omega})$, which further assures that right hand side of \eqref{k2} goes to zero as $n \rightarrow \infty.$ This verifies \eqref{pf4}.}
Hence our claim is established.
\end{itemize}

\textbf{Step IV}:
Since $u^n$ is a solution of the Galerkin equation, we have for all $t\in [0,T]$
\begin{equation*}
(u^n(t),v)_{\Ll}=K^1(u^n,\z^n, U^n,N^n,W^n,v)(t)\qquad P\text{-a.s.}
\end{equation*}
Hence
\begin{equation} \label{eq.un.kn}
\int_0^T \E[|(u^n(t),v)_{\Ll}-K^1(u^n,\z^n, U^n,N^n,W^n,v)(t)|]dt=0.
\end{equation}
Since $\mathcal{L}((\ou^n,\oz^n,\overline{U}^n,\on^n,\ow^n))=\mathcal{L}((u^n,\z^n,U^n,N^n,W^n))$
\begin{equation*}
\int_0^T \overline{\E}[|(\ou^n(t),v)_{\Ll}-K^1(\ou^n,\oz^n,\overline{U}^n,\on^n,\ow^n,v)(t)|]dt=0.
\end{equation*}
\gr{Using \eqref{pf1}, \eqref{eq.un.kn} and \eqref{pf2}
\begin{IEEEeqnarray*}{lrl}
\int_0^T \overline{\E}[|(u^*(t),v)_{\Ll}-K^1(u^*,z^*,U^*,N^*,W^*,v)(t)|]dt\\
\leq \int_0^T \overline{\E}[|(u^*(t),v)_{\Ll}-(\ou^n(t),v)_{\Ll}|]dt \\
+\int_0^T \overline{\E}[|(\ou^n(t),v)_{\Ll}-K^1(\ou^n,\oz^n,\overline{U}^n,\on^n,\ow^n,v)(t)|]dt \\
+ \int_0^T \overline{\E}[|K^1(\ou^n,\oz^n,\overline{U}^n,\on^n,\ow^n,v)(t)-K^1(u^*,z^*,U^*,N^*,W^*,v)(t)|]dt\\
\rightarrow 0\,\,\mbox{as}\,\, n \rightarrow \infty.
\end{IEEEeqnarray*}}
Hence for \gr{almost all $t\in [0,T]$}
\begin{equation*}
(u^*(t),v)_{\Ll}-K^1(u^*,z^*,U^*,N^*,W^*,v)(t)=0 \qquad \overline{P}\text{-a.s.}
\end{equation*}
Similarly we get for \gr{almost all $t\in [0,T]$}
\begin{equation*}
(z^*(t),w)_{\Ll}-K^2(u^*,z^*,w)(t)=0 \qquad \overline{P}\text{-a.s.}
\end{equation*}
Taking $\ou=u^*,\oz=z^*,\overline{U}=U^*,\on=N^*$and $\ow=W^*$. we see that $(\overline{\Omega},\overline{\mathcal{F}},\overline{F},\overline{P},\overline{u},\overline{z},\overline{U},\overline{N},\overline{W})$ is a martingale solution of \eqref{sc1}-\eqref{sc3}.
\end{proof}
Existence of martingale solution (see Theorem \ref{martexist}) and pathwise uniqueness (see Step V of the proof of Theorem \ref{thmint1}) guarantee uniqueness in law due to the classical Yamada-Watanabe technique \cite{YW} (see also Theorems 2 and 11 in Ondrej\'at \cite{On} for the infinite dimensional version of the result).
\subsection{Existence of Optimal Control}
This Subsection has been developed using certain ideas from Brze{\'z}niak and Serrano \cite{BS}, Sritharan \cite{sritharan2000deterministic}.
Let $(\Omega,\mathcal{F},\mathcal{F}_t,P)$ be a given filtered probability space. The objective of this subsection is to study optimal initial  value control problem \eqref{sc1}-\eqref{sc3} of minimizing a finite-horizon cost functional of the form 
\begin{equation}\label{costgen}
\mJ(u,\z,U)=\E\left[\int_0^T \int_{\mathcal{O}}L(t,u(t,x),\z(t,x),U(x))dx\,dt\right].
\end{equation} 
\begin{definition} \label{defi.cost}
A function $\kappa :\Ll(\mathcal{O}) \rightarrow [0,+\infty] $ is called inf-compact iff for every $R \geq 0$ the level set 
$ \{\kappa \leq R\} $ is compact.
\end{definition}
\begin{assumption}\label{ass.cost}
Let the running cost function $L(\cdot,\cdot,\cdot,\cdot,\cdot):[0,T]\times\h(\mathcal{O})\times L^2(\mathcal{O})\times\Ll(\mathcal{O})\times \Omega\rightarrow\mathbb{R}$ be such that   
\begin{enumerate}
\item[(i)] $L(\cdot,\cdot,\cdot,\cdot, \cdot)$ is measurable,
\item[(ii)] $L(t,\cdot,\cdot,\cdot,\omega):\h(\mathcal{O})_w\times L^2(\mathcal{O})\times\Ll(\mathcal{O})$ is lower semicontinuous $\forall t\in [0,T]$ and $w\in \Omega$, where $\h(\mathcal{O})_w$ is the space $\h(\mathcal{O})$ endowed with the weak topology.
\item[(iii)] There exists an inf-compact function $\kappa :\Ll(\mathcal{O}) \rightarrow [0,+\infty]$ such that for every $(t, u,\z,U)\in [0,T]\times\h(\mathcal{O})\times L^2(\mathcal{O})\times\Ll(\mathcal{O})$, $L(t,u,\z,U)\geq \kappa^2(U)$, $P$-a.s., and $\|\kappa(U)\|_{L^2(\Omega; \Ll)} \rightarrow \infty$ as $\|U\|_{L^2(\Omega; \Ll)}$ $\rightarrow \infty$.
\end{enumerate}
\end{assumption}
A specific example of cost functional is as follows:
\begin{equation}\label{excost}
\mJ(u,\z,U)=\E\left[\int_0^T \left(\|u(s)\|_{\h}^2 + \|\z(s)\|_{L^2}^2\right)ds\right] +\E \|U\|_{\Ll}^2.
\end{equation}

\gr{\begin{definition} \label{defi.opt}
Let $u_0\in L^2(\Omega; \Ll(\mathcal{O}))$ and $\z_0\in L^2(\Omega; L^2(\mathcal{O}))$ and let $T>0$ be fixed. A weak admissible control (with time horizon $[0,T]$) is a system 
\begin{align} \label{form.pi}
\pi=(\Omega,\mathcal{F},\mathcal{F}_t,P,\{W(t)\}_{t \geq 0},\{N(t,\cdot)\}_{t \geq 0},\{u(t)\}_{t \geq 0},\{\z(t)\}_{t \geq 0},U)
\end{align} such that 
\begin{itemize}
\item[1.] $(\Omega,\mathcal{F},F,P)$ is a filtered probability space with a filtration $F=\{\mathcal{F}_t\}_{t\geq 0}$,
\item[2.] $N$ is a time homogeneous Poisson random measure over $(\Omega,\mathcal{F},F,P)$ with the intensity measure $\lambda$,
\item[3.] $W$ is a cylindrical Wiener process over $(\Omega,\mathcal{F},F,P)$,
\item[4.] $U$ is measurable with $P$- a.e. $\omega \in\Omega,$ $U(\omega)\in \Ll(\mathcal{O})$,
\item[5.] \gr{$u, z$ are progressively measurable processes with $P$- a.e. $\omega \in \Omega,$ the paths
\begin{align*}
&u(\cdot, \omega)\in D([0,T];\mathbb{H}^{-1}(\mathcal{O}))\cap D([0,T];\Ll_w(\mathcal{O})) \cap L^2_w(0,T;\h(\mathcal{O})) \cap L^2(0,T;\Ll(\mathcal{O})),\\
&\z(\cdot, \omega)\in L^2(0,T;L^2(\mathcal{O})) \cap C([0,T];H^{-1}(\mathcal{O}))
\end{align*}}
 such that for all $t\in [0, T]$, for all $v\in \h(\mathcal{O})$ and for all $w \in L^2(\mathcal{O})$, the following identities hold $P$-a.s.
\begin{IEEEeqnarray}{llr}
(&u(t),v)_{\Ll}+\int_0^t (Au(s),v)_{\Ll}ds+\int_0^t(B(u(s)),v)_{\Ll}ds+\int_0^t\gr{\langle g\nabla z(s) ,v \rangle} ds&\nonumber\\
&\quad=(u_0,v)_{\Ll}+(U,v)_{\Ll}+\int_0^t(f(s),v)_{\Ll}ds\nonumber\\
&\quad\quad+\int_0^t(\s(s,u(s))d W(s),v)_{\Ll}
+\int_0^t\int_Z (H(u(s-),z),v)_{\Ll} \tilde{N}(ds,dz) ,\\
&(\z(t),w)_{L^2}+\int_0^t (Div(hu(s)),w)_{L^2} ds=(\z_0,w)_{L^2},
\end{IEEEeqnarray} 
\item[6.] the mapping $$[0,T]\times\mathcal{O}\times \Omega \ni (t,x,\omega)\mapsto L(t,u(t,x,\omega),\z(t,x,\omega),U(x,\omega)) \in \mathbb{R}$$ belongs to $L^1([0,T]\times\mathcal{O}\times\Omega;\mathbb{R}).$
\end{itemize}
\end{definition}}
\ada{ In the spirit of Definition \ref{defi.mart}, from 1-4 of Definition \ref{defi.opt}, we note that $\pi$ is a martingale solution of \eqref{sc1}-\eqref{sc3} associated with \eqref{costgen}.}
The set of weak admissible controls (with time horizon $[0,T]$) will be denoted by $\bar{\mathcal{U}}^{w}_{ad}(u_0,\z_0,T).$ 
In this context, under this weak formulation, the cost functional is defined as
\begin{align} \label{form.barJ}
J(\pi):= \mathbb{E}\Big[\int_{0}^{T} \int_{\mathcal{O}}L(t,u,\z,U)dx\,dt \Big],\quad \pi \in \bar{\mathcal{U}}^{w}_{ad}(u_0,\z_0,T)
\end{align}
where $\pi$ has the form \eqref{form.pi} and $(u,\z,U)$ are the components of $\pi.$
\begin{remark} 
The optimal control problem is to minimize $J$ over $\bar{\mathcal{U}}^{w}_{ad}(u_0,\z_0,T)$ for $u_0\in L^2(\Omega; \Ll(\mathcal{O}))$ and $\z_0\in L^2(\Omega; L^2(\mathcal{O}))$ and $T>0$ be fixed. Namely, we seek $\tilde{\pi}\in \bar{\mathcal{U}}^{w}_{ad}(u_0,\z_0,T)$ such that
$$ J(\tilde{\pi})=\inf_{\pi \in \bar{\mathcal{U}}^{w}_{ad}(u_0,\z_0,T)} J(\pi).$$
\end{remark}

We now prove Theorem \ref{main2.1}, the main result in this Subsection which guarantees the existence of weak optimal control. 
\vskip.1in
\noindent
{\textbf{Proof of Theorem \ref{main2.1}:}
\begin{proof}
\textbf{Step I}:\,\,
From the assumptions of the Theorem  \ref{main2.1}, we have $u_0\in L^2(\Omega; \Ll(\mathcal{O}))$ and $\z_0\in L^2(\Omega; L^2(\mathcal{O}))$ and $J(\pi)<+\infty.$  Since $J$ is bounded below by zero on $\bar{\mathcal{U}}^{w}_{ad}(u_0,\z_0,T)$, there exists a minimizing sequence $\{\pi^n\}_{n \geq 1}$ such that $J(\pi^n) \rightarrow \inf_{\pi \in \bar{\mathcal{U}}^{w}_{ad}(u_0,\z_0,T)} J(\pi).$ In other words, $\pi^n=(\Omega^n,\mathcal{F}^n,\mathcal{F}^n_t,P^n, \{W^n(t)\}_{t \geq 0},\{N^n(t,\cdot)\}_{t \geq 0},\{u^n(t)\}_{t \geq 0},\{\z^n(t)\}_{t \geq 0},U^n)$ is a minimizing sequence of weak admissible controls, that is,
\begin{align} \label{J.pi}
\lim_{n \rightarrow \infty} J(\pi^n)= \inf_{\pi \in \bar{\mathcal{U}}^{w}_{ad}(u_0,\z_0,T)} J(\pi).
\end{align}
\gr{Since for each $n \in \mathbb{N},$ $\pi^n \in \bar{\mathcal{U}}^{w}_{ad}(u_0,\z_0,T),$ hence each $\pi^n$ satisfies   
\begin{IEEEeqnarray}{llr}
(&u^n(t),v)_{\Ll}+\int_0^t (Au^n(s),v)_{\Ll}ds+\int_0^t(B(u^n(s)),v)_{\Ll}ds+\int_0^t\gr{\langle g\nabla \z^n(s) ,v \rangle} ds&\nonumber\\
&\ =(u_0,v)_{\Ll}+(U^n,v)_{\Ll}+\int_0^t(f(s),v)_{\Ll}ds\nonumber\\
&\quad+\int_0^t(\s(s,u(s))d W^n(s),v)_{\Ll}
+\int_0^t\int_Z (H(u(s-),z),v)_{\Ll} \tilde{N}^n(ds,dz) ,\\
&(\z^n(t),w)_{L^2}+\int_0^t (Div(hu^n(s)),w)_{L^2} ds=(\z_0,w)_{L^2}.
\end{IEEEeqnarray}   
This holds $P^n$-a.s. for almost all $t\in [0, T]$, for all $v\in \h(\mathcal{O})$ and for all $w \in L^2(\mathcal{O}).$ Proceeding as in similar lines as in Propositions \ref{prop}-\ref{moment}, we have the following a-priori estimates (uniformly in $n$)
\begin{align}
&\sup_{n\geq 1}\E^n[\sup_{0\leq s\leq T}\|u^n(s)\|_{\Ll}^2]\leq C,\quad\sup_{n\geq 1}\E^n[\int_0^T\|u^n(s)\|_{\h}^2 ds]\leq C,\label{en.eq}\\
&\mbox{and}\quad \sup_{n\geq 1}\E^n[\sup_{0\leq s\leq T}\|\z^n(s)\|_{L^2}^2]\leq C,\label{en.eq01}
\end{align}where $\mathbb{E}^n$ denotes the expectation with respect to $P^n.$
Also we note from \eqref{J.pi} that there exists a constant $C>0$ such that $J(\pi^n) \leq C.$ 
Using Assumption \ref{ass.cost} (iii), \eqref{costgen} we have
\begin{align} \label{esti.kap}
\|\kappa(U^n)\|_{L^2(\Omega; \Ll)}^2&=\frac{1}{T}~\mathbb{E}^n\Big[\int_{0}^{T} \int_{\mathcal{O}} \kappa^2(U^n)dx\,dt \Big] \leq \frac{1}{T}~\mathbb{E}^n\Big[\int_{0}^{T} \int_{\mathcal{O}}L(t,u^n,\z^n,U^n)dx\,dt \Big]\notag\\&=\frac{1}{T}~ J(\pi^n) \leq C_T.
\end{align} 
In view of \eqref{esti.kap} and Assumption \ref{ass.cost} (iii) we have 
\begin{align}\label{est.cont01}
\sup\limits_{n\geq 1}\E^n[\|U^n\|_{\Ll}^2]\leq C_T.
\end{align}
Due to the uniform a-priori bounds \eqref{en.eq}, one can establish that $\{\mathcal{L}(u^n),n\in\mathbb{N}\}$ is tight on $(\mathcal{Z},\tau)$ by first proving the Aldous condition (as in Lemma \ref{tight}) and then employing Theorem \ref{sdcthm1}. Similarly, due to uniform bounds \eqref{en.eq01} and \eqref{est.cont01}, we can prove (as in Lemmas \ref{tight2}-\ref{tight3}) that the set of measures $\{\mathcal{L}(\z^n,U^n),n\in\mathbb{N}\}$ is tight on $\Big(L^2(0,T;L^2(\mathcal{O}))\cap C([0,T]; H^{-1}(\mathcal{O}))\Big) \times \Ll(\mathcal{O}).$
\gr{Therefore by the Skorokhod theorem}, there exists a subsequence $(n_k)_{k\in\mathbb{N}}$, a probability space $(\tilde{\Omega},\tilde{\mathcal{F}},\tilde{P})$, and, on this space, random variables $(\tilde{u},\tilde{\z},\tilde{U},\breve{N},\tilde{W}),$ $(\tilde{u}^k,\tilde{\z}^k,\tilde{U}^k, \breve{N}^k,\tilde{W}^k)_{k \in \mathbb{N}}$ such that
\begin{itemize}
\item[(i)] $\mathcal{L}((\tilde{u}^k,\tilde{\z}^k,\tilde{U}^k,\breve{N}^k,\tilde{W}^k))=\mathcal{L}((u^{n_k},\z^{n_k},U^{n_k},N^{n_{k}},W^{n_k}))$ for all $k\in\mathbb{N}$,
\item[(ii)]$(\tilde{u}^k,\tilde{\z}^k,\tilde{U}^k, \breve{N}^k,\tilde{W}^k) \rightarrow (\tilde{u},\tilde{\z},\tilde{U},\breve{N},\tilde{W})$ in \gr{$\mathcal{Z} \times \Big(L^2(0,T;L^2(\mathcal{O}))\cap C([0,T]; H^{-1}(\mathcal{O}))\Big) \times \Ll(\mathcal{O}) \times M_{\bar{\mathbb{N}}}([0,T]\times Z) \times C([0,T];\mathbb{R})$} with probability 1 on $(\tilde{\Omega},\tilde{\mathcal{F}},\tilde{P})$ as $k\rightarrow\infty$,
\item[(iii)] $(\breve{N}^k(\tilde{\omega}),\tilde{W}^k(\tilde{\omega})) = (\breve{N}(\tilde{\omega}),\tilde{W}(\tilde{\omega}))$ for all $\tilde{\omega}\in \tilde{\Omega}$.
\end{itemize}}

\noindent\ada{For convenience, we denote these sequences again by $((u^n,\z^n,U^n,N^n,W^n))_{n\in\mathbb{N}}$ and 
$((\tilde{u}^n,\tilde{\z}^n,\tilde{U}^n, \\ \breve{N}^n,\tilde{W}^n))_{n\in\mathbb{N}}$. Now, following the similar steps as in Theorem \ref{martexist}}, one can establish that $$\tilde{\pi}=(\tilde{\Omega},\tilde{\mathcal{F}},\tilde{\mathcal{F}_t},\tilde{P},\{\tilde{W}(t)\}_{t \geq 0},\{\breve{N}(t,\cdot)\}_{t \geq 0},\{\tilde{u}(t)\}_{t \geq 0},\{\tilde{\z}(t)\}_{t \geq 0},\tilde{U})$$ is a martingale solution of \eqref{sc1}-\eqref{sc3}  associated to \eqref{costgen}. 
\vskip.1in\noindent
\textbf{Step II}:\\ We now prove that the cost functional is lower semicontinuous. Below we extend the proof of Lemma 10 and Lemma 11 of Sritharan \cite{sritharan2000deterministic} to the stochastic case.
\begin{lemma}
\label{lowersemcont}
For $\tilde{u}^n\rightarrow \tilde{u}$ in the $L^2(\tilde{\Omega};\gr{\mathcal{Z}}),$ $\tilde{\z}^n\rightarrow \tilde{\z}$ in $L^2(\tilde{\Omega};L^2(0,T;\gr{L^2(\mathcal{O}))}$ and $\tilde{U}^n\rightarrow \tilde{U}$ in $L^2(\tilde{\Omega};\gr{\Ll(\mathcal{O}))}$, we have 
\begin{equation}
\liminf\limits_{n\rightarrow\infty}\tE \left[\int_0^T\int_{\mathcal{O}}L(t,\tilde{u}^n,\tilde{\z}^n,\tilde{U}^n)dx\,dt\right]\geq\tE\left[\int_0^T\int_{\mathcal{O}}L(t,\tilde{u},\tilde{\z},\tilde{U})dx\,dt\right].
\end{equation}
\end{lemma}
\begin{proof}
For each natural number $M,$ let $L_M(t,\tilde{u},\tilde{\z},\tilde{U})=L(t,\tilde{u},\tilde{\z},\tilde{U})\wedge M$. Then by linearity of expectation, we have
\begin{align*}
&\liminf\limits_{n\rightarrow\infty}\tE \left[\int_0^T\int_{\mathcal{O}}L(t,\tilde{u}^n,\tilde{\z}^n,\tilde{U}^n)dx\,dt\right]\\
&\quad\geq\liminf\limits_{n\rightarrow\infty}\tE \left[\int_0^T\int_{\mathcal{O}}L_M(t,\tilde{u}^n,\tilde{\z}^n,\tilde{U}^n)dx\,dt\right]\\
&\quad\geq -\limsup\limits_{n\rightarrow\infty}\tE \left[\int_0^T\int_{\mathcal{O}}\left(L_M(t,\tilde{u}^n,\tilde{\z}^n,\tilde{U}^n)-L_M(t,\tilde{u},\tilde{\z}^n,\tilde{U}^n)\right)^- dx\,dt\right]\\
&\quad\quad-\limsup\limits_{n\rightarrow\infty}\tE \left[\int_0^T\int_{\mathcal{O}}\left(L_M(t,\tilde{u},\tilde{\z}^n,\tilde{U}^n)-L_M(t,\tilde{u},\tilde{\z},\tilde{U}^n)\right)^- dx\,dt\right]\\
&\quad\quad+\liminf\limits_{n\rightarrow\infty}\tE \left[\int_0^T\int_{\mathcal{O}}L_M(t,\tilde{u},\tilde{\z},\tilde{U}^n)dx\,dt\right],
\end{align*}
where $f(x)^-=(-f(x))\vee 0$. The first and second terms on the right-hand side are zero due to Lemma \ref{lowersemi2} given below. Due to the lower semicontinuity of $L_M,$ from Jacod and M{\'e}min \cite{jacod}, Lemma 4 of Sritharan \cite{sritharan2000deterministic}, Proposition 2.1.12 of Castaing et al. \cite{CFV}  we have
\begin{equation*}
\liminf\limits_{n\rightarrow\infty} \left[\int_0^T\int_{\mathcal{O}}L_M(t,\tilde{u},\tilde{\z},\tilde{U}^n)dx\,dt\right]\geq \left[\int_0^T\int_{\mathcal{O}}L_M(t,\tilde{u},\tilde{\z},\tilde{U})dx\,dt\right].
\end{equation*}
 Therefore application of Fatou's Lemma gives
 \begin{equation*}
\liminf\limits_{n\rightarrow\infty} \tE \left[\int_0^T\int_{\mathcal{O}}L_M(t,\tilde{u},\tilde{\z},\tilde{U}^n)dx\,dt\right]\geq \tE \left[\int_0^T\int_{\mathcal{O}}L_M(t,\tilde{u},\tilde{\z},\tilde{U})dx\,dt\right].
\end{equation*}
 Hence using the Beppo-Levi theorem on the bounded measurable functions $L_M$ we have
\begin{align*}
\liminf\limits_{n\rightarrow\infty}\tE \left[\int_0^T\int_{\mathcal{O}}L(t,\tilde{u}^n,\tilde{\z}^n,\tilde{U}^n)dx\,dt\right]\geq \tE\left[\int_0^T\int_{\mathcal{O}}L_M(t,\tilde{u},\tilde{\z},\tilde{U})dx\,dt\right]\\
\geq \tE \left[\int_0^T\int_{\mathcal{O}}L(t,\tilde{u},\tilde{\z},\tilde{U})dx\,dt\right].
\end{align*}
\end{proof}

\begin{lemma} \label{lowersemi2}
Let $\tilde{u}^n\rightarrow \tilde{u}$ in the $L^2(\tilde{\Omega};\gr{\mathcal{Z})},$ $\tilde{\z}^n\rightarrow \tilde{\z}$ in $L^2(\tilde{\Omega};\gr{L^2}(0,T;L^2(\mathcal{O}))$ and $\tilde{U}^n\rightarrow \tilde{U}$ in $L^2(\tilde{\Omega};\gr{\Ll(\mathcal{O})})$. For $\tilde{\omega}\in\tilde{\Omega}$, let $\varphi(\cdot,\cdot,\cdot,\cdot,\tilde{\omega}):[0,T]\times\h(\mathcal{O})\times L^2(\mathcal{O})\times\Ll(\mathcal{O})\rightarrow \mathbb{R}_+$ be a bounded measurable function such that $\forall\, t\in (0,T)$ and $\tilde{\omega}\in\tilde{\Omega}$, $\varphi(t,\cdot,\cdot,\cdot,\tilde{\omega}):\h(\mathcal{O})\times L^2(\mathcal{O})\times\Ll(\mathcal{O})\rightarrow \mathbb{R}_+$ is lower semicontinuous. Then
\begin{equation}
\label{cost1}
\lim\limits_{n\rightarrow\infty}\tE\left[\int_0^T\int_{\mathcal{O}}\left(\varphi(t,\tilde{u}^n,\tilde{\z}^n,\tilde{U}^n)-\varphi(t,\tilde{u},\tilde{\z}^n,\tilde{U}^n)\right)^- dx\,dt\right]=0,
\end{equation}
and
\begin{equation}
\label{cost2}
\lim\limits_{n\rightarrow\infty}\tE\left[\int_0^T\int_{\mathcal{O}}\left(\varphi(t,\tilde{u},\tilde{\z}^n,\tilde{U}^n)-\varphi(t,\tilde{u},\tilde{\z},\tilde{U}^n)\right)^- dx\,dt\right]=0.
\end{equation}
\end{lemma}

\begin{proof}
Define, $\Theta(t,\tilde{v},\tilde{\z},\tilde{U},\tilde{\omega}):=\varphi(t,\tilde{v},\tilde{\z},\tilde{U},\tilde{\omega})-\varphi(t,\tilde{u},\tilde{\z},\tilde{U},\tilde{\omega})$. For $\delta>0$ and $y\in\mathbb{H}^{-1}(\mathcal{O})$, define
\begin{align*}
\begin{split}
Y^m:&=\left\{(t,x,\tilde{\omega})\in[0,T] \times \mathcal{O} \times \tilde{\Omega}:\right.\\
&\qquad\qquad\qquad\left. \inf\limits_{|\langle y,\tilde{u}(t,x)\rangle-\langle y,\tilde{v}(t,x)\rangle|\leq 1/m}\Theta(t,\tilde{v}(t,x),\tilde{\z}(t,x),\tilde{U}(x),\tilde{\omega})\leq -\delta\right\}.
\end{split}
\end{align*}
Note that $Y^{m+1}\subseteq Y^m$, and lower semicontinuity of $\varphi$ implies that each $t$-section of $Y^m$ is closed. Moreover,  $\varphi(t,\cdot,\tilde{\z},\tilde{U},\tilde{\omega}):\h(\mathcal{O})\to\mathbb{R}_+$ being lower semicontinuous implies that if for any sequence $\tilde{v}^n\to \tilde{u}$ in $\h(\mathcal{O})$, we have 
\begin{equation*}
\liminf\limits_{n\rightarrow\infty} \Theta(t,\tilde{v}^n,\tilde{\z},\tilde{U},\tilde{\omega}) \geq 0.
\end{equation*}
Hence
\begin{equation*}
\cap_m Y^m=\emptyset.
\end{equation*}
Now we define
\begin{align*}
\hat{Y}^n:&=\left\{(t,x,\tilde{\omega}):\Theta(t,\tilde{u}^n(t,x),\tilde{\z}^n(t,x),\tilde{U}^n(x),\tilde{\omega})^->\delta\right\}\\
&=\left\{(t,x,\tilde{\omega}):\Theta(t,\tilde{u}^n(t,x),\tilde{\z}^n(t,x),\tilde{U}^n(x),\tilde{\omega})<-\delta\right\}.
\end{align*}
Then for sufficiently large $n$, $\tilde{P}(\hat{Y}^n)\leq \tilde{P}(Y^m)$ (as $\hat{Y}^n\subseteq Y^m$ for $n>m$).
Furthermore,
\begin{align} \label{lim}
\limsup\limits_{n \rightarrow \infty}\tilde{P}(\hat{Y}^n)=\lim_{m \rightarrow \infty}\limsup\limits_{n \rightarrow \infty}\tilde{P}(\hat{Y}^n)\leq \lim_{m \rightarrow \infty} \tilde{P}(Y^m)=\tilde{P}\left(\cap_m Y^m\right)=0.
\end{align}
We now proceed as in Jacod and M{\'e}min \cite{jacod} to obtain \eqref{cost1}.
Since $\varphi$'s are bounded measurable, so $|\Theta(t,\tilde{u}^n,\tilde{\z}^n,\tilde{U}^n,\tilde{\omega})| \leq C\,\, \,\forall \,\,n,$ and using \eqref{lim} we have
\begin{align*}
&\int_{\tilde{\Omega}}\int_0^T\int_\mathcal{O}\Theta(t,\tilde{u}^n(t,x),\tilde{\z}^n(t,x),\tilde{U}^n(x),\tilde{\omega})^{-}\;dx\,dt\,d \tilde{P}(\tilde{\omega})\\
&= \int_{{\hat{Y'}^n}}\Theta(t,\tilde{u}^n(t,x),\tilde{\z}^n(t,x),\tilde{U}^n(x),\tilde{\omega})^{-}\;dx\,dt\,d \tilde{P}(\tilde{\omega})\\
& \quad+\int_{{\hat{Y}^n}}\Theta(t,\tilde{u}^n(t,x),\tilde{\z}^n(t,x),\tilde{U}^n(x),\tilde{\omega})^{-}\;dx\,dt\,d \tilde{P}(\tilde{\omega})\\
& < \delta \tilde{P}({\hat{Y'}^n})+ C \tilde{P}({\hat{Y}^n}) \rightarrow 0, \quad \mbox{as} \quad n \rightarrow \infty
\end{align*} and as $\delta>0$ is arbitrarily small where $\hat{Y'}^n$ denotes the complement of $\hat{Y}^n.$
Therefore $$\lim\limits_{n\rightarrow\infty}\tE\left[\mathlarger{\int_0^T\int_\mathcal{O}}\Theta(t,\tilde{u}^n,\tilde{\z}^n,\tilde{U}^n)^{-}\;dx\,dt\right]=0,$$
which proves \eqref{cost1}.
By similar approach one can prove \eqref{cost2}.
\end{proof} 
\noindent
\textbf{Step III}:\\
Using Lemma \ref{lowersemcont} and using  $\mathcal{L}(\tilde{u}^n,\tilde{\z}^n,\tilde{U}^n)=\mathcal{L}(u^n,\z^n,U^n),$ we have
\begin{align} \label{Fat.F}
\tE\Big[\int_0^T \int_{\mathcal{O}} L(t,\tilde{u}(t),\tilde{\z}(t),\tilde{U})dx\,dt\Big] \leq &\liminf_{n \rightarrow \infty} \tE \Big[\int_0^T \int_{\mathcal{O}} L(t,\tilde{u}^n(t),\tilde{\z}^n(t),\tilde{U}^n)dx\,dt\Big] \notag\\
= &\liminf_{n \rightarrow \infty} \E^n \Big[\int_0^T \int_{\mathcal{O}} L(t,u^n(t),\z^n(t),U^n)dx\,dt\Big]
\end{align} 
which is finite because of \eqref{J.pi} and $J(\pi)<+\infty,$ where $\mathbb{E}^n$ denotes the expectation with respect to $P^n.$
Hence, $J(\tilde{\pi})<+\infty,$ or in other words, the mapping $$[0,T]\times\mathcal{O}\times \tilde{\Omega}\ni (t,x,\tilde{\omega})\mapsto L(t,\tilde{u}(t,x,\tilde{\omega}),\tilde{\z}(t,x,\tilde{\omega}),\tilde{U}(x,\tilde{\omega})) \in \mathbb{R}$$ belongs to $L^1([0,T]\times\mathcal{O}\times\tilde{\Omega};\mathbb{R}).$
Therefore, $\tilde{\pi}$ satisfies all the conditions of Definition \ref{defi.opt}, and hence is a weak admissible control.

Now employing \eqref{J.pi}, \eqref{Fat.F} and $\mathcal{L}(\tilde{u}^n,\tilde{\z}^n,\tilde{U}^n)=\mathcal{L}(u^n,\z^n,U^n),$ it follows that 
\begin{align*}
J(\tilde{\pi})&=\tE\Big[\int_0^T \int_{\mathcal{O}} L(t,\tilde{u}(t),\tilde{\z}(t),\tilde{U})dx\,dt\Big] \\
&\leq \liminf_{n \rightarrow \infty} \Big[\mathbb{E}^n \int_0^T \int_{\mathcal{O}} L(t,u^n(t),\z^n(t),U^n)dx\,dt \Big]\\
&=\inf_{\pi \in \bar{\mathcal{U}}^{w}_{ad}(u_0,\z_0,T)} J(\pi).
\end{align*} 
This proves $\tilde{\pi}$ is a weak optimal control for the control problem, hence this completes the proof.
\end{proof}


\appendix
\section{}
Here we will provide a proof of $\h$ energy estimate of \eqref{eqn5}-\eqref{eqn6}. For this we need to impose more regular conditions on the noise coefficients $\s$ and $H$. We assume that $\s$ and $H$ satisfy the following hypotheses:
\begin{enumerate}
	\item[A.1] $\s\in C([0,T]\times \h(\mathcal{O});L_Q(\mathbb{L}^2,\h)), H\in\mathbb{H}^2_\lambda([0,T]\times Z;\h(\mathcal{O}))$,
	\item[A.2] For all $t\in(0,T)$, there exists a positive constant $K$ such that for all $u\in\h(\mathcal{O})$ 
	\begin{equation*}
	\|\s(t,u)\|_{L_Q(\mathbb{L}^2,\h)}^2+\int_Z\| H(u,z)\|_{\h}^2\la\leq K(1+\|u\|_{\h}^2),
	\end{equation*}
	\item[A.3] For all $t\in(0,T)$, there exists a positive constant $L$ such that for all $u,v\in\h(\mathcal{O})$  
	\begin{equation*}
	\|\s(t,u)-\s(t,v)\|_{L_Q(\mathbb{L}^2,\h)}^2+\int_Z \| H(u,z)-H(v,z)\|_{\h}^2\la 
	\leq L\|u-v\|_{\h}^2.
	\end{equation*}
\end{enumerate}
We also make two additional assumptions, due to technical reasons, as following:
\begin{enumerate}
	\item[A.4] The depth function $h\in C^2_b(\mathcal{O})$. 
	\item[A.5] $curl\ curl u =0$ (i.e., rotation of the fluid velocity is uniform over $\mathcal{O}$).\\
	Note that, in two dimensions $curl u$ is a scalar that should be understood as follows
	$$curl u := \nabla\times u = \nabla\times (u_1e_1+u_2e_2+ 0e_3)= 0e_1+0e_2+(\dfrac{\partial u_2}{\partial x}-\dfrac{\partial u_1}{\partial y})e_3,$$ where $(e_1, e_2, e_3)$ is the canonical basis of $\mathbb{R}^3$. 
\end{enumerate}	
\begin{proposition}\label{PropH1}
	Let
	\begin{equation}
	\left\{\begin{array}{l}
	w^0\in L^4(\Omega;L^4(0,T;\h(\mathcal{O}))),\;f\in L^2(\Omega;L^2(0,T;\h(\mathcal{O}))),\\
	u_0\in L^2(\Omega;\h(\mathcal{O})),\z_0\in L^2(\Omega;H^1_0(\mathcal{O}))
	\end{array}\right.
	\end{equation}
and the assumptions A.1 - A.5 hold.
If $(u^n(t),\z^n(t))$ denotes the unique strong solution of the system \eqref{eqn5}-\eqref{eqn6}, then the following a-priori estimates hold:
	\begin{align}
	&\E\left[\|u^n(t)\|_{\h}^2+\|\z^n(t)\|_{\h}^2\right]+ \frac{3\alpha}{4}\E\int_0^{t}\|\triangle u^n(s)\|^2_{\Ll} ds \leq C_{11}, \;\;\forall t\in[0,T],\label{h1est}\\
	&\E\left[\sup_{0\leq t\leq T}(\|u^n(t)\|_{\h}^2+\|\z^n(t)\|_{\h}^2)\right]+ \frac{3\alpha}{4}\E\int_0^{T}\|\triangle u^n(t)\|^2_{\Ll} dt \leq C_{12},\label{h1est1}
	\end{align}
	where the constants $C_{11}$ and $C_{12}$ depend on the coefficients $\alpha, g,M,\mu$ and the norms $\|f\|_{ L^2(\Omega; L^2(0,T;\h))},$\\$\|w^0\|_{L^4(\Omega; L^4(0,T;\h))},\gr{\|u_0\|_{L^2(\Omega;\h)}}$, $\gr{\|\z_0\|_{L^2(\Omega;\h)}}$ and $T$.
\end{proposition}
\begin{proof}
Define
	\begin{IEEEeqnarray*}{lrl}
		\tau_N :=\inf\{t \geq 0:\|u^n(t)\|_{\h}^2+\|\z^n(t)\|_{\h}^2+\dfrac{3\alpha}{4}\int_0^t\|\triangle u^n(s)\|^2_{\Ll}ds > N\}
	\end{IEEEeqnarray*}
	as the stopping time.
	Applying It\^o's lemma to the \ada{process $\|\nabla u^n(t)\|_{\Ll}^2$ }
	\begin{align*}
	&\|u^n(\L)\|_{\h}^2+ 2\alpha \int_0^{\L}\|\triangle u^n(s)\|^2_{\Ll}ds+2\int_0^{\L}(\nabla B(u^n(s)),\nabla u^n(s))_{\Ll}ds\\&\quad
		+2\int_0^{\L}\gr{\langle \nabla (g\nabla\z^n(s)) ,\nabla u^n(s) \rangle} ds
		=2\int_0^{\L}(\nabla f,\nabla u^n(s))_{\Ll}ds\\
		&\quad+2\int_0^{\L}(\nabla \s^n(s,u^n(s))dW^n(s),\nabla u^n(s))_{\Ll} \\&\quad
	  +\int_0^{\L}\|\nabla\s^n(s,u^n(s))\|_{L_Q(\mathbb{L}^2,\Ll)}^2ds \\
		&\quad+\int_0^{\L}\int_Z\|\nabla H^n(u^n(s-),z)\|_{\Ll}^2 N(ds,dz) + \|u^n(0)\|_{\h}^2\\
		&\quad+2\int_0^{\L}\int_Z (\nabla H^n(u^n(s-),z),\nabla u^n(s-))_{\Ll}\ns .\IEEEyesnumber
	\end{align*}
	Using divergence theorem we obtain
	\begin{align}\label{h11}
	 &\|u^n(\L)\|_{\h}^2+ 2\alpha\int_0^{\L} \|\triangle u^n(s)\|^2_{\Ll} ds\nonumber\\
		&=\|u^n(0)\|_{\h}^2+2\int_0^{\L}( B(u^n(s)),\triangle u^n(s))_{\Ll}ds\nonumber \\& \quad+2\int_0^{\L}(g\nabla \z^n(s) ,\triangle u^n(s))_{\Ll}     ds +2\int_0^{\L}(\nabla f,\nabla u^n(s))_{\Ll}ds\nonumber \\& \quad
		+2\int_0^{\L}(\nabla \s^n(s,u^n(s))dW^n(s),\nabla u^n(s))_{\Ll}\nonumber\\
		 &\quad+\int_0^{\L}\|\nabla\s^n(s,u^n(s))\|_{L_Q(\mathbb{L}^2,\Ll)}^2ds\nonumber \\& \quad
		+\int_0^{\L}\int_Z\|\nabla H^n(u^n(s-),z)\|_{\Ll}^2 N(ds,dz)\nonumber \\
		&\quad+2\int_0^{\L}\int_Z (\nabla H^n(u^n(s-),z),\nabla u^n(s-))_{\Ll}\ns .
	\end{align}
	Let $\{e_j\}_{j=1}^{\infty}$ be an orthonormal basis in
	$\mathbb L^2(\mathcal{O})$ such that $Qe_j=\lambda_je_j$ for all
	$j=1,2,\cdots$, where $\lambda_j$ are the eigenvalues of $Q$. Then by \eqref{hsnorm}, we observe that
	\begin{align}\label{sr3}
&\|\nabla\s^n(t,u^n(t))\|_{L_Q(\mathbb{L}^2,\Ll)}^2=\sum_{j=1}^{\infty}\lambda_j\|\nabla\s^n(t,u^n(t))e_j\|_{\Ll}^2 \nonumber\\&= \sum_{j=1}^{\infty}\lambda_j\|\s^n(t,u^n(t))e_j\|_{\h}^2
=\|\s^n(t,u^n(t))\|_{L_Q(\mathbb{L}^2,\h)}^2.
	\end{align}
		Using equation \eqref{e2}, Young's and Minkowskii inequalities, we obtain
		\begin{IEEEeqnarray*}{lrl}
			2(B(u^n),\triangle u^n)_{\Ll}&\leq &\dfrac{2}{\alpha}\|B(u^n)\|_{\Ll}^2+\frac{\alpha}{2}\|\triangle u^n(t)\|_{\Ll}^2\\
			&\leq & \dfrac{16}{\alpha}(\|u^n(t)\|^4_{\mathbb{L}^4}+\|w^0(t)\|_{\mathbb{L}^4}^4)+\frac{\alpha}{2}\|\triangle u^n(t)\|_{\Ll}^2.
		\end{IEEEeqnarray*}
	Using Young's inequality
	\begin{equation}
	2|g(\nabla\z^n(t) ,\triangle u^n(t))_{\Ll}|\leq \dfrac{2g^2}{\alpha}\|\z^n(t)\|_{\h}^2+\dfrac{\alpha}{2}\|\triangle u^n(t)\|_{\Ll}^2,
	\end{equation}
	and
	\begin{equation}
	2(\nabla f,\nabla u^n(t))_{\Ll}\leq\|f(t)\|_{\h}^2+\|u^n(t)\|_{\h}^2.
	\end{equation}
Hence using the above estimates in \eqref{h11}
	\begin{align}\label{h12}
		&\|u^n(\L)\|_{\h}^2+ \alpha\int_0^{\L}\|\triangle u^n(s)\|^2_{\Ll} ds\notag\\
		&\leq \int_0^{\L} \Big(\|u^n(s)\|_{\h}^2+\dfrac{16}{\alpha}\|u^n(s)\|^4_{\mathbb{L}^4}+\dfrac{2g^2}{\alpha}\|\z^n(s)\|_{\h}^2+\|f(s)\|_{\h}^2\notag\\&\quad\quad+\dfrac{16}{\alpha}\|w^0(s)\|_{\mathbb{L}^4}^4\Big)ds
		+\int_0^{\L}\|\s^n(s,u^n(s))\|_{L_Q(\mathbb{L}^2,\h)}^2ds\notag\\
			&\quad+\int_0^{\L}\int_Z\|H^n(u^n(s-),z)\|_{\h}^2 N(ds,dz)\notag\\
		&\quad+2\int_0^{\L}(\nabla \s^n(s,u^n(s))dW^n(s),\nabla u^n(s))_{\Ll}\notag\\
		&\quad+
		2\int_0^{\L}\int_Z (\nabla H^n(u^n(s-),z),\nabla u^n(s-))_{\Ll}\ns	+\|u(0)\|_{\h}^2. 
	\end{align} 
Now assuming an additional boundary condition that the gradient of the displacement of the free surface with respect to the ocean bottom is zero on the boundary of the domain (i.e. $\nabla\z=0$ on $[0, T]\times\partial\mathcal{O}$), we have from \eqref{eqn6}
	\begin{align}
	\|\nabla\z^n(\L)\|_{\Ll}^2&=-2\int_0^{\L}(\nabla(Div(hu^n(s))),\nabla\z^n(s))_{\Ll}ds+\|\nabla\z^n(0)\|_{\Ll}^2\nonumber\\
	&=-2\int_0^{\L}(\nabla(h Div u^n(s))+\nabla(\nabla h\cdot u^n(s)),\nabla\z^n(s))_{\Ll}ds\nonumber\\&\quad +\|\nabla\z^n(0)\|_{\Ll}^2
	:=\int_0^{\L}(I_1 + I_2) ds+\|\nabla\z^n(0)\|_{\Ll}^2.\label{zestimate}
	\end{align}
Now by use of product rule of gradient, assumption A.5, Minkowskii, H\"older's and Young's inequalities, we have
\begin{align*}
|I_1|&\leq 2\|Div u^n(t) \nabla h + h\nabla(Div u^n(t))\|_{\Ll} \|\nabla\z^n(t)\|_{\Ll}\\
&\leq 2(\|\nabla h\|_{\mathbb{L}^{\infty}}\|Div u^n(t)\|_{L^2}+ \| h\|_{L^{\infty}} \|\nabla(Div u^n(t))\|_{\Ll}) \|\nabla\z^n(t)\|_{\Ll}\\
&\leq 2(M\|u^n(t)\|_{\h}+\mu  \|\nabla(Div u^n(t))\|_{\Ll}) \|\nabla\z^n(t)\|_{\Ll}\\
&\leq 2(M\|u^n(t)\|_{\h}+\mu \|\triangle u^n(t) + curl\ curl u^n(t)\|_{\Ll})\|\nabla\z^n(t)\|_{\Ll}\\
&= 2(M\|u^n(t)\|_{\h}+\mu \|\triangle u^n(t)\|_{\Ll})\|\nabla\z^n(t)\|_{\Ll}\\
&\leq M (\|u^n(t)\|_{\h}^2 + \|\z^n(t)\|_{\h}^2)+\dfrac{\alpha}{4} \|\triangle u^n(t)\|_{\Ll}^2 + \dfrac{2\mu^2}{\alpha} \|\z^n(t)\|_{\h}^2\\
&\leq (M+\dfrac{2\mu^2}{\alpha}) (\|u^n(t)\|_{\h}^2 + \|\z^n(t)\|_{\h}^2) + \dfrac{\alpha}{4} \|\triangle u^n(t)\|_{\Ll}^2.
\end{align*}
By use of classical tame estimates for the product in Sobolev spaces (see, for instance, Corollary 2.4.1 of \cite{Chemin}), we have
\begin{align*}
|I_2|&\leq 2\|\nabla(\nabla h\cdot u^n(t))\|_{\Ll} \|\nabla\z^n(t)\|_{\Ll}\\
&\leq 2c (\|\nabla h\|_{\mathbb{L}^{\infty}}\|\nabla u^n(t)\|_{\Ll} + \|\nabla(\nabla h)\|_{\mathbb{L}^{\infty}}\|u^n(t)\|_{\Ll}) \|\nabla\z^n(t)\|_{\Ll}\\
&\leq c(M+Nc_P) (\|u^n(t)\|_{\h}^2 + \|\z^n(t)\|_{\h}^2),
\end{align*}
where $N:=\max_{x\in\mathcal{O}} |\nabla(\nabla h(x))|$ (as $h\in C^2_b(\mathcal{O})$ by assumption A.4), and $c_P$ is the Poincar\'e constant.

Hence from the estimates of $I_1$ and $I_2$, we have from \eqref{zestimate}, 
\begin{align}\label{eeeeq1}
\|\z^n(\L)\|_{\h}^2 &\leq \|\z(0)\|_{\h}^2+C\int_0^{\L}\Big(\|u^n(s)\|_{\h}^2 + \|\z^n(s)\|_{\h}^2\Big)ds \nonumber\\
&\quad+ \dfrac{\alpha}{4}\int_0^{\L} \|\triangle u^n(s)\|_{\Ll}^2ds  , 
\end{align}
where $C:=M+\dfrac{2\mu^2}{\alpha}+c(M+Nc_P)$.

Adding equations \eqref{h12} and \eqref{eeeeq1}, we obtain
	
\begin{align} \label{h1energy1} 
 &\|u^n(t\wedge \tau_N)\|_{\h}^2+\|\z^n(t\wedge\tau_N)\|_{\h}^2+ \frac{3\alpha}{4}\int_0^{t\wedge\tau_N}\|\triangle u^n(s)\|^2_{\Ll} ds\notag\\
		&\leq K_1\left[\int_0^{t\wedge\tau_N}(\|u^n(s)\|_{\h}^2+\|\z^n(s)\|_{\h}^2)ds\right] +\dfrac{16}{\alpha}\int_0^{t\wedge\tau_N}\|u^n(s)\|^4_{\mathbb{L}^4}ds\notag\\
		&\quad+\int_0^{t\wedge\tau_N}\left(\|f(s)\|_{\h}^2+\dfrac{16}{\alpha}\|w^0(s)\|_{\mathbb{L}^4}^4\right)ds+\|u^n_0\|_{\h}^2+\|\z^n_0\|_{\h}^2\notag\\\ &\quad+K(t\wedge\tau_N),
	\end{align}	
	where $K_1:=\max\{1+C, 2g^2/\alpha + C\}$.\\
	Following the same steps as in Proposition \ref{prop}, we integrate over $0\leq t\leq T\wedge\tau_N$ and take expectation on both sides and employ the facts that $(i)$ stochastic integrals appearing in our calculations are martingales with zero average, $(ii)$ expectation of quadratic variation process and that of Meyer process are same and $(iii)$ assumption A.2, to obtain 
	\begin{IEEEeqnarray*}{lr}
		\E\left[\|u^n(t\wedge \tau_N)\|_{\h}^2+\|\z^n(t\wedge\tau_N)\|_{\h}^2\right]+ \frac{3\alpha}{4}\E\int_0^{t\wedge\tau_N}\|\triangle u^n(s)\|^2_{\Ll} ds\\
		\leq K_1\E\left[\int_0^{t\wedge\tau_N}(\|u^n(s)\|_{\h}^2+\|\z^n(s)\|_{\h}^2)ds\right] +\dfrac{16}{\alpha}\E\int_0^{t\wedge\tau_N}\|u^n(s)\|^4_{\mathbb{L}^4}ds\\
		\quad+\E\left[\int_0^{t\wedge\tau_N}\left(\|f(s)\|_{\h}^2+\dfrac{16}{\alpha}\|w^0(s)\|_{\mathbb{L}^4}^4\right)ds\right]+\E\left[\|u^n_0\|_{\h}^2+\|\z^n_0\|_{\h}^2\right]\\ \quad+K(t\wedge\tau_N),\IEEEyesnumber
	\end{IEEEeqnarray*}
 As before we observe that since $w^0\in L^8(\Omega; L^8(0,T;\mathbb{L}^{8}(\mathcal{O}))) \subset  L^{4}(\Omega; L^{4}(0,T;\\\mathbb{L}^4(\mathcal{O})))$, by Proposition \ref{moment}, for $p=4$, we have $\mathbb{E}[\int_0^T \|u^n(s)\|_{\mathbb{L}^2}^2  \|u^n(s)\|_{\h}^2ds] \leq C_{1(4)}$, which essentially yields $\mathbb{E}[\int_0^T \|u^n(s)\|_{\mathbb{L}^4}^4 ds] \leq C_{1(4)}$ (by the estimate \eqref{eq13}).\\
Hence by applying Gronwall's lemma and taking limit as $N\to\infty$ we have the desired a-priori estimate \eqref{h1est}. 

To prove \eqref{h1est1} one can proceed similarly but needs to take supremum over $[0, T\wedge\tau_N]$ over \eqref{h1energy1} before taking expectation and then suitably apply Burkholder-Davis-Gundy inequality.

\end{proof}

\noindent
{\bf Acknowledgements:} Pooja Agarwal would like to thank Department of Science and Technology, Govt. of India, for the INSPIRE fellowship. Utpal Manna's work has been supported by the National Board of Higher Mathematics of Department of Atomic Energy, Govt. of India, under Grant No. NBHM/RP46/2013/Fresh/421. The authors would also like to thank Sakthivel Kumarasamy of Indian Institute of Space Science and Technology for his comments and pointing our attention to certain references. Finally, the authors would like to sincerely thank the anonymous referee for his/her valuable comments and criticisms which led to the improvement of this paper.


\end{document}